\documentclass[oneside,english]{amsart}
\usepackage[T1]{fontenc}
\usepackage[latin9]{inputenc}
\usepackage{mathtools}
\usepackage{enumitem}
\usepackage{amstext}
\usepackage{amsthm}
\usepackage{amssymb}

\usepackage{xcolor}
\usepackage[normalem]{ulem}
\usepackage{hyperref}

\usepackage{verbatim}

\makeatletter
\numberwithin{equation}{section}
\numberwithin{figure}{section}
\theoremstyle{plain}
\newtheorem{thm}{\protect\theoremname}
\theoremstyle{definition}
\newtheorem{defn}[thm]{\protect\definitionname}
\theoremstyle{remark}
\newtheorem*{rem*}{\protect\remarkname}
\theoremstyle{plain}
\newtheorem{fact}[thm]{\protect\factname}
\newtheorem{claim}[thm]{\protect\claimname}
\theoremstyle{remark}
\newtheorem{rem}[thm]{\protect\remarkname}
\theoremstyle{plain}
\newtheorem{lem}[thm]{\protect\lemmaname}
\newlist{casenv}{enumerate}{4}
\setlist[casenv]{leftmargin=*,align=left,widest={iiii}}
\setlist[casenv,1]{label={{\itshape\ \casename} \arabic*.},ref=\arabic*}
\setlist[casenv,2]{label={{\itshape\ \casename} \roman*.},ref=\roman*}
\setlist[casenv,3]{label={{\itshape\ \casename\ \alph*.}},ref=\alph*}
\setlist[casenv,4]{label={{\itshape\ \casename} \arabic*.},ref=\arabic*}
\theoremstyle{plain}

\theoremstyle{plain}
\newtheorem{prop}[thm]{\protect\propositionname}
\theoremstyle{plain}
\newtheorem*{cor*}{\protect\corollaryname}
\theoremstyle{plain}
\newtheorem*{lem*}{\protect\lemmaname}

\newtheorem{maintheorem}{Theorem}

\newtheorem{maincor}{Corollary}

\def\N{\mathbb N}
\def\Z{\mathbb Z}
\def\R{\mathbb R}

\def\T{\mathbb T}

\newcommand{\Meng}[2]{\left\{#1\mathrel{}\middle|\mathrel{}#2\right\}}

\newcommand{\abs}[1]{\left\lvert#1\right\rvert}
\newcommand{\norm}[1]{\left\lVert#1\right\rVert}
\newcommand{\trees}{\mathcal{T}\kern-.5mm rees}

\newcommand{\sh}{\operatorname{sh}}
\newcommand{\im}{\operatorname{im}}
\newcommand{\Aut}{\operatorname{Aut}}
\newcommand{\rev}{\operatorname{rev}}

\usepackage{babel}
\providecommand{\lemmaname}{Lemma}
\providecommand{\propositionname}{Proposition}
\providecommand{\theoremname}{Theorem}

\makeatother

\usepackage{babel}
\providecommand{\casename}{Case}
\providecommand{\corollaryname}{Corollary}
\providecommand{\definitionname}{Definition}
\providecommand{\factname}{Fact}
\providecommand{\claimname}{Claim}
\providecommand{\lemmaname}{Lemma}
\providecommand{\propositionname}{Proposition}
\providecommand{\remarkname}{Remark}
\providecommand{\theoremname}{Theorem}

\begin{document}
	
	\title{Non-Classifiability of Kolmogorov Diffeomorphisms up to Isomorphism}
	\author{Marlies Gerber$^1$}   
	\thanks{$^1$ Indiana University, Department of Mathematics, Bloomington, IN 47405, USA}
	\author{Philipp Kunde$^2$} 
	\thanks{$^2$ Jagiellonian University in Krakow, Faculty of Mathematics and Computer Science, 30-348 Krakow, Poland. The research of P.K. is part of the project No. 2021/43/P/ST1/02885 co-funded by the National Science Centre and the European Union's Horizon 2020 research and innovation programme under the Marie Sklodowska-Curie grant agreement no. 945339.}
	%\date{\today}
	
	\begin{abstract} 
		We consider the problem of classifying Kolmogorov automorphisms (or $K$-automorphisms for brevity) up to isomorphism. Within the collection of measure-preserving transformations, Bernoulli shifts have the ultimate mixing property, and $K$-automorphisms have the next-strongest mixing properties of any widely considered family of transformations. J.~Feldman observed that unlike Bernoulli shifts, the family of $K$-automorphisms cannot be classified up to isomorphism by a complete numerical Borel invariant. This left open the possibility of classifying $K$-automorphisms with a more complex type of Borel invariant. We show that this is impossible, by proving that the isomorphism equivalence relation restricted to $K$-automorphisms, considered as a subset of the Cartesian product of the set of $K$-automorphisms with itself, is a complete analytic set, and hence not Borel. Moreover, we prove this remains true if we restrict consideration to $K$-automorphisms that are also $C^{\infty}$ diffeomorphisms. This shows in a concrete way that the problem of classifying $K$-automorphisms up to isomorphism is intractible. 
		
		%{\color{red} We also describe an explicit construction of continuum many smooth $K$-diffeomorphisms that are pairwise non-Kakutani equivalent (and therefore non-isomorphic).}
	\end{abstract}
	
	\maketitle
	
	\insert\footins{\footnotesize - 
		}
	
	\section{Introduction} \label{sec:intro}
	A classical question in ergodic theory,
	posed by J.~von Neumann in 1932 \cite{Ne}, is the isomorphism problem:
	classify measure-preserving transformations up to isomorphism.
	We let $\Aut(\mu)$ denote the set
	of all invertible measure-preserving transformations
	on a fixed standard non-atomic probability space, $(\Omega,\mathcal{B},\,\mu).$
	Two elements $S,T\in\Aut(\mu)$ are said to be \emph{isomorphic}
	if there exists $\varphi\in\Aut(\mu)$ such that $S\circ\varphi$
	and $\varphi\circ T$ agree $\mu$-almost everywhere. Two great successes
	in the classification of measure-preserving transformations are the Halmos-von Neumann classification
	of ergodic elements of $\Aut(\mu)$ with pure point spectrum by countable subgroups of
	the unit circle \cite{HN42} and D.~Ornstein's classification of Bernoulli
	shifts by their metric entropy \cite{Or70}. However, the isomorphism
	problem remains unsolved for general ergodic measure-preserving transformations.
	
	In more recent years, focus has shifted to the study of the logical
	complexity of the isomorphism problem, and a series of \emph{anti-classification}
	results have appeared. These results demonstrate the complexity of
	the problem of classification up to isomorphism, or other equivalence
	relations, and give some limitations on what (if any) type of classification
	is potentially achievable. (See \cite{Fsurvey} for a discussion of a hierarchy
	of potential classifications.)
	
	To describe some of the anti-classification results, we endow $\Aut(\mu)$
	with the weak topology. (Recall that $T_{n}\to T$ in the weak topology
	if and only if we have $\mu(T_{n}(A)\Delta T(A))\to0$ for every $A\in\mathcal{B}.)$
	This topology is compatible with a complete separable metric and hence
	makes $\Aut(\mu)$ into a Polish space. Let $\mathcal{E}$, $\mathcal{WM}$,
	$\mathcal{M}$, and $\mathcal{K}$ denote, respectively, the sets of 
	ergodic, weakly mixing, mixing, and Kolmogorov automorphisms (or $K$-automorphisms for brevity)
	in $\Aut(\mu)$. The sets $\mathcal{E}$ and $\mathcal{WM}$ form (dense)
	$G_{\delta}$-sets in $\Aut(\mu).$ Thus the induced topologies
	on $\mathcal{E}$ and $\mathcal{WM}$ are Polish as well. Moreover, the sets
	$\mathcal{M}$ and $\mathcal{K}$ 
	are Borel \cite[Proposition 43]{Fbook} and dense in $\Aut(\mu)$
	(denseness can be seen from \cite{Ha44}), 
	but these sets are meager
	in the sense of category \cite{Ro}. 
	
	The strongest type of classification that is feasible for a Borel family
	of automorphisms in $\Aut(\mu)$ is a Borel function from the family 
	to the real numbers which is a
	complete invariant, that is,  two automorphisms $S$ and $T$ 
	in the family are isomorphic if and only
	if the values of the function are the same at $S$ and $T$. For
	example, \cite{Or70} shows that metric entropy is a complete isomorphism
	invariant for Bernoulli shifts. In contrast, J.~Feldman \cite{Fe74}
	proved that there is no Borel function from the $K$-automorphisms to
	the real numbers that is a complete isomorphism invariant. 
	
	A somewhat weaker type of classification is a Borel way of associating
	complete algebraic invariants to a given Borel subset of $\Aut(\mu)$, such
	as the Halmos-von Neumann classification. However, G.~Hjorth \cite{Hj01}
	showed that there is no Borel way of associating algebraic invariants
	to $\mathcal{E}$ that completely determines isomorphism. Moreover,
	Foreman and B.~Weiss \cite{FW0} proved that there is no generic class
	(that is, a dense $G_{\delta}$-set) within $\Aut(\mu)$ for which there
	is a Borel way of associating a complete algebraic invariant.
	
	Both of these types of classification, by a Borel function that is
	a complete invariant or by a Borel way of associating complete algebraic
	invariants for a particular family $\mathcal{F}$ in $\Aut(\mu),$ 
	imply that the isomorphism relation $\mathcal{R}$,
	when viewed as a subset of the Cartesian product $\Aut(\mu)\times\Aut(\mu)$,
	is such that $\mathcal{R}\cap(\mathcal{F}\times\mathcal{F})$ is Borel.
	Therefore if $\mathcal{R}\cap(\mathcal{F}\times\mathcal{F})$ is not Borel,
	the types of classification discussed above
	are not possible for $\mathcal{F}.$ Hjorth \cite{Hj01} proved that the isomorphism relation
	$\mathcal{R}$ on $\Aut(\mu)$ is not a Borel set and is, in fact, a complete
	analytic set, which is a maximally complex analytic
	set. However, his proof used non-ergodic transformations in an essential
	way. In the case of the family $\mathcal{E}$ of ergodic automorphisms, Foreman,
	D.~Rudolph, and Weiss \cite{FRW} proved that $\mathcal{R}\cap(\mathcal{E}\times\mathcal{E})$
	is a complete analytic set, and hence not Borel. The results in \cite{Hj01}
	and \cite{FRW} are among the strongest types of anti-classification
	results, because they show that even the weakest types of classification,
	as described in \cite{Fsurvey}, are not possible. Showing that an equivalence
	relation on a family in $\mathcal{F}$ is not Borel can be interpreted as saying
	that there is no general method of determining---in a countable number of
	steps, using countable amounts of information---whether or not two
	transformations in the family are equivalent. This demonstrates 
	in a concrete way that the classification problem is intractible. 
	
	In the search for classification or anti-classification results, we
	place increasingly stringent conditions on the family of transformations
	under consideration. Bernoulli shifts, which can be
	viewed as the most random type of transformations, have the strongest
	type of classification, a complete Borel invariant. But what about
	classes of transformations which satisfy a randomness assumption that
	is stronger than ergodicity but weaker than Bernoulliness, such as
	weakly mixing transformations, mixing transformations, and $K$-automorphisms? 
	
	In 1958, A.~Kolmogorov defined what are now called $K$-automorphisms in terms of
	a 0-1 law  \cite{Ko58}. V.~Rokhlin and Ya.~Sinai \cite{RS61} proved that an automorphism
	is $K$ if and only if it has completely positive (metric)
	entropy, that is, the entropy is positive with respect to every non-trivial
	partition. $K$-automorphisms have the strongest mixing properties of
	any widely considered family of transformations. Until Ornstein constructed
	a counterexample \cite{Or73}, it was unknown whether all $K$-automorphisms
	were, in fact, Bernoulli. Soon afterwards, Ornstein and P.~Shields
	constructed a family of non-Bernoulli $K$-automorphisms indexed by sequences
	of zeros and ones, such that two automorphisms in this family are
	isomorphic if and only if the corresponding sequences agree except
	possibly at finitely many terms \cite{OS73}. This led Feldman to the
	previously mentioned result that there is no complete numerical Borel
	isomorphism invariant for $K$-automorphisms.
	
	This left open the following questions that appear in Foreman's surveys \cite{Fbook}
	and \cite{Fsurvey} regarding the complexity of the isomorphism
	equivalence relation on $K$-automorphisms: Is the isomorphism
	equivalence relation restricted to the $K$-automorphisms
	\begin{itemize}
		\item reducible to $E_{0}$ (the minimal level of complexity possible given
		Feldman's result)?
		\item reducible to $=^{+}$(the level of complexity of the Halmos-von Neumann
		type of classification)?
		\item a Borel equivalence relation?
	\end{itemize}
	(The last of these questions was also posed by Newberger and is stated as Question~1.16 in \cite{Problems}.)
	We prove that for the isomorphism equivalence
	relation $\mathcal{R},$ the set $\mathcal{R}\cap(\mathcal{K}\times\mathcal{K})$
	is a complete analytic set, and therefore not Borel, thereby providing
	a negative answer to all three of the above questions. 
	
	The work of \cite{FRW} provides not only anti-classification
	results, but also new methods of construction of ergodic automorphisms.
	The automorphisms are constructed using certain concatenations of
	blocks of symbols (which can also be interpreted in terms of cutting
	and stacking) together with equivalence relations on these blocks.
	In \cite{FRW} there are some random choices in the constructions,
	while in \cite{GK3}, we obtained anti-classification results with a
	method based on \cite{FRW}, but we replaced the random choices by
	a deterministic method using Feldman patterns. In sections \ref{sec:KAut}--\ref{sec:ProofKAut} below, we
	utilize the transformations from \cite{GK3} to construct a
	certain family of skew products that are known to be $K$-automorphisms
	by the results of I.~Meilijson \cite{Me74}. This leads to our anti-classification
	results for $\mathcal{K}.$
	
	Our anti-classification results also hold when the isomorphism relation
	is replaced by Kakutani equivalence. Two ergodic automorphisms are said to
	be \emph{Kakutani equivalent} if they are isomorphic to measurable
	cross-sections of the same ergodic flow. Another equivalent definition
	is that  $T,S\in\mathcal{E}$ are
	Kakutani equivalent if there exist two sets $A,B$ of positive measure
	such that the first return map $T_{A}$ of $T$ to $A$ with the induced
	measure $\mu_{A}=\mu/\mu(A)$ on $A$ is isomorphic to the analogously defined
	first return map $S_{B}$ of $S$ to $B.$ Since the Kakutani equivalence
	relation is weaker than isomorphism, one might expect classification
	to be simpler. However, in \cite{GK3} we showed that the Kakutani equivalence
	relation on ergodic automorphisms is also a complete analytic set. In the current
	paper, we show that this remains true for $K$-automorphisms as well.
	The methods in \cite{GK3} for proving non-classifiability of ergodic
	automorphisms \emph{up to Kakutani equivalence} are also useful in the current
	paper, even for obtaining non-classifiability of $K$-automorphisms \emph{up
		to isomorphism}. 
	
	The classification problem (with respect to isomorphism or Kakutani
	equivalence) can also be restricted to the class $\text{Diff}^{\infty}(M,\mu)$ of smooth 
	diffeomorphisms of a compact manifold $M$ that preserve a smooth
	measure $\mu.$ But even with this restriction, Foreman and Weiss \cite{FW3}
	proved that the isomorphism relation on smooth ergodic diffeomorphisms 
	is a complete analytic set in $\text{Diff}^{\infty}(M,\mu)\times \text{Diff}^{\infty}(M,\mu)$, and
	therefore not Borel if $M$ is the torus, disk, or annulus. For the
	case of the torus, S. Banerjee and P. Kunde \cite{BK2} generalized this result 
	to the real-analytic
	setting. In \cite{GK3} we obtained the analogous
	results for Kakutani equivalence both in the smooth and the real-analytic
	settings. In Section~\ref{sec:smooth}, we show that anti-classification results
	for $K$-automorphisms for both isomorphism and Kakutani equivalence
	also hold in the smooth setting. The extension to the real-analytic
	setting is still an open problem. 
	
	Our anti-classification results for $K$-automorphisms in the smooth setting hold on the five-dimensional torus, and we use a construction due to Katok \cite{Ka80} of $C^{\infty}$ diffeomorphisms that are $K$ but not loosely Bernoulli, and therefore not Bernoulli. Another family of  $C^{\infty}$ diffeomorphisms that are $K$ but not loosely Bernoulli is due to Rudolph \cite{Ru88}. Rudolph's examples are, in fact, real-analytic five-dimensional analogs of the $T$,$T^{-1}$ transformation, which S.~Kalikow proved is not loosely Bernoulli \cite{Kal82}.
	%Our approach to anti-classification results
	%for $K$-automorphisms in the smooth setting uses a construction due
	%to Katok \cite{Ka80} of $C^{\infty}$ diffeomorphisms in dimension 5
	%that are $K$ but not loosely Bernoulli, and therefore not Bernoulli. (
	More recently, A.~Kanigowski, F.~Rodriguez Hertz, and K.~Vinhage
	\cite{KRV} gave an example of a $C^{\infty}$ diffeomorphism in
	dimension $4$ that is $K$ and loosely Bernoulli, but not Bernoulli. (It
	is an open problem whether there exist $C^{\infty}$ diffeomorphisms
	that are $K$ but not Bernoulli in dimension $3$.) 
	
	T. Austin \cite{Aupp} introduced a new isomorphism invariant called scenery entropy that distinguishes
	between some of the known non-Bernoulli $K$-automorphisms, including certain smooth examples. As stated in 
	\cite{Aupp}, conditionally on an invariance principle for certain local times, this leads to a continuum of pairwise non-isomorphic smooth non-Bernoulli $K$-automorphisms (of the same measure entropy) on a fixed compact manifold. However, as far as we know, this invariance principle has not yet been established. 
	In Section~\ref{subsec:continuum}, which 
	can be read independently of Sections~\ref{sec:results}--\ref{sec:ProofKAut}, we indicate another construction
	of a continuum of $C^{\infty}$ diffeomorphisms of the five-dimensional torus
	of the same measure entropy that are $K$ and pairwise non-Kakutani equivalent (and therefore non-isomorphic). We include this explicit
	construction because it utilizes some of the ideas in Sections~\ref{subsec:ConstrSmooth}--\ref{subsec:ProofPropPart2},
	but is much simpler. It is not known if there exists an uncountable family of pairwise non-isomorphic smooth $K$-automorphisms of the same entropy on a four-dimensional manifold. In particular, it is an open question if the type of construction in \cite{KRV}, varying the skewing function, would lead to such a family.
	
	Since $K$-automorphisms are mixing, we obtain anti-classification results
	for mixing automorphisms as a corollary. However, the results for
	$K$-automorphisms do not cover the collections of mixing or weakly mixing
	automorphisms in the zero entropy case. In an earlier work \cite{Kupp},
	the second author has shown that the isomorphism relation, as well
	as the Kakutani equivalence relation, on zero entropy smooth weakly
	mixing transformations of the torus, disk, or annulus is a complete
	analytic set. Moreover, he also obtained this result for zero entropy
	real-analytic weakly mixing diffeomorphisms of the 2-torus. The methods
	used in \cite{Kupp} are very different from those in the current paper,
	and they utilize the ``twisted'' version of the Anosov-Katok construction
	\cite{AK70}. The earlier work of Foreman and Weiss \cite{FW3} in the smooth
	setting uses the ``untwisted'' version of the Anosov-Katok method,
	and the transformations produced are ergodic, but not weakly mixing.
	For the collection $\mathcal{ZM}$ of zero entropy mixing automorphisms,
	it is an open question whether the isomorphism equivalence relation
	(or the Kakutani equivalence relation) restricted to $\mathcal{ZM}$
	is a Borel subset of $\mathcal{ZM}\times\mathcal{ZM}.$ 
	
	Other anti-classification results were recently obtained in the context of topological conjugacy of diffeomorphisms. Two diffeomorphisms $f$ and $g$ of a manifold $M$ are said to be topologically conjugate if there exists a homeomorphism $h$ of $M$ such that $f\circ h=h\circ g$. Foreman and A.~Gorodetski \cite{FG} showed that for a closed manifold of any dimension the topological conjugacy equivalence relation, considered as a subset of the Cartesian product of the set of $C^{\infty}$ diffeomorphisms with itself, is not Borel. In the case of $C^{\infty}$ diffeomorphisms of the circle with Liouvillean rotation numbers, Kunde considered the equivalence relation of topological conjugacy with the additional requirement that the conjugating map $h$ be a $C^{r}$ diffeomorphism. He proved that for each $r$, $1\le r\le\infty,$ there is no complete numerical Borel invariant \cite{KuCircle}. 
	
	\section{Statement of Results and Outline of Proofs} \label{sec:results}
	
	\subsection{Results for $K$-automorphisms}\label{subsec:outline}
	
	Let $(\Omega,\mathcal{B},\mu)$ be a standard non-atomic probability
	space, and let $\mathcal{E}$, $\mathcal{H}$, and $\mathcal{K}$
	denote the set of all measure-preserving automorphisms of $\Omega$
	that are, respectively, ergodic, ergodic with positive entropy, and
	$K$. 
	
	To prove our anti-classification results we follow the general strategy from \cite{FRW}. An important tool is the concept of a
	reduction. 
	\begin{defn}
		Let $X$ and $Y$ be Polish spaces and $A\subseteq X$, $B\subseteq Y$.
		A function $f:X\to Y$ \emph{reduces} $A$ to $B$ if and only if
		for all $x\in X$: $x\in A$ if and only if $f(x)\in B$. 
		
		Such a function $f$ is called a \emph{Borel (respectively, continuous) reduction}
		if $f$ is a Borel (respectively, continuous) function.
	\end{defn}
	
	$A$ being reducible to $B$ can be interpreted as saying that $B$
	is at least as complicated as $A$. We note that if $B$ is Borel
	and $f$ is a Borel reduction, then $A$ is also Borel. Equivalently, if $f$ is a Borel reduction of $A$ to $B$ and $A$ is not Borel,
	then $B$ is not Borel.
	
	\begin{defn}
		If $X$ is a Polish space and $B\subseteq X$, then $B$ is \emph{analytic}
		if and only if it is the continuous image of a Borel subset of a Polish
		space. Equivalently, there is a Polish space $Y$ and a Borel set
		$C\subseteq X\times Y$ such that $B$ is the $X$-projection of $C$.
	\end{defn}
	
	\begin{defn}
		An analytic subset $A$ of a Polish space $X$ is called \emph{complete analytic} if every analytic set can be continuously reduced to A.
	\end{defn}
	If $A$ is continuously reducible to an analytic set $B$ and $A$ is a complete analytic set, then so is $B$.
	There are analytic sets that are not Borel  (see, for example, \cite[section 14]{Kechris}). This implies that a complete analytic set is not Borel. 
	
	In the Polish space $\mathcal{T}rees$ (to be defined in Section~\ref{subsec:trees}), the subset consisting of trees with an infinite branch is an example of a complete analytic set \cite[section~27]{Kechris}.
	
	The map $\Phi$ in Theorem \ref{mainthm:KAut} below provides a continuous reduction from the set of trees with an infinite branch to $\Meng{T \in \mathcal{K}}{T\text{\  is isomorphic to\  } T^{-1}}$.
	
	\begin{maintheorem}\label{mainthm:KAut}
		There is a continuous map $\Phi:\mathcal{T}rees\to\mathcal{K}$
		such that for $\mathcal{T}\in\mathcal{T}rees$ with $V=\Phi(\mathcal{T})$,
		the following conditions hold: 
		\begin{enumerate}
			\item [(a)]If $\mathcal{T}$ has an infinite branch, then $V$ and $V^{-1}$
			are isomorphic (and therefore Kakutani equivalent);
			\item [(b)]If $\mathcal{T}$ does not have an infinite branch, then $V$and
			$V^{-1}$ are not Kakutani equivalent (and therefore not isomorphic).
		\end{enumerate}
	\end{maintheorem}
	
	To obtain our anti-classification result in Corollary~\ref{maincor:KAut}, we define the continuous embedding $i: \Aut(\mu) \to \Aut(\mu)\times \Aut(\mu)$ by $i(T)=(T,T^{-1})$, and note that $i\circ\Phi$ is a continuous reduction from the set of trees with an infinite branch to $\Meng{(S,T)\in\mathcal{K}\times\mathcal{K}}{S \text{\ is isomorphic to\ } T}$. This, together with the observation
	in Subsection \ref{subsec:Analytic} that the isomorphism and Kakutani equivalence relations considered as subsets of $\mathcal{E}\times\mathcal{E}$ are analytic sets, yields Corollary~\ref{maincor:KAut}. 
	
	\begin{maincor}\label{maincor:KAut}
		The set $\{(S,T)\in\mathcal{K}\times\mathcal{K}:S$
		and $T$ are isomorphic\} is a complete analytic subset of $\mathcal{K}\times\mathcal{K}$
		and hence not a Borel set. The same is true for ``isomorphic'' replaced by ``Kakutani
		equivalent''. 
	\end{maincor}
	
	To prove Theorem \ref{mainthm:KAut}, we start with the following more explicit Theorem~\ref{mainthm:Kproof}. 
	
	\begin{maintheorem} \label{mainthm:Kproof}
		{\bf Part 1. } There is a continuous
		map $\Upsilon:\mathcal{T}rees\to\mathcal{H}$ such that for $\mathcal{T}\in\mathcal{T}rees$
		with $U=\Upsilon(\mathcal{T})$, the following conditions hold: 
		\begin{enumerate}
			\item[(a)] If $\mathcal{T}$ has an infinite branch, then $U$ and $U^{-1}$
			are isomorphic;
			\item[(b)] If $\mathcal{T}$ does not have an infinite branch, then $U$
			and $U^{-1}$ are not Kakutani equivalent (and therefore not isomorphic).
		\end{enumerate}
		{\bf Part 2.} The map $\Upsilon$ in Part 1 can
		be chosen so that there is also a set $A\in\mathcal{B}$ of positive
		measure such that for $\mathcal{T}\in\mathcal{T}rees$ with $U=\Upsilon(\mathcal{T})$
		we have:
		\begin{enumerate}
			\item[(c)] The induced map $U_{A}$ is a $K$-automorphism;
			\item[(d)] If $\mathcal{T}$ has an infinite branch, then there exists
			$\phi\in$ Aut$(\mu)$ such that $\phi\circ U=U^{-1}\circ\phi$ and $\phi(A)=A$.
			Consequently, $U_{A}$ is isomorphic to $(U_{A})^{-1}.$
		\end{enumerate} 
	\end{maintheorem}
	
	Part 1 of Theorem \ref{mainthm:Kproof} with $\mathcal{H\,}$ replaced by $\mathcal{E}$
	is in \cite{GK3}. If, in addition, ``Kakutani equivalent'' is replaced
	by ``isomorphic'' the result was proved by M.~Foreman, D.~Rudolph,
	and B.~Weiss in \cite{FRW}, and their methods, combined with techniques
	from Kakutani equivalence theory (see, for example, \cite{ORW}),
	were used in \cite{GK3}. 
	
	\begin{proof}[Proof of Theorem~\ref{mainthm:KAut}]
		Theorem \ref{mainthm:KAut} follows from Theorem~\ref{mainthm:Kproof} by taking $V=U_{A}$ equipped with the normalized induced measure $\mu_A$. For part (b) of
		Theorem \ref{mainthm:KAut}, observe that $V$ and $V^{-1}$ are not Kakutani equivalent
		if $U$ and $U^{-1}$ are not Kakutani equivalent. 
	\end{proof}
	
	For $\mathcal{T}\in\mathcal{T}rees$,
	the transformation $U=\Upsilon(\mathcal{T})$ in Theorem~\ref{mainthm:Kproof} is defined as $U=\Psi(\mathcal{T}) \times \mathbb{B}$, where 
	$\Psi(\mathcal{T})$ is the transformation constructed in sections 7--8 of \cite{GK3} and $\mathbb{B}$ is the $(1/2,1/2)$-Bernoulli shift. The definition of $\Psi(\mathcal{T})$ is reviewed in Section~\ref{sec:review} below.
	
	\begin{proof}[Proof of Part~1~(a) of Theorem~\ref{mainthm:Kproof}]
		According to \cite[section~5]{GK3}, $\Psi(\mathcal{T})$ is isomorphic to $\Psi(\mathcal{T})^{-1}$ if $\mathcal{T}$ has an infinite branch. Part~(a) follows immediately.
	\end{proof}
	Part 2 of Theorem~\ref{mainthm:Kproof} is
	based on \cite{Me74} and is proved in Section~\ref{sec:KAut}. The main new idea in the automorphism case of the present work occurs in
	the proof of Part~1~(b) in Section~\ref{subsec:Non-Equiv}.

	\subsection{Results for smooth $K$-diffeomorphisms}\label{subsec:smoothK}
	If $M$ is a finite-dimensional compact $C^{\infty}$ manifold and $\mu$ is a standard measure on $M$ determined by a smooth volume element, we let $\text{Diff}^{\infty}(M,\mu)$ denote the $\mu$-preserving $C^{\infty}$ diffeomorphisms on $M$. Equipped with the $C^{\infty}$ topology, $\text{Diff}^{\infty}(M,\mu)$ is a Polish space. We can also view each element of $\text{Diff}^{\infty}(M,\mu)$ as an element of $\Aut(\mu)$, and with this interpretation, it follows from \cite{Bj} that $\text{Diff}^{\infty}(M,\mu)$ is a Borel subset of $\Aut(\mu)$.
	
	In Section~\ref{sec:smooth} we obtain anti-classification results for $K$-automorphisms that are
	also $C^{\infty}$ diffeomorphisms preserving a smooth measure. Our work utilizes A. Katok's construction of  smooth non-Bernoulli $K$-automorphisms 
	\cite{Ka80}. 
	
	Let $\mathbb{T}^{n}$ denote the $n$-dimensional
	torus, that is, the Cartesian product of $n$ copies of $S^{1}=\mathbb{R}/\mathbb{Z},$ and let $\lambda_{n}$ denote
	Lebesgue measure on $\mathbb{T}^{n}.$ 
	
	We will prove the following smooth version of Theorem~\ref{mainthm:KAut}.
	\begin{maintheorem}\label{mainthm:Kdiffeo}
		There is a continuous map
		\begin{align*}
			\Phi: \mathcal{T}rees & \to \left(\mathcal{K}\cap\text{Diff}^{\,\infty}(\mathbb{T}^{5},\lambda_{5})\right) \times \left(\mathcal{K}\cap\text{Diff}^{\,\infty}(\mathbb{T}^{5},\lambda_{5})\right) \\
			\mathcal{T} & \mapsto \left( \Phi_1(\mathcal{T}), \Phi_2(\mathcal{T})\right)
		\end{align*}
		such that
		\begin{enumerate}
			\item[(a)] If $\mathcal{T}$ has an infinite branch, then $\Phi_{1}(\mathcal{T})$
			and $\Phi_{2}(\mathcal{T})$ are isomorphic (and therefore Kakutani
			equivalent);
			\item[(b)] If $\mathcal{T}$ does not have an infinite branch, then $\Phi_{1}(\mathcal{T})$
			and $\Phi_{2}(\mathcal{T})$ are not Kakutani equivalent (and therefore
			not isomorphic).
		\end{enumerate}
	\end{maintheorem}
	As in the case of $K$-automorphisms, this immediately gives the following anti-classification result.
	
	\setcounter{maincor}{2}
	\begin{maincor}
		The set 
		\[
		\Meng{(S,T)\in \left(\mathcal{K}\cap\text{Diff}^{\,\infty}(\mathbb{T}^{5},\lambda_{5})\right) \times \left(\mathcal{K}\cap\text{Diff}^{\,\infty}(\mathbb{T}^{5},\lambda_{5})\right)}{\text{$S$ and $T$ are isomorphic}} 
		\]
		is a complete analytic subset of $\text{Diff}^{\,\infty}(\mathbb{T}^{5},\lambda_{5})\times \text{Diff}^{\,\infty}(\mathbb{T}^{5},\lambda_{5})$ with respect to the $C^{\infty}$ topology and hence not a Borel set. The same is true for ``isomorphic'' replaced by ``Kakutani
		equivalent''. 
	\end{maincor}
	
	Since the weak topology on $\text{Diff}^{\infty}(\mathbb{T}^{5},\lambda_{5})$ is coarser than the $C^{\infty}$ topology, $\Meng{(S,T)\in \left(\mathcal{K}\cap\text{Diff}^{\infty}(\mathbb{T}^{5},\lambda_{5})\right) \times \left(\mathcal{K}\cap\text{Diff}^{\infty}(\mathbb{T}^{5},\lambda_{5})\right)}{\text{$S$ and $T$ are isomorphic}}$ is also a complete analytic subset of $\text{Diff}^{\infty}(\mathbb{T}^{5},\lambda_{5})\times \text{Diff}^{\infty}(\mathbb{T}^{5},\lambda_{5})$ in the weak topology.

	\section{Preliminary material}
	
	\subsection{The space of trees}\label{subsec:trees}
	As we mentioned in 
	Section ~\ref{subsec:outline},
	the collection of trees
	%with
	within  $\mathcal{T}\kern-.5mm rees$
	that have 
	an infinite branch is an example of a complete analytic set, and it is fundamental to our construction. 
	We now define trees, the space  $\mathcal{T}\kern-.5mm rees$, and infinite branches of trees.
	
	Let $\mathbb{N}$ be the nonnegative integers, and let $\mathbb{N}^{<\mathbb{N}}$ denote the collection 
	of finite sequences in $\mathbb{N}$. By convention, the empty sequence $\emptyset$ is in $\mathbb{N}^{<\mathbb{N}}$.

	%Here,
	\begin{defn}
		A \emph{tree} is a set $\mathcal{T}\subseteq\mathbb{N}^{<\mathbb{N}}$
		such that if $\tau=\left(\tau_{1},\dots,\tau_{n}\right)\in\mathcal{T}$
		and $\sigma=\left(\tau_{1},\dots,\tau_{m}\right)$ with $m\leq n$
		is an initial segment of $\tau$, then $\sigma\in\mathcal{T}$.
		Each tree $\mathcal{T}$ contains $\emptyset$, which is regarded as
		an initial segment of all $\tau\in\mathcal{T}$.
	\end{defn}
	
	\begin{defn} 
		A function $f:\{0,\dots,n-1\}\to (\tau_1,\dots,\tau_n)\in\mathcal{T}$ is called a \emph{branch} of $\mathcal{T}$ of length $n$. 
		The collection  $\mathcal{T}\kern-.5mm rees$
		consists of those trees that contain branches of arbitrarily long length.
		An \emph{infinite branch} through $\mathcal{T}$ is a function $f:\mathbb{N}\to\mathbb{N}$ such that for all $n\in\mathbb{N}$ we have $\left(f(0),\dots,f(n-1)\right)\in\mathcal{T}$. If a tree has an infinite branch, it is called \emph{ill-founded}.
		If it does not have an infinite branch, it is called \emph{well-founded}.
	\end{defn}
	
	To describe a topology on the collection of trees, let $\left\{ \sigma_{n}:n\in\mathbb{N}\right\} $
	be an enumeration of $\mathbb{N}^{<\mathbb{N}}$ with the property
	that every 
	initial segment
	of $\sigma_{n}$ is some $\sigma_{m}$
	for $m \leq n$. Under this enumeration subsets $S\subseteq\mathbb{N}^{<\mathbb{N}}$
	can be identified with characteristic functions $\chi_{S}:\mathbb{N}\to\left\{ 0,1\right\} $.
	Such $\chi_{S}$ can be viewed as 
	members of
	an infinite product space $\left\{ 0,1\right\} ^{\mathbb{N}^{<\mathbb{N}}}$
	homeomorphic to the Cantor space. Here, each function $a:\left\{ \sigma_{m}:m<n\right\} \to\left\{ 0,1\right\} $
	determines a basic open set 
	\[
	\left\langle a\right\rangle =\left\{ \chi:\chi\upharpoonright\left\{ \sigma_{m}:m<n\right\} =a\right\} \subseteq\left\{ 0,1\right\} ^{\mathbb{N}^{<\mathbb{N}}}
	\]
	and the collection of all such $\left\langle a\right\rangle $ forms
	a basis for the topology. In this topology the collection of trees is a closed (hence
	compact) subset of $\left\{ 0,1\right\} ^{\mathbb{N}^{<\mathbb{N}}}$. Moreover, the collection $\mathcal{T}\kern-.5mm rees$ 
	%of trees containing arbitrarily long finite sequences
	is a dense $\mathcal{G}_{\delta}$ subset. Hence, $\mathcal{T}\kern-.5mm rees$ is a Polish space. 
	
	Since the topology on the space of trees was introduced via basic
	open sets, 
	% giving us a finite amount of information about the trees in it,
	we can characterize continuous maps defined on $\mathcal{T}\kern-.5mm rees$
	as follows.
	\begin{fact}
		\label{fact:contTree}Let $Y$ be a topological space. Then a map
		$f:\mathcal{T}\kern-.5mm rees\to Y$ is continuous if and only if for 
		every open set
		$O\subseteq Y$ and every $\mathcal{T}\in\mathcal{T}\kern-.5mm rees$ with
		$f(\mathcal{T})\in O$ there exists $M\in\mathbb{N}$ such that for all
		$\mathcal{T}^{\prime}\in\mathcal{T}\kern-.5mm rees$ we have:
		
		if $\mathcal{T}\cap\left\{ \sigma_{n}:n\leq M\right\} =\mathcal{T}^{\prime}\cap\left\{ \sigma_{n}:n\leq M\right\} $,
		then $f\left(\mathcal{T}^{\prime}\right)\in O$. 
	\end{fact}
	
	During our constructions the following maps will prove useful.
	\begin{defn}
		\label{def:M-and-s}We define a map $M:\mathcal{T}\kern-.5mm rees\to\mathbb{N}^{\mathbb{N}}$
		by setting $M\left(\mathcal{T}\right)(s)=n$ if and only if $n$ is
		the least number such that $\sigma_{n}\in\mathcal{T}$ and the length $lh(\sigma_n)=s$.
		Dually, we also define a map $s:\mathcal{T}\kern-.5mm rees\to\mathbb{N}^{\mathbb{N}}$
		by setting $s\left(\mathcal{T}\right)(n)$ to be the length of the
		longest sequence $\sigma_{m}\in\mathcal{T}$ with $m\leq n$. 
	\end{defn}
	
	\begin{rem*}
		When $\mathcal{T}$ is clear from the context we write $M(s)$ and
		$s(n)$. We also note that $s(n)\leq n$ and that $M$ as well as
		$s$ is a continuous function when we endow $\mathbb{N}$ with the
		discrete topology and $\mathbb{N}^{\mathbb{N}}$ with the product
		topology.
	\end{rem*}
	
	\subsection{Isomorphism and Kakutani Equivalence.}\label{subsec:Analytic}
	In order to obtain Corollary \ref{maincor:KAut} from Theorem \ref{mainthm:KAut}, we need to explain why the 
	isomorphism and Kakutani equivalence relations are analytic subsets of $\mathcal{E}\times\mathcal{E}$.
	For $S,T\in\mathcal{E}$, we write $S\cong T$ if $S$ and $T$ are isomorphic, and $S\sim T$ if $S$ and $T$
	are Kakutani equivalent. Then $\{(S,T)\in\mathcal{E}\times\mathcal{E}:S\cong T\}=\{(S,\phi S\phi^{-1}):S\in\mathcal{E},
	\phi\in \text{Aut}(\mu)\}$. Thus $\{(S,T)\in \mathcal{E}\times\mathcal{E}:S\cong T\}$ is the image
	of $\mathcal{E}\times\text{Aut}(\mu)$ under the continuous map $(S,\phi)\mapsto (S,\phi S\phi^{-1})$, and is therefore analytic.
	
	To show that the Kakutani equivalence relation is an analytic subset of $\mathcal{E}\times\mathcal{E},$ we define, 
	for every Borel set $A\subset [0,1]$ of positive measure, the invertible measure-preserving map $f^A:(A,\mu_A)\to ([0,1],\mu)$ 
	by $f^A(x)=\mu(A)^{-1}\mu(A\cap [0,x]).$ As in the introduction, $\mu_A=\mu/\mu(A)$ on $A$, and for $T\in \mathcal{E}$, $T_A$ 
	denotes the first return map of $T$ to $A$. The Borel sets $\mathcal{B}$ form a Polish space under the metric $d(A,B)=\mu(A\Delta B)$,
	and $\{A\in\mathcal{B}:\mu(A)>0\}$ is an open subset of $\mathcal{B}.$ Note that for $A,B\in\mathcal{B}$ with $\mu(A)>0, \mu(B)>0,$ 
	two automorphisms $T,S\in\mathcal{E}$ are such that $T_A$ is isomorphic to $S_B$ if and only if $f^AT_A(f^A)^{-1}\cong f^BS_B(f^B)^{-1}$, 
	which is true if and only if there exists $\phi\in$ Aut$(\mu)$ such that 
	\begin{equation}\label{eq:Kak}
		\phi f^AT_A(f^A)^{-1}=f^BS_B(f^B)^{-1}\phi.
	\end{equation}
	Let $\mathcal{A}=\{(T,S,A,B,\phi):T,S\in\mathcal{E}, A,B\in\mathcal{B} \text{\ with\ } \mu(A)>0, \mu(B)>0,
	\phi\in \text{Aut}(\mu) \text{\ such that equation (\ref{eq:Kak}) holds} \}.$
	Then $\mathcal{A}$ is a Borel subset of $\mathcal{E}\times\mathcal{E}\times\mathcal{B}\times\mathcal{B}\times$ Aut$(\mu)$
	and the projection of $\mathcal{A}$ to the first two coordinates is $\{(T,S)\in \mathcal{E}\times \mathcal{E}:T\sim S\},$ which
	is therefore analytic.

	\subsection{The $\overline{f}$ distance}
	In this subsection, we present some preliminaries regarding a notion of distance between
	sequences of symbols introduced by Feldman \cite{Fe}. This distance, now called $\overline{f}$, is
	important in the study of Kakutani equivalence, and it plays a role analogous to the Hamming metric 
	$\overline{d}$ in Ornstein's isomorphism theory.
	%Feldman \cite{Fe} introduced the following notion of distance, now called $\overline{f}$.
	\begin{defn}
		\label{def:fbar}A \emph{match} between two strings of symbols $a_{1}a_{2}\dots a_{n}$
		and $b_{1}b_{2}\dots b_{m}$ from a given alphabet $\Sigma,$ is a
		collection $\mathcal{M}$ of pairs of indices $(i_{s},j_{s})$, $s=1,\dots,r$
		such that $1\le i_{1}<i_{2}<\cdots<i_{r}\le n$, $1\le j_{1}<j_{2}<\cdots<j_{r}\le m$
		and $a_{i_{s}}=b_{j_{s}}$ for $s=1,2,\dots,r.$ Then 
		\begin{equation}
			\begin{array}{ll}
				\overline{f}(a_{1}a_{2}\dots a_{n},b_{1}b_{2}\dots b_{m})=\hfill\\
				{\displaystyle 1-\frac{2\max\{|\mathcal{M}|:\mathcal{M}\text{\ is\ a\ match\ between\ }a_{1}a_{2}\cdots a_{n}\text{\ and\ }b_{1}b_{2}\cdots b_{m}\}}{n+m}.}
			\end{array}\label{eq:cl}
		\end{equation}
		We will refer to $\overline{f}(a_{1}a_{2}\cdots a_{n},b_{1}b_{2}\cdots b_{m})$
		as the ``$\overline{f}$-distance'' between $a_{1}a_{2}\cdots a_{n}$
		and $b_{1}b_{2}\cdots b_{m},$ even though $\overline{f}$ does not
		satisfy the triangle inequality unless the strings are all of the
		same length. A match $\mathcal{M}$ is called a\emph{ best possible
			match} if it realizes the supremum in the definition of $\overline{f}$.
	\end{defn}
	
	\begin{rem*}
		Alternatively, one can view a match as an injective order-preserving
		function $\pi:\mathcal{D}(\pi)\subseteq\left\{ 1,\dots,n\right\} \to\mathcal{R}(\pi)\subseteq\left\{ 1,\dots,m\right\} $
		with $a_{i}=b_{\pi(i)}$ for every $i\in\mathcal{D}(\pi)$. Then $\overline{f}\left(a_{1}\dots a_{n},b_{1}\dots b_{m}\right)=1-\max\left\{ \frac{2|\mathcal{D}(\pi)|}{n+m}:\pi\text{ is a match}\right\} $.
	\end{rem*}
	
	The next result from \cite[Lemma 14]{GK3} provides lower bounds on the $\overline{f}$ distance when
	symbols $a_i$ and $b_j$ are replaced by blocks
	$A_{a_i}$ and $A_{b_j}$.
	\begin{lem}[Symbol by block replacement]
		\label{lem:symbol by block replacement}Suppose that $A_{a_{1}},A_{a_{2}},\dots,A_{a_{n}}$
		and $A_{b_{1}},A_{b_{2}},\cdots,A_{b_{m}}$ are blocks of symbols
		with each block of length $L.$ Assume that $\alpha\in(0,1/7),R\ge2,$
		and for all substrings $C$ and $D$ consisting of consecutive symbols
		from $A_{a_{i}}$ and $A_{b_{j}},$ respectively, with $|C|,|D|\ge L/R,$
		we have 
		\[
		\overline{f}(C,D)\ge\alpha\text{, if }a_{i}\ne b_{j}.
		\]
		Then 
		\[
		\overline{f}(A_{a_{1}}A_{a_{2}}\cdots A_{a_{n}},A_{b_{1}}A_{b_{2}}\cdots A_{b_{m}})>\alpha\overline{f}(a_{1}a_{2}\dots a_{n},b_{1}b_{2}\dots b_{m})-\frac{1}{R}.
		\]
	\end{lem}
	
	The following facts from \cite{GK3} will prove useful for our $\overline{f}$ estimates as well.
	\begin{fact}
		\label{fact:omit_symbols}Suppose $a$ and $b$ are strings of symbols
		of length $n$ and $m,$ respectively, from an alphabet $\Sigma$.
		If $\tilde{a}$ and $\tilde{b}$ are strings of symbols obtained by
		deleting at most $\lfloor\gamma(n+m)\rfloor$ terms from $a$ and
		$b$ altogether, where $0<\gamma<1$, then 
		\begin{equation}
			\overline{f}(a,b)\ge\overline{f}(\tilde{a},\tilde{b})-2\gamma.\label{eq:omit_symbols}
		\end{equation}
		Moreover, if there exists a best possible match between $a$ and
		$b$ such that no term that is deleted from $a$ and $b$ to form
		$\tilde{a}$ and $\tilde{b}$ is matched with a non-deleted term,
		then 
		\begin{equation}
			\overline{f}(a,b)\ge\overline{f}(\tilde{a},\tilde{b})-\gamma.\label{eq:omit_symbols2}
		\end{equation}
		Likewise, if $\tilde{a}$ and $\tilde{b}$ are obtained by adding
		at most $\lfloor\gamma(n+m)\rfloor$ symbols to $a$ and $b$, then
		($\ref{eq:omit_symbols2}$) holds.
	\end{fact}
	
	\begin{fact}
		\label{fact:substring_matching}If $x=x_{1}x_{2}\cdots x_{n}$ and
		$y=y_{1}y_{2}\cdots y_{n}$ are decompositions of the strings of symbols
		$x$ and $y$ into corresponding substrings under a best possible
		$\overline{f}$-match between $x$ and $y,$ that is, one that achieves
		the minimum in the definition of $\overline{f},$ then 
		\[
		\overline{f}(x,y)=\sum_{i=1}^{n}\overline{f}(x_{i},y_{i})v_{i},
		\]
		where 
		\begin{equation}
			v_{i}=\frac{|x_{i}|+|y_{i}|}{|x|+|y|}.\label{eq:substring_matching}
		\end{equation}
	\end{fact}
	
	\begin{fact}
		\label{fact:string_length}If $x$ and $y$ are strings of symbols
		such that $\overline{f}(x,y)\le\gamma,$ for some $0\le\gamma<1,$
		then 
		\begin{equation}
			\left(\frac{1-\gamma}{1+\gamma}\right)|x|\leq|y|\le\left(\frac{1+\gamma}{1-\gamma}\right)|x|.\label{eq:string_length}
		\end{equation}
		
	\end{fact}
	
	Even though the $\overline{f}$ distance does not satisfy the triangle inequality for strings of different lengths, Facts~\ref{fact:substring_matching} and \ref{fact:string_length} give us an approximative version of it.
	
	\begin{fact}\label{fact:triangle}
		If $x,y,z$ are strings of symbols with $\overline{f}(x,z)<\frac{1}{2}$ and $\overline{f}(y,z)<\frac{1}{2}$, then
		\begin{align*}
			\overline{f}(x,y) & \leq \frac{|x|+|z|}{|x|+|y|}\overline{f}(x,z)+\frac{|z|+|y|}{|x|+|y|}\overline{f}(z,y) \\
			& \leq \overline{f}(x,z)+\overline{f}(z,y)+8\cdot \overline{f}(x,z) \cdot \overline{f}(z,y).
		\end{align*}
	\end{fact}
	
	\subsection{Feldman patterns} \label{subsec:Feldman}
	To construct the symbolic systems in \cite{GK3}, the $n$-words in the construction sequence are built using specific patterns of blocks. These patterns are called \emph{Feldman patterns} since they originate from Feldman's first example of an ergodic zero-entropy transformation that is not loosely Bernoulli \cite{Fe}. They are useful in our construction, because different Feldman patterns cannot be matched well in $\overline{f}$ even after a finite coding.
	
	Let $T,N,M\in\mathbb{Z}^{+}$. A $(T,N,M)$-Feldman pattern in building
	blocks $A_{1},\dots,A_{N}$ of equal length $L$ is one of the strings
	$B_{1},\dots,B_{M}$ that are defined by
	\begin{eqnarray*}
		B_{1}= &  & \left(A_{1}^{TN^{2}}A_{2}^{TN^{2}}\dots A_{N}^{TN^{2}}\right)^{N^{2M}}\\
		B_{2}= &  & \left(A_{1}^{TN^{4}}A_{2}^{TN^{4}}\dots A_{N}^{TN^{4}}\right)^{N^{2M-2}}\\
		\vdots &  & \vdots\\
		B_{M}= &  & \left(A_{1}^{TN^{2M}}A_{2}^{TN^{2M}}\dots A_{N}^{TN^{2M}}\right)^{N^{2}}
	\end{eqnarray*}
	
	Thus $N$ denotes the number of building blocks, $M$ is the number
	of constructed patterns, and $TN^2$ gives the minimum number of consecutive occurrences
	of a building block. We also note that $B_{r}$ is built with $N^{2(M+1-r)}$ many
	so-called \emph{cycles}: Each cycle winds through all the $N$ building blocks. We refer to $B_r$ as the $(T,N,M)$-Feldman pattern of type $r$.
	
	We will use the following statement from \cite[Proposition 41]{GK3} on the $\overline{f}$-distance between
	different $(T,N,M)$-Feldman patterns.
	\begin{lem}\label{lem:Feldman} Let $N\geq20$, $M\geq2$, and $B_{j}$, $1\leq j\leq M$,
		be the $(T,N,M)$-Feldman patterns in the building blocks $A_{1},\dots,A_{N}$
		of equal length $L$. Assume that $\alpha\in\left(0,\frac{1}{7}\right)$,
		%$R>0$, 
		$R\ge 2$, and $\overline{f}(C,D)>\alpha$, for all substrings $C$ and
		$D$ consisting of consecutive symbols from $A_{i_{1}}$ and $A_{i_{2}}$,
		respectively, where $i_{1}\neq i_{2}$, with $|C|,|D|\geq\frac{L}{R}$. 
		
		Then for all $j,k\in\left\{ 1,\dots,M\right\} $, $j\neq k$, and
		all sequences $B$ and $\overline{B}$ of at least $TN^{2M+2}L$ consecutive
		symbols from $B_{j}$ and $B_{k}$, respectively, we have 
		\begin{equation}
			\overline{f}\left(B,\overline{B}\right)\geq\alpha-\frac{4}{\sqrt{N}}-\frac{1}{R}.
		\end{equation}
	\end{lem}

	\section{Construction of the map \ensuremath{\Upsilon:\mathcal{T}rees\to\mathcal{H}}} \label{sec:KAut}
	For  $\mathcal{T}\in \mathcal{T}\kern-.5mm rees,$
	we let $T=\Psi(\mathcal{T})$
	be as in \cite{GK3}. This automorphism is defined as a left shift on
	certain bi-infinite strings of symbols in an alphabet $\Sigma.$ These
	strings are concatenations of words in a construction sequence $\{\mathcal{W}_{n}\}_{n \in \N}$,
	as we recall in Section \ref{sec:review} below. The map $T$ can also be viewed as an automorphism
	of $[0,1]$ obtained from cutting and stacking subintervals of $[0,1]$,
	where the columns in the $n$th tower all have the same height and
	the same width, and there is a one-to-one correspondence between columns
	in the $n$th tower and words in $\mathcal{W}_{n}.$ The levels of
	a column are labeled with symbols in $\Sigma$ that match the symbols
	in the corresponding word in $\mathcal{W}_{n}.$ This defines a partition
	$\mathcal{P}_{\Sigma}$ of $[0,1]$ that is a generating partition
	for $T.$ 
	
	All of the transformations $U=\Upsilon(\mathcal{T})$ in Theorem~\ref{mainthm:Kproof} are on the Lebesgue space $[0,1]\times \{0,1\}^{\Z}$ with the product of Lebesgue measure $\mu$ on $[0,1]$ and $(1/2,1/2)$-Bernoulli measure on $\{0,1\}^{\Z}$.
	
	The cutting and stacking construction for realizing $T=\Psi(\mathcal{T})$ on $[0,1]$ is done in \cite{GK3} in such a way that $\Psi: \mathcal{T}\kern-.5mm rees \to \Aut([0,1],\mu)$ is continuous. The automorphism $U=\Upsilon(\mathcal{T})$ is the direct product of $\Psi(\mathcal{T})$ and a $(1/2,1/2)$-Bernoulli shift $\mathbb{B}$. We take the set $A$ in Part 2 of Theorem \ref{mainthm:Kproof} to be 
	\begin{equation} \label{eq:setA}
		A=[0,1] \times \Meng{(\upsilon_i)_{i\in \Z} \in \{0,1\}^{\Z}}{\upsilon_0=0}.
	\end{equation}
	\begin{comment}
	The automorphism $U=\Upsilon(\mathcal{T})$ is obtained
	by maintaining the cutting and stacking construction for $T$ while
	introducing another partition $\mathcal{P_{\mathcal{\mathcal{S}}}}$,
	where $\mathcal{S}=\{0,1\}.$ (We assume that $\Sigma$ does not contain
	0 or 1.) Each level of each column in the $n$th tower will be labeled
	with a pair of symbols $(\sigma,\upsilon)$, where $\sigma\in\Sigma$ and
	$\upsilon\in\text{\ensuremath{\mathcal{S}.}}$ Each column of height $h_{n}$
	in the $n$th tower for $T$ is divided into $2^{h_{n}}$ subcolumns
	of equal width that are labeled with all possible sequences of 0's
	and 1's. Thus the 0's and 1's go in independently of each other
	and independently of the symbols in $\Sigma.$ The automorphism $U$
	is the process $(T,\mathcal{P}_{\Sigma}\vee\mathcal{P}_{\mathcal{S}}),$
	which is in $\mathcal{H}.$ (An equivalent description in terms of
	the left shift on strings of pairs of symbols is given in Subsection
	\ref{subsec:shift}.) The set $A$ used for inducing consists of the levels of the columns in the
	tower for $U$ that are labeled with $(\sigma,\upsilon)$, where $\upsilon=0.$ In this description, the sequence of $\upsilon$'s can be interpreted as a sequence of shadings, where the levels with $0$ are unshaded, while the levels with $1$ are shaded and thus ``invisible'' to the induced transformation $T_A.$
	\end{comment}
	
	\begin{proof}[Proof of Part 2 (c) of Theorem \ref{mainthm:Kproof}]
		The return time to $A$ is a geometrically distributed random variable,
		and according to Lemma 1 of \cite{OS77}, this implies that $U_{A}$
		is isomorphic to a skew product to which Meilijson's Theorem \cite{Me74} applies. Therefore
		$U_{A}$ is a $K$-automorphism. Actually the hypotheses in \cite{Me74} are
		not satisfied, but, as was observed in \cite{OS77}, the proof in \cite{Me74}
		can be generalized to cover the case at hand. In fact, the first proposition in 
		\cite{Ge97}, which is essentially contained in \cite{Me74} and \cite{Ho88}, implies exactly what we need.
	\end{proof}

	\begin{proof}[Proof of Part 2 (d) of Theorem \ref{mainthm:Kproof}]
		Suppose that $\mathcal{T}$ has an infinite branch. From \cite[section~5]{GK3} we know that $T=\Psi(\mathcal{T})$ and $T^{-1}=\Psi(\mathcal{T})^{-1}$ are isomorphic. Then by taking the product of this isomorphism and the identity on the $\{0,1\}^{\Z}$ factor, we obtain an isomorphism $\phi$ between $U=\Upsilon(\mathcal{T})$ and $U^{-1}=\Upsilon(\mathcal{T})^{-1}$ such that $\phi(A)=A$, where $A$ is as in~\eqref{eq:setA}.
	\end{proof}

	\section{Review of Construction of map $\Psi: \trees \to \mathcal{E}$} \label{sec:review}
	The proof of the anti-classification results in \cite{GK3} is based on the construction of a continuous map $\Psi: \trees \to \mathcal{E}$ such that for $\mathcal{T}\in\mathcal{T}\kern-.5mm rees$ and $T=\Psi(\mathcal{T})$:
	\begin{enumerate}
		\item If $\mathcal{T}$ has an infinite branch, then $T$ and $T^{-1}$
		are isomorphic.
		\item If $T$ and $T^{-1}$ are Kakutani equivalent, then $\mathcal{T}$
		has an infinite branch.
	\end{enumerate}
	In this section we collect important properties of the construction of $\Psi$ in sections 7--8 of \cite{GK3}.
	
	The transformation $T=\Psi(\mathcal{T})$ can be obtained from a cutting  and stacking construction on $[0,1]$ (as mentioned in Section~\ref{sec:KAut}) or from a  construction of a symbolic left shift $(\mathbb{K},\sh)$ that we will abbreviate by $\mathbb{K}$. When taking the second approach, for each $\mathcal{T}\in\mathcal{T}\kern-.5mm rees$, the symbolic system is built using a construction sequence $\Meng{\mathcal{W}_{n}\left(\mathcal{T}\right)}{\sigma_n \in \mathcal{T}}$ in a basic alphabet $\Sigma$ of cardinality $2^{12}$, where for each $n\in\mathbb{N}$ with $\sigma_{n}\in\mathcal{T}$ the set $\mathcal{W}_{n}=\mathcal{W}_{n}(\mathcal{T})$ of allowed $n$-words depends only on $\mathcal{T}\cap\left\{ \sigma_{m}:m\leq n\right\} $. Then
	\begin{equation*}
		\begin{split}
			\mathbb{K} = \Meng{(\sigma_j)_{j\in \Z} \in \Sigma^{\Z}}{\forall n \in \mathbb{Z}^+, (\sigma_j)_{j\in \Z} \text{ is a bi-infinite concatenation of $n$-words}}.
		\end{split}
	\end{equation*}
	
	The inverse $T^{-1}$ could be viewed as a right shift on $\mathbb{K}$. Instead we will always regard it as a left shift on $\rev(\mathbb{K})$. Here, we write $\rev(w)$ for the reverse string of a finite or infinite string $w$. In particular, if $x$ is in $\mathbb{K}$
	we define $\rev(x)$ by setting $\rev(x)(k)=x(-k)$. Then for a collection $W$ of words, $\rev(W)$ is the collection of reverses of words in $W$ and $\Meng{\rev\left(\mathcal{W}_{n}\left(\mathcal{T}\right)\right)}{\sigma_n \in \mathcal{T}}$ is a construction sequence for $\rev(\mathbb{K})$.
	
	The structure of the tree $\mathcal{T}\subset\mathbb{N}^{\mathbb{N}}$ is also used to build a sequence of groups $G_s(\mathcal{T})$ as follows. Here, the index $s$ indicates \emph{level} $s$ of $\mathcal{T}$, which consists of those elements of $\mathcal{T}$ of length $s$. We define $G_{0}(\mathcal{T})$ to be the trivial group and assign
	to each level $s>0$ a so-called \emph{group of involutions} 
	\[
	G_s(\mathcal{T})=\sum_{\tau\in\mathcal{T},\,lh(\tau)=s}(\mathbb{Z}_{2})_{\tau}.
	\]
	Each $(\mathbb{Z}_{2})_{\tau}$ is a copy of $\mathbb{Z}_2$, and we let $g_{\tau}$ denote its generator. We have a well-defined notion of \emph{parity}
	for elements in such a group of involutions: an element is called \emph{even} if
	it can be written as the sum of an even number of generators.
	Otherwise, it is called \emph{odd}. 
	
	For levels $0<s<t$ of $\mathcal{T}$ we have a canonical homomorphism
	$\rho_{t,s}:G_{t}\left(\mathcal{T}\right)\to G_{s}\left(\mathcal{T}\right)$
	that sends $g_{\tau}$ to $g_{\sigma}$ where $\sigma$ is the initial segment of $\tau$ with $lh(\sigma)=s$. The map $\rho_{t,0}$ is the trivial homomorphism $\rho_{t,0}:G_t(\mathcal{T})\to G_0(\mathcal{T})=\{0\}$. %We denote the inverse limit of $\left\langle G_{s}\left(\mathcal{T}\right),\rho_{t,s}:s<t\right\rangle $ by $G_{\infty}\left(\mathcal{T}\right)$ and we let $\rho_{s}:G_{\infty}\left(\mathcal{T}\right)\to G_{s}\left(\mathcal{T}\right)$ be the projection map.
	A tree $\mathcal{T} \in \trees$
	has an infinite branch if and only if there is an infinite sequence $\left(g_{s}\right)_{s\in\mathbb{Z}^+}$
	of generators $g_{s}\in G_{s}\left(\mathcal{T}\right)$ with $\rho_{t,s}\left(g_{t}\right)=g_{s}$
	for $t>s>0$.
	
	%Since there is a one-to-one correspondence between the infinite branches of $\mathcal{T}$ and infinite sequences $\left(g_{s}\right)_{s\in\mathbb{Z}^+}$ of generators $g_{s}\in G_{s}\left(\mathcal{T}\right)$ with $\rho_{t,s}\left(g_{t}\right)=g_{s}$ for $t>s>0$, we obtain that $G_{\infty}\left(\mathcal{T}\right)$ has a nonidentity element of odd parity if and only if $\mathcal{T} \in \trees$ has an infinite branch. 
	
	During the construction one uses the following finite approximations:
	For every $n\in \N$ we let $G_{0}^{n}\left(\mathcal{T}\right)$ be the trivial group and
	for $s>0$ we let 
	\[
	G_{s}^{n}\left(\mathcal{T}\right)=\sum\Meng{\left(\mathbb{Z}_{2}\right)_{\tau}}{\tau\in\mathcal{T}\cap\left\{ \sigma_{m}:m\leq n\right\} ,\,lh(\tau)=s}.
	\]
	We also introduce the finite approximations $\rho_{t,s}^{(n)}:G_{t}^{n}(\mathcal{T})\to G_{s}^{n}(\mathcal{T})$
	to the canonical homomorphisms.
	
	In the following, we simplify notation by reindexing $\Meng{\mathcal{W}_{n}\left(\mathcal{T}\right)}{\sigma_n \in \mathcal{T}}$ and $\Meng{G_{s}^{n}\left(\mathcal{T}\right)}{\sigma_n \in \mathcal{T}}$ as $\{\mathcal{W}_n\}_{n\in \mathbb{N}}$ and $\{G^n_s\}_{n\in \mathbb{N}}$, respectively.
	To describe their construction, let $\left(\epsilon_{n}\right)_{n\in\mathbb{N}}$ be a decreasing sequence of positive numbers
	such that
	\begin{equation}
		\sum_{n\in\mathbb{N}}\epsilon_{n} < \infty.\label{eq:Condeps}
	\end{equation}
	During the course of construction one also defines a fast growing sequence of positive integers $\left(e(n)\right)_{n\in \N}$ and an increasing sequence of prime numbers $(\mathfrak{p}_n)_{n\in \N}$ (see equations~(8.5) and (8.7) in \cite{GK3}). 
	
	We now collect important properties of the construction sequence $\{\mathcal{W}_n\}_{n\in \mathbb{N}}$. We start by setting $\mathcal{W}_{0}=\Sigma$. 
	\begin{enumerate}[label=(E\arabic*)]
		\item\label{item:E1}  All words in $\mathcal{W}_{n}$ have the same length $h_{n}$ and
		the cardinality $|\mathcal{W}_{n}|$ is a power of $2$.
		\item\label{item:E2}  There are $f_{n},k_n\in\mathbb{Z}^{+}$ such that every word in $\mathcal{W}_{n+1}$ is built by concatenating $k_n$ words
		in $\mathcal{W}_{n}$ and such that
		every word in $\mathcal{W}_{n}$ occurs in each word of $\mathcal{W}_{n+1}$
		exactly $f_{n}$ times. The number $f_{n}$ is a product of $\mathfrak{p}_{n}^{2}$
		and powers of $2$. Clearly, we have $k_n=f_n|\mathcal{W}_n|$.
		\item\label{item:E3}  If $w=w_{1}\dots w_{k_n}\in \mathcal{W}_{n+1}$ and $w^{\prime}=w_{1}^{\prime}\dots w_{k_n}^{\prime}\in\mathcal{W}_{n+1}$, where $w_{i},w_{i}^{\prime}\in\mathcal{W}_{n}$, then for any $k\geq\lfloor\frac{k_n}{2}\rfloor$ and $1\leq i \leq k_n-k$, we have $w_{i+1}\dots w_{i+k}\neq w_{1}^{\prime}\dots w_{k}^{\prime}$. 
	\end{enumerate}
	\begin{rem*}
		In particular, these specifications say that $\{\mathcal{W}_n\}_{n\in \mathbb{N}}$
		is a uniquely readable and strongly uniform construction sequence
		as defined in \cite[section 3.3]{FW1}. Hence, the corresponding symbolic system $\mathbb{K}$ has a unique non-atomic ergodic shift-invariant measure by \cite[Lemma 11]{FW1}.
	\end{rem*}
	
	For each $s \leq s(n)$, there is an equivalence relation $\mathcal{Q}_{s}^{n}$ on $\mathcal{W}_{n}$ satisfying the following specifications. To start, we let $\mathcal{Q}_{0}^{0}$ be the equivalence relation
	on $\mathcal{W}_{0}=\Sigma$ which has one equivalence class, that is,
	any two elements of $\Sigma$ are equivalent. 
	\begin{enumerate}[label=(Q\arabic*)]
		\setcounter{enumi}{3}
		\item\label{item:Q4} Suppose that $n=M(s)$ for some $s \in \Z^+$. There is a specific number $J_{s(n),n}\in\mathbb{Z}^{+}$
		such that each word in $w_{n}\in\mathcal{W}_{n}$ is a concatenation
		$w_{n}=w_{n,1}\dots w_{n,J}$ of $J=J_{s(n),n}$ strings of equal
		length, where each $w_{n,i}$ is a concatenation of $(n-1)$-words. Then any two words in the same $\mathcal{Q}_{s}^{n}$ class
		agree with each other except possibly on initial or final strings
		of length at most $\frac{\epsilon_{n}}{2}\frac{h_{n}}{J_{s(n),n}}$
		on the segments $w_{n,i}$ for $i=1,\dots,J_{s(n),n}$.
		\item\label{item:Q5}  For $n\geq M(s)+1$ we can consider words in $\mathcal{W}_{n}$
		as concatenations of words from $\mathcal{W}_{M(s)}$ and define $\mathcal{Q}_{s}^{n}$
		as the product equivalence relation of $\mathcal{Q}_{s}^{M(s)}$.\footnote{Given an equivalence relation $\mathcal{Q}$ on a set $X$ we define the \emph{product equivalence relation} $\mathcal{Q}^{n}$ on $X^{n}$ by setting $x_{0}\dots x_{n-1}\sim x_{0}^{\prime}\dots x_{n-1}^{\prime}$ if and only if $x_{i}\sim x_{i}^{\prime}$ with respect to $\mathcal{Q}$ for all $i=0,\dots,n-1$.} 
		\item\label{item:Q6} $\mathcal{Q}_{s+1}^{n}$ refines $\mathcal{Q}_{s}^{n}$ and each
		$\mathcal{Q}_{s}^{n}$ class contains $2^{4e(n)}$ many $\mathcal{Q}_{s+1}^{n}$
		classes.\footnote{Given two equivalence relations $\mathcal{Q}$ and $\mathcal{R}$ on a set $X$ we say that $\mathcal{R}$ \emph{refines} $\mathcal{Q}$ if considered as sets of ordered pairs we have $\mathcal{R}\subseteq\mathcal{Q}$.}
	\end{enumerate}
	
	Each equivalence relation $\mathcal{Q}_{s}^{n}$ will induce an equivalence
	relation on $\rev(\mathcal{W}_{n})$, which we will also call $\mathcal{Q}_{s}^{n}$,
	as follows: $\rev(w),\rev(w')\in \rev(\mathcal{W}_{n})$ are equivalent
	with respect to $\mathcal{Q}_{s}^{n}$ if and only if $w,w'\in\mathcal{W}_{n}$
	are equivalent with respect to $\mathcal{Q}_{s}^{n}$. 
	
	\begin{rem*}
		In case that the exponent is not relevant we will refer to the $\mathcal{Q}_{s}^{n}$
		as $\mathcal{Q}_{s}$. For $u\in\mathcal{W}_{n}$ we write $[u]_{s}$
		for its $\mathcal{Q}_{s}^{n}$ class.
	\end{rem*}
	
	In \cite[section 5]{FRW} and \cite[section 4.2.2]{GK3} the equivalence relations $\mathcal{Q}_{s}$ are used to define a canonical sequence
	of factors $\mathbb{K}_s$ and $\rev(\mathbb{K}_s)$ with factor maps $\pi_s: \mathbb{K} \to \mathbb{K}_s$ and $\pi_{s',s}: \mathbb{K}_{s'} \to \mathbb{K}_s$ for all $s'>s$.
	
	We now list specifications on actions by the groups of involutions $G^n_s$.
	\begin{enumerate}[label=(A\arabic*)]
		\setcounter{enumi}{6}
		\item\label{item:A7}  $G_{s}^{n}$ acts freely on $\mathcal{W}_{n}/\mathcal{Q}_{s}^{n}$
		and the $G_{s}^{n}$ action is subordinate\footnote{Suppose
			$\mathcal{Q}$ and $\mathcal{R}$ are equivalence relations on a set
			$X$ with $\mathcal{R}$ refining $\mathcal{Q}$, $G$ and $H$ are groups with $G$ acting on $X/\mathcal{Q}$ and
			$H$ acting on $X/\mathcal{R}$, $\rho:H\to G$ is a homomorphism.
			Then we say that the $H$ action is \emph{subordinate} to the $G$
			action if for all $x\in X$, whenever $[x]_{\mathcal{R}}\subset[x]_{\mathcal{Q}}$
			we have $h[x]_{\mathcal{R}}\subset\rho(h)[x]_{\mathcal{Q}}$.} to the $G_{s-1}^{n}$ action
		on $\mathcal{W}_{n}/\mathcal{Q}_{s-1}^{n}$ via the canonical homomorphism
		$\rho_{s,s-1}^{(n)}:G_{s}^{n}\to G_{s-1}^{n}$.
		\item\label{item:A8} Suppose $n>M(s)$. We view $G_{s}^{n}=G_{s}^{n-1}\oplus H$.
		Then the action of $G_{s}^{n-1}$ on $\mathcal{W}_{n-1}/\mathcal{Q}_{s}^{n-1}$
		is extended to an action on $\mathcal{W}_{n}/\mathcal{Q}_{s}^{n}$
		by the skew diagonal action.\footnote{If $G$ is a group of involutions with a collection
			of canonical generators that acts on $X$, then the canonical \emph{skew diagonal action} of
			$G$ on $X^{n}$ is defined by 
			\[
			g\left(x_{0}x_{1}\dots x_{n-1}\right)= \begin{cases}
				gx_{0}\,gx_{1}\dots gx_{n-1}, & \text{ if $g \in G$ is of even parity,} \\
				gx_{n-1}\,gx_{n-2}\dots gx_{0}, & \text{ if $g\in G$ is of odd parity.}
			\end{cases}
			\] } In particular, $\mathcal{W}_{n}/\mathcal{Q}_{s}^{n}$
		is closed under the skew diagonal action by $G_{s}^{n-1}$.
	\end{enumerate}

	\begin{rem}\label{rem:BuildingIso}
		Let $s\in\mathbb{Z}^+$ and $g\in G_{s}^{m}$ for some $m\geq M(s)$. Suppose that $g$ has odd parity. Since $g$ acts on $\mathcal{W}_{n}/\mathcal{Q}_{s}^{n}$ by the skew-diagonal action by \ref{item:A8}, it induces  for all $n>m$ a map $\eta_g:\mathcal{W}_n/\mathcal{Q}^n_s \to \rev(\mathcal{W}_n)/\mathcal{Q}^n_s$ by the diagonal action. Then $g$ yields a shift-equivariant isomorphism $\mathbb{K}_{s}\to \rev(\mathbb{K}_{s})$ that we denote by $\eta_g$ as well (see \cite[Lemma 37]{GK3}). 
	\end{rem}
	
	In case that $\mathcal{T}\in\mathcal{T}\kern-.5mm rees$ has an infinite branch, then there is an infinite sequence $\left(g_{s}\right)_{s\in\mathbb{Z}^+}$
	of canonical generators $g_{s}\in G_{s}\left(\mathcal{T}\right)$ with $\rho_{t,s}\left(g_{t}\right)=g_{s}$
	for $t>s>0$. Hereby we obtain a coherent sequence of isomorphisms $\eta_{g_s}$ between
	$\mathbb{K}_{s}$ and $\rev(\mathbb{K}_{s})$ satisfying $\pi_{s+1,s}\circ\eta_{g_{s+1}}=\eta_{g_{s}}\circ\pi_{s+1,s}$ for every $s \in \N$ by \ref{item:A7}. This yields an isomorphism between $\mathbb{K}$ and $\rev(\mathbb{K})\cong \mathbb{K}^{-1}$ (see \cite[section 5]{GK3} for details).
	
	To exclude Kakutani equivalence between $\mathbb{K}$ and $\rev(\mathbb{K})$ for $\mathcal{T} \in \trees$ without an infinite branch, $n$-words are constructed by substituting Feldman patterns of $s$-equivalence classes of $(n-1)$-words into Feldman patterns of $(s-1)$-equivalence classes of $(n-1)$-words. In this iteration of substitution steps one determines parameters $T_s, M_s$ recursively and proceeds as follows: 
	\begin{enumerate}
		\item\label{item:step1} We divide the $2^{4e(n-1)}$ many $s$-classes contained in any $[A]_{s-1} \in \mathcal{W}_{n-1}/\mathcal{Q}^{n-1}_{s-1}$  into $2^{t_s}$ tuples of the form $\left([A_{1}]_{s},[A_{2}]_{s},\dots,[A_{N_{s+1}}]_{s}\right)$ such that each tuple intersects each $\ker(\rho_{s,s-1}^{(n-1)})$ orbit exactly once and the tuples are images of each other under the action by $\ker(\rho_{s,s-1}^{(n-1)})$. Here, $N_{s+1}=2^{4e(n-1)-t_s}$ with $t_{s}\in\mathbb{N}$
		such that 
		\[
		2^{t_{s}}=|\ker(\rho_{s,s-1}^{(n-1)})|. 
		\]
		We note that for $i\neq j$ the $G^{n-1}_s$ orbit of $[A_i]_s$ does not intersect the $G^{n-1}_s$ orbit of $[A_j]_s$.
		%In the construction of $\mathcal{W}_{n}/\mathcal{Q}^{n}_{s}$ this subdivision into tuples transversal to the $\ker(\rho_{s,s-1}^{(n-1)})$ action will be used to guarantee the $\overline{f}$ distance of elements in $\mathcal{W}_{n}/\mathcal{Q}^{n}_{s}$ that are images of each other under the skew diagonal action by $\ker(\rho_{s,s-1}^{(n-1)})$. Since those elements have the same $s$-pattern structure, we need the building blocks from different tuples to be $\overline{f}$-apart.
		\item\label{item:step2} We choose a set $\Gamma_{s-1}\subset\mathcal{W}_{n}/\mathcal{Q}^{n}_{s-1}$ that intersects
		each orbit of the skew-diagonal action by the group $\im(\rho_{s,s-1}^{(n-1)})$ exactly once.
		\item\label{item:step3} We construct a collection
		of $M_{s+1}$ different $\left(T_{s+1},N_{s+1},M_{s+1}\right)$-Feldman
		patterns, where the tuple of building blocks is to be determined in
		step (\ref{item:step6}). 
		\item\label{item:step4} By the induction assumption of the substitution step, each $[w]_{s-1} \in\Gamma_{s-1}$  can be decomposed into a concatenation of different $(T_s,N_s,M_s)$-Feldman patterns.
		\item\label{item:step5} Let $K_s \in \Z^+$ such that $K_{s}|\ker(\rho_{s,s-1}^{(n-1)})|=2^{4e(n)}$. For each $[w]_{s-1} \in\Gamma_{s-1}$ we choose $K_s$ many different concatenations of $\left(T_{s+1},N_{s+1},M_{s+1}\right)$-Feldman
		patterns. We enumerate these choices by $\left\{ 1,\dots,K_s\right\}$.
		\item\label{item:step6} Let $[w]_{s-1} \in\Gamma_{s-1}$ and $j\in\left\{ 1,\dots,K_s\right\} $. We decompose $[w]_{s-1}$ into repetitions $[A]^{T_s}_{s-1}$ of single classes. For the $i$-th such repetition we substitute the $i$-th Feldman pattern occurring in $j$ using one of the subordinate building tuples from step (\ref{item:step1}). In fact, the actual tuple is chosen by cycling through the possible ones. Let $\Omega^{\prime}_s$ be the collection of such  strings obtained for all $[w]_{s-1} \in\Gamma_{s-1}$ and $j\in\{1,\dots,K_s\}$.
		\item\label{item:step7} Then we let $\mathcal{W}_{n}/\mathcal{Q}^{n}_s= G^{n-1}_s\Omega^{\prime}_s$. In particular, $\mathcal{W}_{n}/\mathcal{Q}^{n}_s$ is closed under the skew-diagonal action by $G^{n-1}_s$ and each $[w]_{s-1}\in \mathcal{W}_{n}/\mathcal{Q}^{n}_{s-1}$ contains $2^{4e(n)}$ many $\mathcal{W}_n/\mathcal{Q}^n_s$ classes by the choice of $K_s$ in step (\ref{item:step5}). 
		\item\label{item:step8} By (\ref{item:step6}) and (\ref{item:step7}), each element in $\mathcal{W}_{n}/\mathcal{Q}^{n}_s$ is a concatenation of different $(T_{s+1}, N_{s+1}, M_{s+1})$-Feldman patterns, the so-called $s$-Feldman patterns. We will call the sequence of $s$-pattern types that appear in an element of $\mathcal{W}_n/\mathcal{Q}^n_s$ its $s$-pattern structure. We note that a particular type of $(T_{s+1}, N_{s+1}, M_{s+1})$-Feldman pattern can occur in different elements  $[w]_s$, $[w']_s \in \mathcal{W}_{n}/\mathcal{Q}^{n}_s$ only if $[w]_s=g[w']_s$ for some nontrivial $g\in G^{n-1}_s$. In this case the building tuples of that particular type of $(T_{s+1}, N_{s+1}, M_{s+1})$-Feldman pattern are disjoint from each other by step~(\ref{item:step1}). 
	\end{enumerate}
	
	\begin{comment}
	In our construction of $s$-classes of $n$-words we always concatenate different $(T_{s+1}, N_{s+1}, M_{s+1})$-Feldman-patterns, the so-called $s$-patterns. Here, $N_{s+1}=2^{4e(n-1)-t_s}$ with $t_{s}\in\mathbb{N}$
	such that 
	\[
	2^{t_{s}}=|ker(\rho_{s,s-1}^{(n)})|.
	\]
	Moreover, the $s$-classes of $(n-1)$-words within a building tuple of an $s$-pattern have disjoint $G^{n-1}_s$-orbits. 
	\end{comment}
	
	By this construction in \cite[sections 7-8]{GK3} we obtain the following statement on $\overline{f}$-distances.
	
	\begin{lem}[\cite{GK3}, Proposition 56] \label{lem:fDist}
		For every $s \in \mathbb{Z}^+$ there is $0<\alpha_s <\frac{1}{8}$ such that for every $n\geq M(s)$ we have
		\begin{equation}
			\overline{f}\left(W,W^{\prime}\right)>\alpha_{s}
		\end{equation}
		on any substrings $W$ and $W^{\prime}$ of at least $\frac{h_{n}}{R_{n}}$
		consecutive symbols in any words $w,w^{\prime}\in\mathcal{W}_{n}$
		with $[w]_{s}\neq[w^{\prime}]_{s}$.
	\end{lem}
	
	The sequence $(\alpha_s)_{s\in \N}$ is decreasing. Here, the numbers $R_n$ are determined inductively \cite[eq.~(8.6)]{GK3}. Even though $R_n\to \infty$, strings of length $h_n/R_n$ are long enough to contain many Feldman patterns of equivalence classes of $(n-1)$-words.
	
	Lemma~\ref{lem:fDist} also allows us to show in the following lemma that a substantial string of consecutive symbols from an $s$-pattern with the building tuple traversed in reverse order, is $\overline{f}$-apart from any string in $\mathbb{K}$. We use this statement in the proof of Lemma \ref{lem:DifferentPatterns}, where we adapt methods from \cite[chapter 13]{ORW} to exclude the existence of good $\overline{f}$-codes between Feldman patterns of different types.
	\begin{lem}\label{lem:ReverseOrder}
		Suppose that the $(T_{s+1}, N_{s+1}, M_{s+1})$-Feldman pattern of type $r$ appears with building blocks $\left([A_{1}]_{s},[A_{2}]_{s},\dots,[A_{N_{s+1}}]_{s}\right)$ in some $[w]_s \in \mathcal{W}_n/\mathcal{Q}^n_s$. Let $F$ be a finite string of symbols in $\Sigma$ of length $|F|\geq 2^{e(n-1)}T_{s+1}N^{2M_{s+1}+2}_{s+1}h_{n-1}$ such that $F$ is part of the $(T_{s+1}, N_{s+1}, M_{s+1})$-Feldman pattern of type $r$ with building blocks $\left([A_{N_{s+1}}]_{s},\dots,[A_{1}]_{s}\right)$ (that is, the building tuple is traversed in opposite order). Then 
		\begin{equation}\label{eq:ReverseOrder}
			\overline{f}(F,D)>\frac{\alpha_s}{2}
		\end{equation}
		for any finite string $D$ in $\mathbb{K}$. 
	\end{lem}
	
	\begin{proof}
		Assume that there is a string $D$ in $\mathbb{K}$ with $\overline{f}(F,D)\leq \frac{\alpha_s}{2}$ and let $\pi$ be a best possible match between $F$ and $D$. Then 
		\begin{equation} \label{eq:Lengths}
			1-3\frac{\alpha_s}{2}<\frac{1-\frac{\alpha_s}{2}}{1+\frac{\alpha_s}{2}} \leq \frac{|D|}{|F|}\leq \frac{1+\frac{\alpha_s}{2}}{1-\frac{\alpha_s}{2}} < 1+3\frac{\alpha_s}{2}
		\end{equation}
		by Fact \ref{fact:string_length}. In particular, $D$ has to lie within at most three $s$-patterns and we decompose $D$ into $D=D_1D_2D_3$, where the substrings $D_i$ belong to different $(T_{s+1}, N_{s+1}, M_{s+1})$-patterns. Furthermore, for $i=1,2,3,$ let $F_i$ denote the substring of $F$ matched to $D_i$ under $\pi$. (Some of the substrings $D_i,F_i$ could be empty.)
		
		If $\overline{f}(F_i,D_i)>\frac{\alpha_s}{2}$, then we stop examining this case.
		
		Otherwise, we have $\frac{1}{2}<1-3\frac{\alpha_s}{2}<\frac{|D_i|}{|F_i|} < 1+3\frac{\alpha_s}{2}<2$ by the same estimate as in \eqref{eq:Lengths}. If $|D_i|<2T_{s+1}N^{2M_{s+1}+2}_{s+1}h_{n-1}$, then we ignore the strings $D_i$ and $F_i$ in the following considerations. This might increase the $\overline{f}$-distance by at most $8/2^{e(n-1)}$ by Fact \ref{fact:omit_symbols}. If $|D_i|\geq 2T_{s+1}N^{2M_{s+1}+2}_{s+1}h_{n-1}$, then $|F_i|\geq T_{s+1}N^{2M_{s+1}+2}_{s+1}h_{n-1}$. In case that the $(T_{s+1}, N_{s+1}, M_{s+1})$-Feldman pattern types of $D_i$ and $F_i$ are different from each other, these lengths allow the application of Lemma \ref{lem:Feldman} %\cite[Proposition 41]{GK3} 
		and we obtain 
		\[
		\overline{f}(F_i,D_i)>\alpha_s - \frac{4}{2^{2e(n-1)-0.5t_s}}-\frac{1}{R_{n-1}}.
		\]
		In the next step, we investigate the case that the $(T_{s+1}, N_{s+1}, M_{s+1})$-Feldman pattern types of $D_i$ and $F_i$ agree with each other, that is, the type is $r$. We recall from step (\ref{item:step8}) in the substitution process that in the construction of $\mathcal{W}_n/\mathcal{Q}^n_s$ the only other occurrences of that Feldman pattern of type $r$ come with other tuples of building blocks that are disjoint from $\left([A_{1}]_{s},[A_{2}]_{s},\dots,[A_{N_{s+1}}]_{s}\right)$. In both possible scenarios (either disjoint tuples of building blocks or the building blocks are traversed in opposite directions), we can apply the symbol by block replacement lemma in Lemma \ref{lem:symbol by block replacement} %\cite[Lemma 14]{GK3} 
		and obtain
		\[
		\overline{f}(F_i,D_i)>\alpha_s \cdot \left(1-\frac{1}{2^{4e(n-1)-t_s}}\right)-\frac{1}{R_{n-1}}. %-\frac{10}{2^{e(n-1)}}
		\] 
		Altogether, we conclude \eqref{eq:ReverseOrder} by Fact \ref{fact:substring_matching}. 
	\end{proof}
	
	The repetitions of $s$-classes in $s$-Feldman patterns and the choices of building tuples in step (\ref{item:step6}) guarantee the following statement on uniformity of occurrences of substitution instances (see \cite[Lemma 46]{GK3} and \cite[Remark 52]{GK3}).
	
	\begin{rem}
		\label{rem:Occurrence-Substitutions}Let $w\in\mathcal{W}_{n}$ and
		$[P_{n-1}]_{s}$ be a $\left(T_{s+1},N_{s+1},M_{s+1}\right)$-Feldman
		pattern in $[w]_{s}$ with building blocks from the tuple $\left([A_{1}]_{s},\dots,[A_{N_{s+1}}]_{s}\right)$
		in $\mathcal{W}_{n-1}/\mathcal{Q}_{s}^{n-1}$. Then for every $i=1,\dots,N_{s+1}$ each $a_{i,j} \in \mathcal{W}_{n-1}$
		with $[a_{i,j}]_s=[A_{i}]_{s}$ is substituted the same number of times into each occurrence of $[A_i]_s^{T_{s+1}\cdot2^{t_{s+1}}}$ in $[P_{n-1}]_s$ within $w$.
	\end{rem}
	
	We end our review of the constructions in \cite{GK3} with the observation that at each stage we can choose the number of concatenations of different Feldman-patterns arbitrarily large. This allows us to produce sufficiently many substitution instances of each $n$-word to guarantee requirement \ref{item:R9} below.

	\section{Completion of Proofs of Theorems \ref{mainthm:KAut} and \ref{mainthm:Kproof}} \label{sec:ProofKAut}

	\subsection{Shift space description of the process} \label{subsec:shift}
	As described in Section \ref{sec:review}, the transformation $T$ was constructed as a symbolic system $\mathbb{K}$ using a construction sequence $\{\mathcal{W}_n\}_{n\in \mathbb{N}}$. Given this construction sequence, the transformation $U=T \times \mathbb{B}=\Upsilon(\mathcal{T})$ is isomorphic to the left shift,
	$\sh$, on a sequence space that we now describe. Let
	\[
	\begin{array}{ccl}
		\mathbb{S} & = & \{(\sigma_{j},\upsilon_{j})_{j\in\mathbb{Z}} \in (\Sigma\times\{0,1\})^{\mathbb{Z}}\, :\text{ for\ every\ }n\in\mathbb{Z}^{+},(\sigma_{j})_{j\in\mathbb{Z}}
		%\text{ is a\ sequence }\\
		\text{ is a bi-infinite }\\
		&  &  \text{\ \ \ \ \ \ \ \ \ \ \ \ \ \ \ \ \ \ \ \ \ \ \ \ \ \ \ \ \ \ \ \ \ \ \ \ \ \ \ \ concatenation of \ensuremath{n}-words from $\mathcal{W}_n$}\},
		%and \ensuremath{(\upsilon_{j})_{j\in\mathbb{Z}}} is a sequence of  \ensuremath{0}'s and %\ensuremath{1}'s}\},
	\end{array}
	\]
	with the topology induced from the product topology on $(\Sigma\times\{0,1\})^{\mathbb{Z}}$. We abbreviate $(\mathbb{S},\sh)$ by $\mathbb{S}$.
	Define for each $n \in \N$ the set 
	\[
	\mathcal{V}_n = \Meng{(w,\upsilon)}{w\in \mathcal{W}_n \text{ and } \upsilon \in \{0,1\}^{h_n}}.
	\]
	Since the construction sequence $\{\mathcal{W}_n\}_{n\in \mathbb{N}}$ is uniquely readable and strongly uniform, the symbolic system $\mathbb{K}$ has a unique non-atomic shift-invariant measure $\nu'$. By putting
	\[
	\nu\left(\langle(w,\upsilon)\rangle\right) = \frac{\nu'(\langle w\rangle)}{2^{h_n}}
	\]
	for each cylinder set $\langle (w,\upsilon)\rangle$ with $(w,\upsilon)\in \mathcal{V}_n$, we define a
	$\sh$-invariant measure $\nu$ on the Borel subsets of $\mathbb{S}$ which is the same as the product of $\nu'$ and the $(1/2,1/2)$-Bernoulli measure.
	
	We will often refer to the sequences $\upsilon \in \{0,1\}^{h_n}$ as \emph{shadings} of $n$-words. That is, for the pair $(w,\upsilon)\in \mathcal{V}_n$ we regard the $w_j$ entry in $w$ as \emph{shaded} if $\upsilon_j=1$ and \emph{unshaded} if $\upsilon_j=0$.

	As mentioned at the end of Section \ref{sec:review} we can concatenate sufficiently many Feldman patterns to produce a large number $f_n$ of occurrences of each $n$-word $w_i$ in each $(n+1)$-word $w$. Then the law of large numbers allows us to obtain the following ``almost uniformity'' of shadings of $w_i$ within shadings of $w$.
	\begin{enumerate}[label=(R\arabic*)]
		\setcounter{enumi}{8}
		\item\label{item:R9} For $w \in \mathcal{W}_{n+1}$ a proportion of at least $\frac{999}{1000}$ of  the shadings $\upsilon \in \{0,1\}^{h_{n+1}}$ are such that for all $w_i \in \mathcal{W}_n$ and $\upsilon_j \in \{0,1\}^{h_n}$ the number $R_{i,j}$ of times the shading $\upsilon_j$ appears with $w_i$ satisfies
		\begin{equation} \label{eq:EquiShading}
			\abs{\frac{R_{i,j}}{f_n}-\frac{1}{2^{h_n}}} < \frac{1}{2\cdot 2^{h_n}}.
		\end{equation} 
	\end{enumerate}

	\subsection{Proof of Non-Kakutani Equivalence} \label{subsec:Non-Equiv}
	\begin{comment}
	\begin{prop}\label{prop:Nonequiv}
	If we let $\mathcal{P}_{\Sigma}$
	denote the partition according to the collection $\Sigma$ of symbols
	used in \cite{GK3}, and we let $\mathcal{P}_{\mathcal{S}}$ be the
	partition into shaded and unshaded parts, then the processes $(T,\mathcal{P}_{\Sigma}\vee\mathcal{P}_{\mathcal{S}})$
	and $(T^{-1},\mathcal{P}_{\Sigma}\vee\mathcal{P}_{\mathcal{S}})$
	(without discarding the shaded parts) are not Kakutani equivalent
	if $T=\Psi(\mathcal{T})$ for $\mathcal{T}$ without an infinite branch.
	\end{prop}If $\mathbb{S}\cong (\Psi(\mathcal{T}),\mathcal{P}_{\Sigma}\vee\mathcal{P}_{\mathcal{S}})$
	and $\mathbb{S}^{-1}\cong (\Psi(\mathcal{T})^{-1},\mathcal{P}_{\Sigma}\vee\mathcal{P}_{\mathcal{S}})$ are Kakutani equivalent, then it can
	only be by an even equivalence. 
	\end{comment}
	
	Since $\mathbb{S}\cong \Upsilon(\mathcal{T})$
	has positive entropy, any Kakutani equivalence
	between $\mathbb{S}$ and
	$\mathbb{S}^{-1}\cong \rev(\mathbb{S})$ would
	have to be an even equivalence.\footnote{Two transformations $S:(X,\mu)\to(X,\mu)$ and
		$T:(Y,\nu)\to(Y,\nu)$ are said to be \emph{evenly equivalent} if
		there are subsets $A\subseteq X$ and $B\subseteq Y$ of equal measure
		$\mu(A)=\nu(B)>0$ such that $S_{A}$ and $T_{B}$ are isomorphic
		to each other.} Thus, to prove part (b) in Theorem \ref{mainthm:Kproof} it remains to show that under the assumption
	that the tree $\mathcal{T}\in\mathcal{T}\kern-.5mm rees$ does not have an infinite
	branch, $\mathbb{S}$ is not evenly equivalent to
	$\rev(\mathbb{S})$. A generalization of Proposition 3.2 in \cite[p.~92]{ORW} given in \cite[section 9.2]{GK3} guarantees the existence of a consistent sequence of well-approximating finite codes in $\overline{f}$ between two evenly equivalent transformations (see properties \ref{item:C1} and \ref{item:C2} below for details). To exclude the existence of such a sequence of finite codes from $\mathbb{S}$ to $\rev(\mathbb{S})$ in case $\mathcal{T}$ does not have an infinite branch, we use the method of \cite[chapter 13]{ORW} and properties of the construction from \cite{GK3}. This generalizes section 9.3 of \cite{GK3}. 
	
	Here, a \emph{code} of length $2K+1$ is a function $\phi:(\Sigma\times \{0,1\})^{\mathbb{Z}\cap[-K,K]}\to\Sigma\times \{0,1\}$. Given such a code $\phi$ of length $2K+1$ the \emph{stationary code}
	$\bar{\phi}$ on $(\Sigma\times \{0,1\})^{\mathbb{Z}}$ determined by $\phi$ is defined
	as $\bar{\phi}(s)$ for any $s\in(\Sigma\times \{0,1\})^{\mathbb{Z}}$, where for
	any $l\in\mathbb{Z}$,
	\[
	\bar{\phi}(s)(l)=\phi\left(s\upharpoonright[l-K,l+K]\right).
	\]
	In this setting, $\bar{\phi}(s)\upharpoonright[-N,N]$ denotes the
	string of symbols
	\[
	\bar{\phi}(s)(-N)\,\bar{\phi}(s)(-N+1)\,\dots\,\bar{\phi}(s)(N-1)\,\bar{\phi}(s)(N)
	\]
	in $(\Sigma\times \{0,1\})^{2N+1}$.
	
	\begin{rem*}%\label{rem:end-effects}
		There is an ambiguity in applying a code $\phi$
		of length $2K+1$ to blocks of the form $s\upharpoonright[a,b]$:
		It does not make sense to apply it to the initial string $s\upharpoonright[a,a+K-1]$
		or the end string $s\upharpoonright[b-K+1,b]$. However, if $b-a$ is
		large with respect to the code length $2K+1$, we can fill in $\phi(s)\upharpoonright[a,a+K-1]$
		and $\phi(s)\upharpoonright[b-K+1,b]$ arbitrarily and it makes a
		negligible difference to the $\overline{f}$ distance.
		We refer to the general phenomenon
		of ambiguity or disagreement at the beginning and end of strings
		as \emph{end effects}. Strings are always chosen long enough such that end effects can be neglected in our $\overline{f}$ estimates.
	\end{rem*}
	
	We show in Lemma \ref{lem:groupelement} that a well-approximating finite code $\phi$ can be identified with an element of the group action. A key step in the proof is Lemma \ref{lem:DifferentPatterns}. Here, we show for substantial substrings $E$ and $\overline{E}$ of $s$-Feldman patterns $P$ and $\overline{P}$ of different types that there is a one-to-one correspondence between shadings $\upsilon \in \{0,1\}^{|E|}$ such that $\overline{f}\left(\phi(E,\upsilon),(\overline{E},\overline{\upsilon})\right)$ is small and shadings $\upsilon^{\prime} \in \{0,1\}^{|E|}$ such that $\phi(E,\upsilon^{\prime})$ is $\overline{f}$-apart from any substring in $\rev(\mathbb{S})$. This says that the code cannot work well on at least half of the shadings. To show the existence of such a shading $\upsilon^{\prime}$ we decompose $\overline{E}$ into complete cycles $\overline{C}_{k}$ of its $s$-Feldman pattern $\overline{P}$ and look at the string $C_{k}$ with its associated shading $\upsilon_{k}$ whose image under the code $\phi$ is matched to such a cycle $\overline{C}_{ k}$ under a best possible $\overline{f}$-match. We then build a reshuffling  $\upsilon^{\prime}_{ k}$ of the shading $\upsilon_{k}$ such that $\phi( C_{ k} , \upsilon^{\prime}_{ k} )$  is $\overline{f}$-close to a cycle where the $s$-classes of words are traversed in the opposite direction. Such a string is $\overline{f}$-apart from any string in $\rev(\mathbb{S})$ by Lemma~\ref{lem:ReverseOrder}. To build such a reshuffling it is crucial that all substitution instances of the $s$-classes in the complete cycle occur the same number of times by Remark~\ref{rem:Occurrence-Substitutions}. We visualize the building of reshufflings in Figures~\ref{fig:fig1} and \ref{fig:fig2}.%\cite[Lemma 46]{GK3}.

	\begin{lem}\label{lem:DifferentPatterns}
		Let $E$ and $\overline{E}$ be strings of symbols in $\Sigma$ that are substitution instances of $s$-Feldman patterns $P$ and $\overline{P}$ of $\mathcal{Q}_s^{n-1}$ equivalence classes, respectively, where $n-1\ge M(s)$. Assume $P$ and $\overline{P}$ are different types of $s$-Feldman patterns, and $E$ and $\overline{E}$ both have length at least 
		$\frac{|P|}{2^{2e(n-1)}}$. Suppose $\phi$ is a code of length $K$, where $h_{n-1}$ is large compared to $K$. Then there is a one-to-one map from shadings $\upsilon$ in $\{0,1\}^{|E|}$ such that 
		\[
		\overline{f}\left(\phi(E,\upsilon),(\overline{E},\overline{\upsilon})\right)<\frac{\alpha^2_s}{2000}
		\]
		for \emph{some} $\overline{\upsilon}\in \{0,1\}^{|\overline{E}|}$ to shadings $\upsilon^{\prime} \in \{0,1\}^{|E|}$ such that 
		\[
		\overline{f}\left(\phi(E,\upsilon^{\prime}),(\overline{D},\overline{u})\right)>\frac{\alpha_s}{20} 
		\]
		for \emph{all} strings $(\overline{D},\overline{u})$ in $\rev(\mathbb{S})$.
	\end{lem}
	
	\begin{comment}
	Let $E$ and $\overline{E}$ lie in $s$-Feldman patterns $P$ and $\overline{P}$ of different type of $(n-1)$-words and both have length at least $\frac{|P|}{2^{2e(n-1)}}$. Then there is a one-to-one map from shadings $\upsilon \in \{0,1\}^{|E|}$ with $\overline{f}\left(\phi(E,\upsilon),\overline{E}\right)<\frac{\alpha^2_s}{2000}$ to shadings $\upsilon^{\prime} \in \{0,1\}^{|E|}$ such that 
	\[
	\overline{f}\left(\phi(E,\upsilon^{\prime}),\overline{D}\right)> \frac{\alpha_s}{20}
	\]
	for all strings $\overline{D}$ in $\rev(\mathbb{S})$.
	\end{comment}

	\begin{proof}
		Recall from the construction in \cite{GK3} that the $s$-Feldman patterns are $(T_{s+1}, N_{s+1}, M_{s+1})$-Feldman patterns with $N_{s+1}=2^{4e(n-1)-t_s}$. Accordingly the length of $s$-patterns is $|P|=T_{s+1}N^{2M_{s+1}+3}_{s+1}h_{n-1}$ and the maximum cycle length is $T_{s+1}N^{2M_{s+1}+1}_{s+1}h_{n-1}$. In particular, there are at least $2^{4e(n-1)}$ complete cycles of $P$ and $\overline{P}$ in $E$ and $\overline{E}$, respectively. We delete partial cycles of $\overline{P}$ which might increase the $\overline{f}$ distance by at most $\frac{2}{2^{4e(n-1)}}$. 
		
		We denote the tuple of building blocks for $[P]_{s}$ by $\left([A_{1}]_{s},\dots,[A_{N_{s+1}}]_{s}\right)$
		and the tuple of building blocks for $[\overline{P}]_{s}$ by $\left([\overline{A}_{1}]_{s},\dots,[\overline{A}_{N_{s+1}}]_{s}\right)$. Moreover, let $r$ and $\overline{r}$ denote the types of $s$-Feldman pattern of $P$ and $\overline{P}$, respectively. We look at cycles $\overline{C}_{ k}$ in $\overline{P}$. Then we decompose such a cycle $\overline{C}_{ k}$ into the substitution instances $\overline{C}_{ k,i}$ of maximum consecutive repetitions $[\overline{A}_i]^{T_{s+1}N^{2\overline{r}}_{s+1}}_s$ of the same $s$-class. Let $C_{k}$ and $C_{k,i}$ denote the strings in $P$ matched under a best possible $\overline{f}$-match to $\overline{C}_{ k}$ and $\overline{C}_{ k,i}$, respectively. For a proportion of at least $1-\frac{\alpha_s}{10}$ of the indices in $E$  we have that they lie in some $C_{k,i}$ satisfying 
		\begin{equation}\label{eq:LengthC}
			1-3\frac{\alpha_s}{80}< \frac{1-\frac{\alpha_s}{80}}{1+\frac{\alpha_s}{80}} \leq \frac{|C_{k,i}|}{|\overline{C}_{ k,i}|}\leq \frac{1+\frac{\alpha_s}{80}}{1-\frac{\alpha_s}{80}} < 1+3\frac{\alpha_s}{80}
		\end{equation}
		by Fact \ref{fact:string_length}. We ignore those indices that lie in or are matched to some $C_{k,i}$ violating inequality \eqref{eq:LengthC}. This might increase the $\overline{f}$-distance by at most $\frac{\alpha_s}{4}$. Now we distinguish between the two possible cases $r<\overline{r}$ and $r> \overline{r}$. 
		\begin{itemize}
			\item In the first instance, let $r<\overline{r}$, that is, the cycle length $T_{s+1}N^{2r+1}_{s+1}h_{n-1}$ in $P$ is smaller than the length $|\overline{C}_{ k,i}|=T_{s+1}N^{2\overline{r}}_{s+1}h_{n-1}$ of maximum repetition of the same $s$-class in $\overline{P}$. Since $|C_{k,i}|>(1-3\frac{\alpha_s}{80})\cdot |\overline{C}_{ k,i}|$ by \eqref{eq:LengthC}, there are at least $\lfloor (1-3\frac{\alpha_s}{80})N^{2(\overline{r}-r)-1}_{s+1}-2 \rfloor$ complete cycles in each $C_{k,i}$. We reduce each $C_{k,i}$ to this repetition of $\lfloor (1-3\frac{\alpha_s}{80})N^{2(\overline{r}-r)-1}_{s+1}-2 \rfloor$ many cycles and delete the other symbols. By \eqref{eq:LengthC} this might increase the $\overline{f}$ distance by at most $\frac{12\alpha_s}{80}+\frac{5}{N_{s+1}}$. By Remark \ref{rem:Occurrence-Substitutions} %\cite[Lemma 46]{GK3} 
			each $C_{k,i}$ now contains each substitution instance of each $\left([A_{1}]_{s},\dots,[A_{N_{s+1}}]_{s}\right)$ the same number of times. Furthermore, let $\upsilon_{k,i}$ denote the shading of $C_{k,i}$. Since every $C_{k,i}$ contains each substitution instance of each $\left([A_{1}]_{s},\dots,[A_{N_{s+1}}]_{s}\right)$ the same number of times, there is a reshuffling $\upsilon^{\prime}_{ k,i}$ of the shading $\upsilon_{k,N_{s+1}-i}$ such that each substitution instance on $(C_{k,i}, \upsilon^{\prime}_{ k,i})$ is paired with the same shadings as on $(C_{k,N_{s+1}-i},\upsilon_{k,N_{s+1}-i})$. In fact, we build this reshuffling as follows: if $\nu \in \{0,1\}^{h_{n-1}}$ is the $\ell$-th shading occurring with a substitution instance $w \in \mathcal{W}_{n-1}$ in $(C_{k,N_{s+1}-i},\upsilon_{k,N_{s+1}-i})$, then in $\upsilon^{\prime}_{ k,i}$ we take $\nu$ to be the shading associated with the $\ell$-th occurrence of that substitution instance $w$ in $C_{k,i}$. In particular, the images $\phi(C_{k,i}, \upsilon^{\prime}_{ k,i})$ and $\phi\left(C_{k,N_{s+1}-i},\upsilon_{k,N_{s+1}-i}\right)$ are close to each other in $\overline{f}$.
			\begin{figure}
				\centering
				\includegraphics[width=\textwidth]{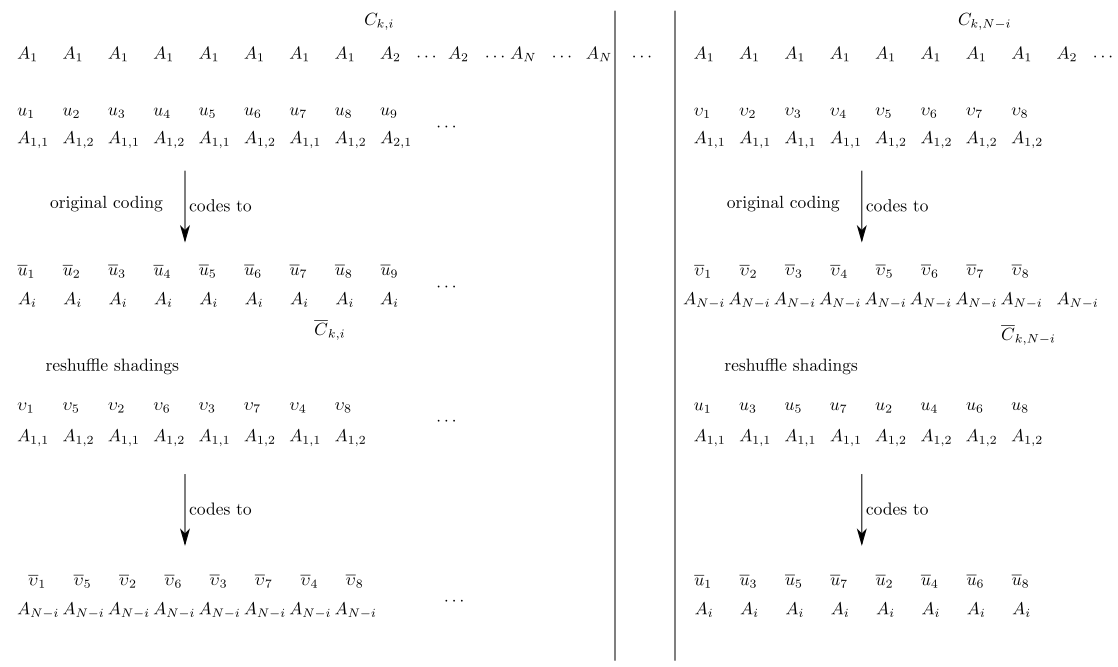}
				\caption{Illustration of the reshuffling procedure in case of Feldman pattern types $r<\overline{r}$.}
				\label{fig:fig1}
			\end{figure}
			\item We now turn to the case $r> \overline{r}$, that is, the length $T_{s+1}N^{2r}_{s+1}h_{n-1}$
			of maximum repetitions of the same $s$-class in $P$ is larger than the cycle length $T_{s+1}N^{2\overline{r}+1}_{s+1}h_{n-1}$ in $\overline{P}$. For a proportion of at most $1/N_{s+1}$, the strings $C_{k}$ overlap substitution instance of two different $s$-classes. We ignore those. Otherwise, for some $m\in \{1,\dots , N_{s+1}\}$ each $C_{k,i}$, $i=1,\dots , N_{s+1}$, contains a repetition of at least $\lfloor (1-3\frac{\alpha_s}{80})N_{s+1}-2 \rfloor$ many $[A_m]^{T_{s+1}N^{2\overline{r}-1}_{s+1}}_s$. This string contains each substitution instance of $[A_m]_s$ the same number of times. We reduce $C_{k,i}$ to this string and delete the other symbols. By \eqref{eq:LengthC} this might increase the $\overline{f}$ distance by at most $\frac{12\alpha_s}{80}+\frac{5}{N_{s+1}}$. Since every $C_{k,i}$ contains each substitution instance of $[A_m]_s$ the same number of times, we can build a reshuffling $\upsilon^{\prime}_{k,i}$ of the shading $\upsilon_{k,N_{s+1}-i}$ as in the other case such that each substitution instance on $(C_{k,i}, \upsilon^{\prime}_{ k,i})$ is paired with the same shadings as on $(C_{k,N_{s+1}-i},\upsilon_{k,N_{s+1}-i})$. As in the other case, the images $\phi\left(C_{k,i}, \upsilon^{\prime}_{ k,i}\right)$ and $\phi(C_{k,N_{s+1}-i},\upsilon_{k,N_{s+1}-i})$ are close to each other in $\overline{f}$.
			\begin{figure}
				\centering
				\includegraphics[width=\textwidth]{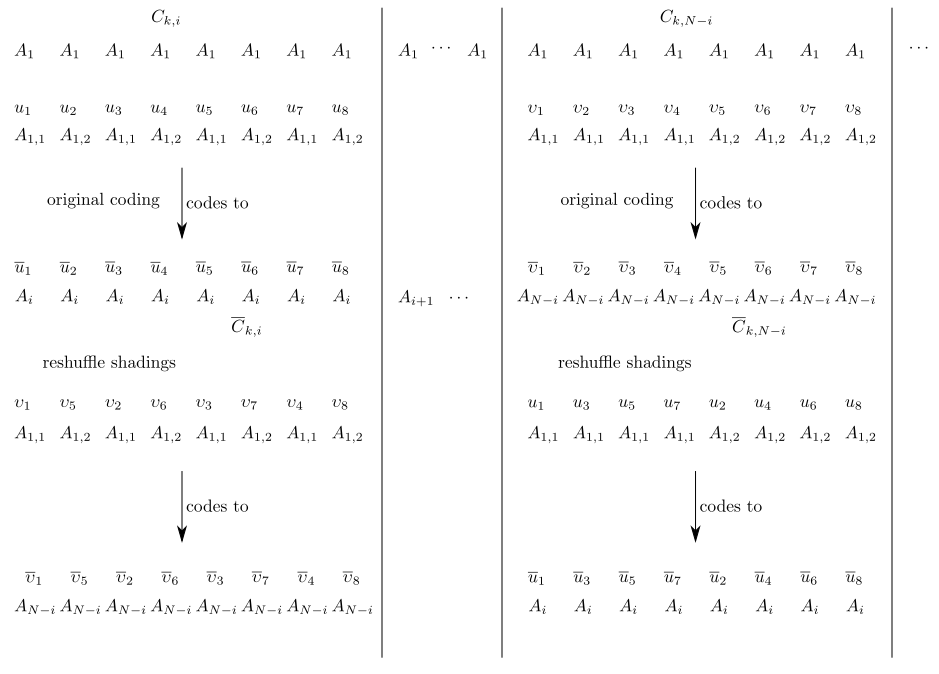}
				\caption{Illustration of the reshuffling procedure in case of Feldman pattern types $r>\overline{r}$.}
				\label{fig:fig2}
			\end{figure}
		\end{itemize}
		
		In both cases $r<\overline{r}$ and $r> \overline{r}$, we obtain, for each string $C_{k,1}\dots C_{k,N_{s+1}}$, a new shading such that the image under the code $\phi$ is $\overline{f}$-close to a cycle of the $s$-pattern $\overline{P}$ traversed in the opposite direction. Thus, we obtain a new shading $\upsilon^{\prime} \in \{0,1\}^{|E|}$ such that 
		\[
		\overline{f}\left(\phi(E,\upsilon^{\prime}),(\overline{F},\overline{\upsilon})\right)< \frac{\alpha^2_s}{2000}+\frac{2}{2^{4e(n-1)}}+\frac{\alpha_s}{4}+\frac{12\alpha_s}{80}+\frac{7}{N_{s+1}}<\frac{17\alpha_s}{40},
		\]
		where $\overline{F}$ is a string of length at least $|P|/2^{2e(n-1)+1}$ such that $\overline{F}$ is part of a $(T_{s+1}, N_{s+1}, M_{s+1})$-Feldman pattern of type $\overline{r}$ with the tuple of building blocks $\left([\overline{A}_{N_{s+1}}]_{s},[\overline{A}_{N_{s+1}-1}]_{s},\dots,[\overline{A}_{1}]_{s}\right)$ in opposite order. Since such a string $\overline{F}$ is at least $\frac{\alpha_s}{2}$-apart in $\overline{f}$ from every string in $\rev(\mathbb{K})$ by Lemma~\ref{lem:ReverseOrder}, we conclude the proof of Lemma~\ref{lem:DifferentPatterns} with the aid of Fact~\ref{fact:triangle}.
	\end{proof}

	We are now ready to prove the observation similar to \cite[Lemma~74]{GK3} that a well-approximating finite code can be identified with an element of the group action, where we recall the notation $\eta_g$ from Remark \ref{rem:BuildingIso}. %\cite[Definition~38]{GK3}.
	\begin{lem}
		\label{lem:groupelement}Let $s\in\mathbb{N}$, $0<\theta<\frac{1}{6}$,  $0<\delta<\frac{\alpha^3_s}{4\cdot 10^4}$,
		$0<\varepsilon<\frac{\alpha_{s}}{200}\delta$, and $\phi$ be a finite
		code from $\mathbb{S}$ to $\rev(\mathbb{S})$. 
		For $n$ sufficiently large (depending only on the  code length, $s$, $\delta$, and $\varepsilon$),
		for every $w\in\mathcal{W}_{n}$ such that for a proportion of at least $1-\theta$ of the shading sequences $\upsilon \in \{0,1\}^{h_n}$ of $w$ there is a string $(\overline{A},\overline{\upsilon})$ in $\rev(\mathbb{S})$ with $\overline{f}\left(\phi(w,\upsilon),(\overline{A},\overline{\upsilon})\right)<\varepsilon$, the following condition holds for at least a proportion $1-2\theta$ of the shading sequences $\upsilon \in \{0,1\}^{h_n}$ of $w$: For every string $(\overline{A},\overline{\upsilon})$
		in $\rev(\mathbb{S})$ with $\overline{f}\left(\phi(w,\upsilon),(\overline{A},\overline{\upsilon})\right)<\varepsilon$
		there must be exactly one $\overline{w}\in \rev(\mathcal{W}_{n})$
		with $|\overline{A}\cap\overline{w}|\geq(1-\delta)h_n$
		and $\overline{w}$ must be of the form $\left[\overline{w}\right]_{s}=\eta_{g}\left[w\right]_{s}$
		for a unique $g=g(w,\upsilon)$ $\in G_{s}$, which is necessarily of odd parity.
		
	\end{lem}
	
	\emph{Outline of Proof:} If $\overline{f}\left(\phi(w,\upsilon),(\overline{A},\overline{\upsilon})\right)<\varepsilon$, then it follows from a straightforward estimate on the length of strings in \eqref{eq:LengthA} that there can be at most one $\overline{w}\in \rev(\mathcal{W}_{n})$
	with $|\overline{A}\cap\overline{w}|\geq(1-\delta)|\overline{w}|$. The harder part is to show that there is exactly one such word $\overline{w}\in \rev(\mathcal{W}_{n})$ and that it has to be of the form $\left[\overline{w}\right]_{s}=\eta_{g}\left[w\right]_{s}$ for a unique $g\in G_{s}$, that is, the $s$-pattern structure of $w$ and $\overline{w}$ has to be the same. This will follow from an application of Lemma \ref{lem:DifferentPatterns} saying that for substantial substrings $E$ and $\overline{E}$ of $s$-Feldman patterns of different types there is a one-to-one map from shadings $\upsilon \in \{0,1\}^{|E|}$ such that $\overline{f}\left(\phi(E,\upsilon),(\overline{E},\tilde{\upsilon})\right)$ is small to shadings $\upsilon^{\prime} \in \{0,1\}^{|E|}$ such that $\phi(E,\upsilon^{\prime})$ is $\overline{f}$-apart from any substring in $\rev(\mathbb{S})$. Thus, the code cannot work well on at least half of the shadings if $w$ and $\overline{w}$ have a different $s$-pattern structure. 
	
	We now turn to the detailed proof.
	
	\begin{proof}
		Our initial requirement on the size of $n$ is that $h_{n-1}$ be much
		larger than the length of the code, so that end effects are negligible for $(n-1)$-words. Let $\upsilon \in \{0,1\}^{h_n}$ be a shading of $w$ such that there is a string $(\overline{A},\overline{\upsilon})$ in $\rev(\mathbb{S})$ with $\overline{f}\left(\phi(w,\upsilon),(\overline{A},\overline{\upsilon})\right)<\varepsilon$. By Fact
		\ref{fact:string_length} we know from $\overline{f}\left(\phi(w,\upsilon),(\overline{A},\overline{\upsilon})\right)<\varepsilon$
		that 
		\begin{equation}\label{eq:LengthA}
			1-3\varepsilon\leq\frac{1-\varepsilon}{1+\varepsilon}\leq\frac{|\overline{A}|}{h_n}\leq\frac{1+\varepsilon}{1-\varepsilon}\leq1+3\varepsilon.
		\end{equation}
		Hence, there can be at most one $\overline{w}\in \rev(\mathcal{W}_{n})$
		with $|\overline{A}\cap\overline{w}|\geq(1-\delta)|\overline{w}|$.
		
		\begin{claim} \label{claim:claim0}
			Suppose that there is no such $\overline{w}\in \rev(\mathcal{W}_{n})$. Then there exists another shading $\upsilon^{\prime} \in \{0,1\}^{h_n}$ such that $\overline{f}\left(\phi(w,\upsilon^{\prime}),(\overline{D},\overline{u})\right)>2\varepsilon$ for all strings $(\overline{D},\overline{u})$ in $\rev(\mathbb{S})$.
			
			In fact, there will be an one-to-one map from shadings $\upsilon\in\{0,1\}^{h_n}$ satisfying 
			\begin{equation*}
				\begin{split}
					\overline{f}\left(\phi(w,\upsilon),(\overline{A},\overline{\upsilon})\right)<\varepsilon \ &\text{ for a string $(\overline{A},\overline{\upsilon})$ in $\rev(\mathbb{S})$, where there is } \\
					& \text{ no $\overline{w}\in \rev(\mathcal{W}_{n})$
						with $|\overline{A}\cap\overline{w}|\geq(1-\delta)|\overline{w}|$,}
				\end{split}   
			\end{equation*}
			to shadings $\upsilon^{\prime} \in \{0,1\}^{h_n}$ such that 
			\[
			\overline{f}\left(\phi(w,\upsilon^{\prime}),(\overline{D},\overline{u})\right)>2\varepsilon \ \text{ for all strings $(\overline{D},\overline{u})$ in $\rev(\mathbb{S})$.}
			\]
		\end{claim}
		
		We prove Claim \ref{claim:claim0} below. It implies that for a proportion of at least $1-2\theta$ of possible shading sequences $\upsilon\in \{0,1\}^{h_n}$, for any string $(\overline{A},\overline{\upsilon})$
		in $\rev(\mathbb{S})$ with $\overline{f}\left(\phi(w,\upsilon),(\overline{A},\overline{\upsilon})\right)<\varepsilon$,
		there must be exactly one $\overline{w}\in \rev(\mathcal{W}_{n})$ with $|\overline{A}\cap\overline{w}|\geq(1-\delta)h_n$. By the same method of proof, one also obtains that $w$ and $\overline{w}$
		need to have the same $s$-pattern structure. By the construction in \cite{GK3}, this
		only happens if their equivalence classes $\left[w\right]_{s}$ and
		$\left[\overline{w}\right]_{s}$ lie on the same orbit of the group
		action by $G_{s}^{n}$ (recall from step (\ref{item:step5})
		that substitution instances not lying on the same $G_{s}^{n}$ orbit
		were constructed as a different concatenation of different Feldman
		patterns). We have seen in Remark~\ref{rem:BuildingIso} that an element $g\in G_s^n$ of odd parity induces $\eta_g:\mathcal{W}_n/\mathcal{Q}^n_s\to \rev(\mathcal{W}_n)/\mathcal{Q}^n_s$ that preserves the order of $s$-Feldman patterns, while an element $g\in G_s^n$ of even parity would induce a map $\mathcal{W}_n/\mathcal{Q}^n_s\to \rev(\mathcal{W}_n)/\mathcal{Q}^n_s$ that reverses the order of $s$-Feldman patterns. By Lemma \ref{lem:fDist} and the freeness of the group action according to specification \ref{item:A7} %\cite[Proposition 56]{GK3}, 
		we conclude that there is a unique $g\in G_{s}$, which has to be of odd parity, with $\left[\overline{w}\right]_{s}=\eta_{g}\left[w\right]_{s}$. This yields Lemma \ref{lem:groupelement}.

		\begin{proof}[Proof of Claim \ref{claim:claim0}]
			By assumption, there is no $\overline{w}\in \rev(\mathcal{W}_{n})$
			with $|\overline{A}\cap\overline{w}|\geq(1-\delta)|\overline{w}|$. Since we know $|\overline{A}|\geq(1-3\varepsilon)|\overline{w}|$, there have to be two words $\overline{w}_{1},\overline{w}_{2}\in \rev(\mathcal{W}_{n})$
			such that $\overline{A}$ is a substring of $\overline{w}_{1}\overline{w}_{2}$
			and 
			\begin{equation} \label{eq:LengthBarAj}
				|\overline{A}\cap\overline{w}_{j}|\geq\left(\delta-3\varepsilon\right)\cdot|\overline{w}_{j}|>\frac{180}{\alpha_{s}}\varepsilon h_{n}\text{ for }j=1,2.
			\end{equation}
			Then the beginning of $\phi(w,\upsilon)$ is matched with a string in $\rev(\mathbb{S})$ associated with the end of $\overline{w}_{1}$,
			while the end of $\phi(w,\upsilon)$ is matched with a string in $\rev(\mathbb{S})$ associated with the beginning of $\overline{w}_{2}$.
			There are two possible cases: Either at least one of $\overline{w}_{1},\overline{w}_{2}$
			has a different $s$-pattern structure from that of $w$, or $\overline{w}_{1}$ and $\overline{w}_{2}$
			have the same $s$-pattern structure as $w$. Here we point out that the
			Feldman pattern types in the
			$s$-pattern structure of the beginning and end of $w$ are different
			from each other by construction. (Recall from 
			step (\ref{item:step5}) that substitution instances
			were constructed as a concatenation of different Feldman patterns.)
			
			To investigate both cases we denote by $\overline{A}_{j}$ the part of $\overline{A}$ corresponding to $\overline{w}_{j}$, $j=1,2$.
			Let $\pi$ be a best possible match in $\overline{f}$ between $\phi(w,\upsilon)$
			and $(\overline{A},\overline{\upsilon})$. Then we denote the parts in $\phi(w,\upsilon)$ corresponding
			to $(\overline{A}_j,\overline{\upsilon}_j)$ under $\pi$ by $\phi(A_{j},\upsilon_j)$. Since $\pi$ is a map on the indices, we can also view it as a
			map between $\phi(A_{j},\upsilon_j)$ and $(\overline{A}_j,\overline{\upsilon}_j)$. 
			
			We note that $\overline{f}\left(\phi(w,\upsilon),(\overline{A},\overline{\upsilon})\right)<\varepsilon$ requires $\overline{f}\left(\phi(A_j,\upsilon_j),(\overline{A}_j,\overline{\upsilon}_j)\right)<\frac{\alpha_s}{60}$ for both $j\in \{1,2\}$ because otherwise
			\begin{align*}
				& \overline{f}\left(\phi(w,\upsilon),(\overline{A},\overline{\upsilon})\right) \\
				= & \frac{|A_1|+|\overline{A}_1|}{|w|+|\overline{A}|}\cdot \overline{f}\left(\phi(A_1,\upsilon_1),(\overline{A}_1,\overline{\upsilon}_1)\right) + \frac{|A_2|+|\overline{A}_2|}{|w|+|\overline{A}|}\cdot \overline{f}\left(\phi(A_2,\upsilon_2),(\overline{A}_2,\overline{\upsilon}_2)\right) \\
				\geq & \frac{\frac{180}{\alpha_{s}}\varepsilon h_{n}}{h_n + (1+3\varepsilon)h_n}\cdot \frac{\alpha_s}{60} > \varepsilon 
			\end{align*}
			by Fact~\ref{fact:substring_matching} and our length estimate on $\overline{A}_j$ in \eqref{eq:LengthBarAj}. Then Fact~\ref{fact:string_length} and inequality~\eqref{eq:LengthBarAj} yield
			\begin{equation}\label{eq:LengthAj}
				|A_j|\geq \frac{1-\frac{\alpha_s}{60}}{1+\frac{\alpha_s}{60}}\cdot |\overline{A}_j| > (1-3\frac{\alpha_s}{60})\cdot \frac{180}{\alpha_{s}}\varepsilon h_{n} > \frac{160}{\alpha_{s}}\varepsilon h_{n}
			\end{equation}
			for both $j\in \{1,2\}$.
			
			In order to have $\overline{f}\left(\phi(w,\upsilon),(\overline{A},\overline{\upsilon})\right)<\varepsilon$, we need to have $\overline{f}\left(\phi(A_j,\upsilon_j),(\overline{A}_j,\overline{\upsilon}_j)\right)<\varepsilon$ for at least one $j\in \{1,2\}$. 
			
			\begin{claim}\label{claim:claim1}
				There is a one-to-one map from shadings $\upsilon_j\in \{0,1\}^{|A_j|}$ such that 
				\begin{equation}\label{eq:fbarj}
					\overline{f}\left(\phi(A_j,\upsilon_j),(\overline{A}_j,\overline{\upsilon}_j)\right)<\varepsilon
				\end{equation}
				to shadings $\upsilon^{\prime}_j \in \{0,1\}^{|A_j|}$ such that 
				\[
				\overline{f}\left(\phi(A_j,\upsilon^{\prime}_j),(\overline{D}_j,\overline{u}_j)\right)> \frac{\alpha_s}{30}
				\]
				for all strings $(\overline{D}_j,\overline{u}_j)$ in $rev(\mathbb{S})$.
			\end{claim}
			
			We prove Claim \ref{claim:claim1} below. Suppose that $(\overline{D},\overline{u})$ is any string in $\rev(\mathbb{S})$ with length $|\overline{D}|<(1+6\varepsilon)h_n<1.5h_n$ (otherwise, we have $\overline{f}\left(\phi(w,\upsilon^{\prime}),(\overline{D},\overline{u})\right)>2\varepsilon$ by Fact \ref{fact:string_length}). For every shading $\upsilon$ of $w$, we write $\upsilon=(\upsilon_1,\upsilon_2)$, where $\upsilon_j$ is the shading of $w_j$. If equation \eqref{eq:fbarj} holds with $j=1$, define $\upsilon^{\prime}=(\upsilon_1^{\prime},\upsilon_2)$; otherwise define $\upsilon^{\prime} =(\upsilon_1,\upsilon_2^{\prime}).$  
			We combine the result from Claim \ref{claim:claim1} with the estimate on $|A_j|$ from \eqref{eq:LengthAj} to conclude 
			%with the aid of 
			using Fact \ref{fact:substring_matching} that 
			%there is some 
			the shading $\upsilon^{\prime} \in \{0,1\}^{h_n}$ 
			is such that 
			\begin{equation}\label{eq:badj}
				\overline{f}\left(\phi(w,\upsilon^{\prime}),(\overline{D},\overline{u})\right) \geq \frac{|A_j|+|\overline{D}_j|}{|w|+|\overline{D}|}\cdot \frac{\alpha_s}{30} > \frac{2}{5} \cdot \frac{160}{\alpha_{s}}\varepsilon \cdot \frac{\alpha_s}{30} > 2\varepsilon.
			\end{equation}
			Clearly the map $\upsilon\mapsto\upsilon^{\prime}$ is one-to-one in each of these two cases separately. If there were $\upsilon=(\upsilon_1,\upsilon_2)$ such that equation \eqref{eq:fbarj} holds for $j=1$ and $\tilde{\upsilon}=(\tilde{\upsilon}_1,\tilde{\upsilon}_2)$ such that \eqref{eq:fbarj} does not hold for $j=1$, then it is not possible to have $\upsilon_2={\tilde{\upsilon}}_2^{\prime}$ since $v_2=\tilde{v}_2^{\prime}$ would imply $\overline{f}\left(\phi(w,\upsilon),(\overline{A},\overline{\upsilon}\right)>2\varepsilon,$
			by the estimate in \eqref{eq:badj}, which contradicts the assumption that $\overline{f}(\phi(w,\upsilon),(\overline{A},\overline{\upsilon}))<\varepsilon.$
			Therefore the map $\upsilon\mapsto \upsilon^{\prime}$ is one-to-one. 
			This finishes the proof of Claim \ref{claim:claim0}.
		\end{proof}
		
		\begin{proof}[Proof of Claim \ref{claim:claim1}]
			Since $\overline{f}\left(\phi(A_j,\upsilon_j),(\overline{A}_j,\overline{\upsilon}_j)\right)<\varepsilon$, we have $|A_{j}|\geq\frac{160}{\alpha_{s}}\varepsilon h_{n}$ by \eqref{eq:LengthAj}.
			Hence, for sufficiently large $n$ the strings $A_j$ and $\overline{A}_{j}$ are longer than $h_{n}/2^{e(n-1)}$. In particular, there are at least $2^{6e(n-1)}$ complete $s$-Feldman patterns (out of building blocks in $\mathcal{W}_{n-1}/\mathcal{Q}_{s}^{n-1}$)
			in $A_j$ and $\overline{A}_j$. Moreover, all of them have the same length.
			(If we are in Case 2, $s(n)=s(n-1)+1$, of the construction in \cite[section~8.2]{GK3} we ignore the strings coming
			from $\mathcal{W}^{\dagger\dagger}$, which might increase the $\overline{f}$
			distance by at most $\frac{2}{\mathfrak{p}_{n}}$ according to the length portion of strings from $\mathcal{W}^{\dagger\dagger}$ in words of $\mathcal{W}_n$
			and Fact \ref{fact:omit_symbols}.) We complete partial $s$-patterns
			at the beginning and end of $A_j$ and $\overline{A}_j$, which might increase
			the $\overline{f}$ distance by at most $\frac{2}{2^{6e(n-1)}}$ according
			to Fact \ref{fact:omit_symbols} as well. The augmented strings obtained
			in this way are denoted by $A_{j,aug}$ and $\overline{A}_{j,aug}$, respectively,
			and we decompose them into the $s$-Feldman patterns:
			\begin{align*}
				A_{j,aug}=P_{j,n-1,1}\dots P_{j,n-1,m} &  &  & \overline{A}_{j,aug}=\overline{P}_{j,n-1,1}\dots\overline{P}_{j,n-1,\overline{m}}.
			\end{align*}
			In the next step, we write $A_{j,aug}$ and $\overline{A}_{j,aug}$ as
			\begin{align*}
				A_{j,aug}=E_{j,1}\dots E_{j,t} &  &  & \overline{A}_{j,aug}=\overline{E}_{j,1}\dots\overline{E}_{j,t},
			\end{align*}
			where $E_{j,i}$ and $\overline{E}_{j,i}$ are maximal substrings such
			that $E_{j,i}\subseteq P_{j,n-1,\ell(i)}$ and $\overline{E}_{j,i}\subseteq\overline{P}_{j,n-1,\overline{\ell}(i)}$
			correspond to each other under $\pi$. In particular, we have $t\geq2^{6e(n-1)}$.
			
			If $A_j$ and $\overline{A}_j$ lie in $w$ and $\overline{w}_j$ with the same $s$-pattern structure, then $|A_{j'}|>\frac{160}{\alpha_{s}}\varepsilon h_{n}$ for $j'\neq j$ from \eqref{eq:LengthAj} and $\overline{f}\left(\phi(A_j,\upsilon_j),(\overline{A}_j,\overline{\upsilon}_j)\right)<\varepsilon$ together imply that the Feldman pattern types of $ P_{j,n-1,\ell(i)}$ and $\overline{P}_{j,n-1,\overline{\ell}(i)}$ cannot be the same.
			
			If the $s$-pattern structures of $w$ and $\overline{w}_j$ are different from each other, then we have at most one pair of $E_{j,i}$ and $\overline{E}_{j,i}$ where the type of Feldman pattern coincides. We delete this possible pair.
			This might increase the $\overline{f}$ distance by at most $\frac{2}{2^{6e(n-1)}}$
			according to Fact \ref{fact:omit_symbols} because we have $t\geq2^{6e(n-1)}$.
			
			Thus, in the following we can assume that all pairs $E_{j,i}$ and $\overline{E}_{j,i}$ lie in Feldman patterns of different type. 
			
			After completing the partial $s$-patterns and deleting at most one pair of $E_{j,i}$ and $\overline{E}_{j,i}$, we still have  $\overline{f}\left(\phi(A_{j,aug},\upsilon_{j,aug}),(\overline{A}_{j,aug},\overline{\upsilon}_{j,aug})\right)<2\varepsilon<\frac{\alpha^4_s}{4\cdot 10^6}$ if $n$ is sufficiently large. Hence, a proportion of at least $1-\frac{\alpha^2_s}{100}$ of the indices have to lie in pairs $E_{j,i}$ and $\overline{E}_{j,i}$ with $\overline{f}\left(\phi(E_{j,i},\upsilon_{j,i}),(\overline{E}_{j,i},\overline{\upsilon}_{j,i})\right)<\frac{\alpha^2_s}{2000}$. 
			
			If $\overline{f}\left(\phi(E_{j,i},\upsilon_{j,i}),(\overline{E}_{j,i},\overline{\upsilon}_{j,i})\right)<\frac{\alpha^2_s}{2000}$, then we have 
			\begin{equation}
				1-\frac{1}{2^{4}}<\frac{1-\frac{1}{2^{5}}}{1+\frac{1}{2^{5}}}\leq\frac{|E_{j,i}|}{|\overline{E}_{j,i}|}\leq\frac{1+\frac{1}{2^{5}}}{1-\frac{1}{2^{5}}}<1+\frac{1}{2^{3}}\label{eq:Glength}
			\end{equation}
			by Fact \ref{fact:string_length}. Since a string well-matched with $P_{j,n-1,\ell}$ has to lie within
			at most three consecutive patterns in $\overline{A}_{j,aug}$, in each $P_{j,n-1,\ell}$ there are
			at most two segments $E_{j,i}$ with $|E_{j,i}|<\frac{|P_{j,n-1,\ell}|}{2^{2e(n-1)}}$. For their
			corresponding segments we also get $|\overline{E}_{j,i}|<\frac{|P_{j,n-1,\ell}|}{2^{2e(n-1)-1}}$
			by the estimate in \eqref{eq:Glength}. By the same reasons, there
			are at most two segments $\overline{E}_{j,i}$ with $|\overline{E}_{j,i}|<\frac{|P_{j,n-1,\ell}|}{2^{2e(n-1)}}$
			in each $\overline{P}_{j,n-1,\ell}$ and we can also bound the lengths
			of their corresponding strings. Altogether, the proportion of symbols
			in those situations is at most $\frac{1}{2^{2e(n-1)-4}}$ and we ignore
			them in the following consideration.
			
			Hence, we consider the case that both $|E_{j,i}|$ and $|\overline{E}_{j,i}|$
			are at least $\frac{|P_{j,n-1,\ell}|}{2^{2e(n-1)}}$. This allows us to apply Lemma \ref{lem:DifferentPatterns}. Combining it with the previous estimates, we conclude the existence of a $\upsilon^{\prime}_j \in \{0,1\}^{|A_j|}$ such that 
			\[
			\overline{f}\left(\phi(A_j,\upsilon^{\prime}_j),(\overline{D}_j,\overline{u}_j)\right)>\frac{\alpha_s}{20}-\frac{2}{\mathfrak{p}_{n}}-\frac{4}{2^{e(n-1)}}-2\frac{\alpha^2_s}{100} > \frac{\alpha_s}{30}
			\]
			for all strings $(\overline{D}_j,\overline{u}_j)$ in $rev(\mathbb{S})$. This finishes the proof of Claim \ref{claim:claim1}. 
		\end{proof}
		
		Altogether, we completed the proof of Lemma \ref{lem:groupelement}.
		
	\end{proof}
	
	Assume $\mathbb{S}$ and $\rev(\mathbb{S})$ are evenly equivalent. From \cite[Corollary 71]{GK3} applied to $\mathbb{S}$ there exists a sequence of finite
	codes $\left(\phi_{\ell}\right)_{\ell \in\mathbb{N}}$ from $\mathbb{S}$ to $\rev(\mathbb{S})$ such that for every $0<\gamma<\frac{1}{2}$ there exists $N=N(\gamma)\in \Z^+$ such that  for $N\ge N(\gamma)$
	and $k\in\mathbb{Z}^{+}$ and all sufficiently large $n$ (how large
	depends on $N$,$\gamma$, and $k$) there is a collection
	$\mathcal{V}_{n}'\subset\mathcal{V}_{n}$ that includes at least $1-\gamma$
	of the strings in $\mathcal{V}_{n}$ and the following condition holds:
	For $(w,\upsilon)\in\mathcal{V}_{n}'$ there exists $(\overline{A},\overline{\upsilon})$ that
	occurs with positive
	frequency in the $\rev(\mathbb{S})$ process and we have 
	\[
	\overline{f}\left(\phi_{N}(w,\upsilon),(\overline{A},\overline{\upsilon})\right)<\gamma,
	\]
	and
	\[
	\overline{f}\left(\phi_{N}(w,\upsilon),\phi_{N+k}(w,\upsilon)\right)<\gamma.
	\]
	We apply it with $\gamma=\gamma_s  = \frac{\alpha^8_s}{4\cdot 10^{24}}$ for $s \in \mathbb{Z}^+$ to obtain $N(s)\in \mathbb{Z}^+$ such that for every $N \geq N(s)$ and $k\in \mathbb{Z}^+$ there is $n(s,N,k) \in \mathbb{Z}^+$ sufficiently large so that for all $n\geq n(s,N,k)$ the hypothesis of Lemma~\ref{lem:groupelement} with $\theta=\sqrt{\gamma_s}$, $\delta=\delta_s = \frac{\alpha^6_s}{10^{22}}$, $\varepsilon=\gamma_s$, and $\phi=\phi_{N}$ holds, and there exists a collection $\mathcal{W}^{\prime}_n$ of $n$-words (that includes at least $1-\sqrt{\gamma_s}$ of the $n$-words) satisfying the following properties:
	\begin{enumerate}[label=(C\arabic*)]
		\item\label{item:C1} For $w \in \mathcal{W}^{\prime}_n$ we have for a proportion of at least $1-\sqrt{\gamma_s}$ of possible shading sequences $\upsilon \in \{0,1\}^{h_n}$ that there exists $z = z(w,\upsilon) \in \rev(\mathbb{S})$ such that $z\upharpoonright[0,h_n-1]$ occurs with positive frequency in $\rev(\mathbb{S})$ and 
		\[
		\overline{f}\left(\phi_{N}(w,\upsilon),z\upharpoonright[0,h_n-1]\right)<\frac{\alpha^8_s}{4\cdot 10^{24}}.
		\]
		\item\label{item:C2} For $w \in \mathcal{W}^{\prime}_n$ we have for a proportion of at least $1-\sqrt{\gamma_s}$ of possible shading sequences $\upsilon \in \{0,1\}^{h_n}$ that
		\[
		\overline{f}\left(\phi_{N}(w,\upsilon),\phi_{N+k}(w,\upsilon)\right)<\frac{\alpha^8_s}{4\cdot 10^{24}}.
		\]
	\end{enumerate} 
	
	Then we can prove the following analogue of \cite[Lemma 75]{GK3}.
	
	\begin{lem}
		\label{lem:codeGroup}Suppose that $\mathbb{S}$ and $\rev(\mathbb{S})$ are evenly equivalent and let $s\in\mathbb{N}$. There is a unique $g_{s}\in G_{s}$
		such that for every $N\geq N(s)$ and sufficiently large $n\in\mathbb{N}$
		we have for every $w\in\mathcal{W}_{n}^{\prime}$ that for a proportion of at least $\frac{99}{100}$ of possible shading sequences $\upsilon \in \{0,1\}^{h_n}$ there is $\overline{w}\in \rev(\mathcal{W}_{n})$
		with $\left[\overline{w}\right]_{s}=\eta_{g_{s}}\left[w\right]_{s}$
		and some shading sequence $\overline{\upsilon} \in \{0,1\}^{h_n}$ such that
		\begin{equation}\label{eq:codeGroup}
			\overline{f}\left(\phi_{N}(w,\upsilon),(\overline{w},\overline{\upsilon})\right)<\frac{\alpha_{s}}{4},
		\end{equation}
		where $\left(\phi_{\ell}\right)_{\ell \in\mathbb{N}}$ is a sequence of finite
		codes as described above and $\mathcal{W}_{n}^{\prime}$ is the associated collection of $n$-words satisfying properties \ref{item:C1} and \ref{item:C2}.
		
		Moreover, $g_{s}\in G_{s}$ is of odd parity and the sequence $(g_{s})_{s\in\mathbb{N}}$
		satisfies $g_{s}=\rho_{s+1,s}(g_{s+1})$ for all $s\in\mathbb{N}$.
	\end{lem}
	
	\emph{Outline of Proof:} An application of Lemma \ref{lem:groupelement} will imply that for every $w\in\mathcal{W}_{n+1}^{\prime}$ there is a large proportion of shadings such that \eqref{eq:codeGroup} holds for $\overline{w}\in \rev(\mathcal{W}_{n+1})$
	with $\left[\overline{w}\right]_{s}=\eta_{g_{s}}\left[w\right]_{s}$ for some $g_s \in G_s$. The hard part is to show that actually the same group element of $G_s$ works for all words in $\mathcal{W}_{n+1}^{\prime}$. This is the content of Claim~\ref{claim:UniqueG}. Its proof bases upon the observation that a good $\overline{f}$ match between $\phi_{N}(w,\upsilon)$ and $(\overline{w},\overline{\upsilon})$ has to respect the $s$-pattern structure in $w$, that is, if $w=P_{1}\dots P_{r}$ is the decomposition into $s$-patterns and $\overline{w}=\overline{P}_1\dots \overline{P}_r$ is the corresponding decomposition under a best possible $\overline{f}$-match between $\phi_N(w,\upsilon)$ and $(\overline{w},\overline{\upsilon})$, then most of the $[\overline{P}_{i}]_{s}$
	have to lie in $\eta_{g_{s}}[P_{i}]_{s}$. Since the classes	within the building tuple of $P_{i}$ have disjoint $G_{s}$ orbits by steps (\ref{item:step1}) and (\ref{item:step8}) of the construction in Section~\ref{sec:review}, we can show that for a large proportion of shaded $n$-words $(w_{j},\upsilon_j)$  there has to be $\overline{w}_{j}\in \rev(\mathcal{W}_{n})$ with $\left[\overline{w}_j\right]_{s}=\eta_{g_{s}}\left[w_j\right]_{s}$ such that $\overline{f}\left(\phi_{N}(w_{j},\upsilon_j),(\overline{w}_{j},\overline{\upsilon}_j)\right)$ is small for some shading sequence $\overline{\upsilon}_j \in \{0,1\}^{h_{n}}$. Using assumption \ref{item:R9} this implies that for a large proportion of $n$-words $w_j$ we have for a large proportions of shadings $\upsilon_j$ that there is $\overline{w}_{j}\in \rev(\mathcal{W}_{n})$ with $\left[\overline{w}_j\right]_{s}=\eta_{g_{s}}\left[w_j\right]_{s}$ such that $\overline{f}\left(\phi_{N}(w_{j},\upsilon_j),(\overline{w}_{j},\overline{\upsilon}_j)\right)$ is small. This will allow us to conclude that the same group element of $G_s$ works for all words in $\mathcal{W}_{n+1}^{\prime}$.
	
	%this allows us to conclude using assumption (R1) that $\left[\overline{w}\right]_{s}=\eta_{g_{s}}\left[w\right]_{s}$ for a large proportion of $n$-words. 
	
	Afterwards we show that the same group element $g_{s}\in G_{s}$
	works for all $\phi_{N}$ with $N\geq N(s)$. Here, we use that images under $\phi_N$ and $\phi_{N+k}$ are $\overline{f}$-close to each other, while the properties from the construction in \cite{GK3} yield that distinct $s$-classes have $\overline{f}$ distance at least $\alpha_s$ by Lemma \ref{lem:fDist} and that the action by $G_s$ is free by specification \ref{item:A7}.
	
	In a final step we combine these properties with specification \ref{item:A7} (the $G_{s+1}^{n}$ action is subordinate to the $G_{s}^{n}$
	action) to show that the sequence $(g_{s})_{s\in\mathbb{N}}$ must satisfy $g_{s}=\rho_{s+1,s}(g_{s+1})$ for all $s\in\mathbb{N}$. 
	
	\begin{proof}
		In the first part of the proof we fix $N \geq N(s)$. For $n\geq n(s,N,1)$ we see from property \ref{item:C1} that for any $w\in \mathcal{W}^{\prime}_n$ there is a proportion of at least $1-\sqrt{\gamma_s}$ of possible shading sequences $\upsilon \in \{0,1\}^{h_n}$ such that there exists $z\in \rev(\mathbb{S})$ with 
		\[
		\overline{f}\left(\phi_{N}(w,\upsilon),z\upharpoonright[0,h_n-1]\right)<\gamma_s=\frac{\alpha^8_s}{4\cdot 10^{24}} < \frac{\alpha_s}{200}\cdot \delta_s.
		\]
		We denote $(\overline{A},\hat{\upsilon})\coloneqq z\upharpoonright[0,h_n-1]$ in $\rev(\mathbb{S})$.
		By Lemma \ref{lem:groupelement} we obtain for a proportion of at least $1-2\sqrt{\gamma_s}$ of possible shading sequences $\upsilon \in \{0,1\}^{h_n}$ that there is exactly one $\overline{w} = \overline{w}(w,\upsilon)\in \rev(\mathcal{W}_{n})$
		with $|\overline{A}\cap\overline{w}|\geq(1-\delta_{s})|\overline{w}|$
		and $\overline{w}$ must be of the form $\left[\overline{w}\right]_{s}=\eta_{g}\left[w\right]_{s}$
		for a unique $g=g(w,\upsilon)\in G_{s}$, which is of odd parity. Thus,
		we conclude that there is some shading sequence $\overline{\upsilon} = \overline{\upsilon}(w,\upsilon) \in \{0,1\}^{h_n}$ such that 
		\begin{equation}\label{eq:GforN}
			\overline{f}\left(\phi_{N}(w,\upsilon),(\overline{w},\overline{\upsilon})\right)<\gamma_s+\delta_{s}<10^{-20}\alpha_{s}^{6}.
		\end{equation} 
		
		We now show that one group element of $G_{s}$ works for
		$\phi_{N}$ and all words in $\mathcal{W}^{\prime}_{n+1}$.
		
		\begin{claim} \label{claim:UniqueG}
			There is a unique $g_s \in G_s$ such that for every $w\in \mathcal{W}^{\prime}_{n+1}$ we have for a proportion of at least $\frac{99}{100}$ of possible shading sequences $\upsilon \in \{0,1\}^{h_{n+1}}$ that there is $\overline{w}=\overline{w}(w,\upsilon)\in \rev(\mathcal{W}_{n+1})$
			with $\left[\overline{w}\right]_{s}=\eta_{g_{s}}\left[w\right]_{s}$
			and some shading sequence $\overline{\upsilon}=\overline{\upsilon}(w,\upsilon) \in \{0,1\}^{h_{n+1}}$ such that
			\[
			\overline{f}\left(\phi_{N}(w,\upsilon),(\overline{w},\overline{\upsilon})\right)<\frac{\alpha_{s}}{4}.
			\]
		\end{claim}

		In order to prove Claim \ref{claim:UniqueG}, we repeat the argument from \eqref{eq:GforN} for a word $w\in\mathcal{W}^{\prime}_{n+1}$ which is a concatenation of $n$-words
		by construction. Hence, for a proportion of at least $1-2\sqrt{\gamma_s}$ of possible shading sequences $\upsilon \in \{0,1\}^{h_{n+1}},$ 
		\begin{equation}\label{eq:gReq}
			\begin{split}
				& \text{there is a unique $g_{s}=g_s(w,\upsilon)\in G_{s}$ and some $\overline{w}=\overline{w}(w,\upsilon)\in \rev(\mathcal{W}_{n+1})$}\\
				& \text{with $\left[\overline{w}\right]_{s}=\eta_{g_{s}}\left[w\right]_{s}$ 
					and some shading sequence $\overline{\upsilon} = \overline{\upsilon}(w,\upsilon) \in \{0,1\}^{h_{n+1}}$ }\\
				& \text{such that $\overline{f}\left(\phi_{N}(w,\upsilon),(\overline{w},\overline{\upsilon})\right)<10^{-20}\alpha_{s}^{6}$.}
			\end{split}
		\end{equation}

		\begin{claim} \label{claim:UniqueGpre}
			For every $w\in \mathcal{W}^{\prime}_{n+1}$ there is a collection $\mathcal{S} = \mathcal{S} (w) \subseteq \{0,1\}^{h_{n+1}}$ with $|\mathcal{S}|\geq \frac{99}{100}2^{h_{n+1}}$ such that for each $\upsilon \in \mathcal{S}$ we have for the $g_s = g_s(w,\upsilon) \in G_s$ from \eqref{eq:gReq} that for a proportion of at least $\frac{8}{10}$ of $n$-blocks $\tilde{w}\in \mathcal{W}_n$ and for a proportion of at least $\frac{8}{10}$ of all shading sequences $\tilde{\upsilon} \in \{0,1\}^{h_n}$ there is $\hat{w}\in \rev(\mathcal{W}_{n})$
			with $\left[\hat{w}\right]_{s}=\eta_{g_{s}}\left[\tilde{w}\right]_{s}$  and $\overline{f}\left(\phi_{N}(\tilde{w},\tilde{\upsilon}),(\hat{w},\hat{\upsilon})\right)<\frac{\alpha_{s}}{4}$ for some shading sequence $\hat{\upsilon} \in \{0,1\}^{h_{n}}$.
		\end{claim}

		\begin{proof}[Proof of Claim \ref{claim:UniqueGpre}]
			Combining the requirement \eqref{eq:gReq} with property \ref{item:R9}, there is a collection $\mathcal{S}^{\prime}\subset \{0,1\}^{h_{n+1}}$ with $|\mathcal{S}^{\prime}|>\frac{499}{500}2^{h_{n+1}}$ of shading sequences $\upsilon$ satisfying \eqref{eq:EquiShading} and that there is a unique $g_{s}=g_s(w,\upsilon)\in G_{s}$ and $\overline{w}=\overline{w}(w,\upsilon)\in \rev(\mathcal{W}_{n+1})$
			with $\left[\overline{w}\right]_{s}=\eta_{g_{s}}\left[w\right]_{s}$ and some shading sequence $\overline{\upsilon} = \overline{\upsilon}(w,\upsilon) \in \{0,1\}^{h_{n+1}}$
			such that $\overline{f}\left(\phi_{N}(w,\upsilon),(\overline{w},\overline{\upsilon})\right)<10^{-20}\alpha_{s}^{6}$.
			
			We decompose $w$ into its $s$-Feldman patterns: $w=P_{1}\dots P_{r}$. For each $\upsilon \in \mathcal{S}^{\prime}$ we also decompose $(w,\upsilon)=(P_1,\upsilon_1)\dots (P_r,\upsilon_r)$ with $\upsilon_i \in \{0,1\}^{|P_i|}$. 
			
			Consider those $(P_i,\upsilon_i)$ for which there exists a substring $(\tilde{P}_i,\tilde{\upsilon}_i)$ of length at least $\frac{\alpha^4_{s}}{10^{12}}|P_i|$ and a $(\hat{P},\hat{\upsilon})$, where $\hat{P}$ is a substring of some element of $G_s[w]_s$ and of different pattern type from $P_i$, with $\overline{f}\left(\phi_N(\tilde{P}_i,\tilde{\upsilon}_i),(\hat{P},\hat{\upsilon})\right)<\frac{\alpha^2_s}{2000}$. For such $(P_i,\upsilon_i)$ let $\upsilon^{\prime}_i$ be the reshuffling from Lemma~\ref{lem:DifferentPatterns} according to the first pattern type that occurs in such a $\hat{P}$. For the other $i$'s, let $\upsilon^{\prime}_i=\upsilon_i$. Let $\upsilon^{\prime}=\upsilon^{\prime}_1\ldots \upsilon^{\prime}_r$. This map $\upsilon \mapsto \upsilon^{\prime}$ is one-to-one. 
			If $\upsilon^{\prime}_i\neq \upsilon_i$ for at least $\frac{1}{500}$ of the $i$'s, then $\upsilon^{\prime}\notin \mathcal{S}^{\prime}$. So there is a collection $\mathcal{S} \subset \mathcal{S}^{\prime}$ with $|\mathcal{S}|>\frac{498}{500}2^{h_{n+1}}$ for which $\upsilon^{\prime}_i = \upsilon_i$ for all but $\frac{1}{500}$ of the $i$'s.
			
			Let $\upsilon \in \mathcal{S}$ and $(\overline{w},\overline{\upsilon})=(\overline{P}_{1},\overline{\upsilon}_1)\dots(\overline{P}_{r},\overline{\upsilon}_r)$ with $\overline{\upsilon}_i \in \{0,1\}^{|\overline{P}_i|}$
			be the decomposition corresponding to $(w,\upsilon)=(P_1,\upsilon_1)\dots (P_r,\upsilon_r)$ under a best possible $\overline{f}$-match between $\phi_N(w,\upsilon)$ and $(\overline{w},\overline{\upsilon})$.
			According to Fact \ref{fact:substring_matching} the proportion of
			such substrings $(P_i,\upsilon_i)$ with 
			\begin{equation}\label{eq:GoodMatch}
				\overline{f}\left(\phi_N(P_{i},\upsilon_i),(\overline{P}_{i},\overline{\upsilon}_i)\right)<10^{-17}\alpha_{s}^{6}
			\end{equation}
			is at least $\frac{499}{500}$. Thus, for a proportion of at least $\frac{249}{250}$ of $i$'s we have that \eqref{eq:GoodMatch} holds and $\upsilon^{\prime}_i=\upsilon_i$.
			
			Choose $i$ such that \eqref{eq:GoodMatch} holds and $\upsilon^{\prime}_i=\upsilon_i$. After applying $\phi_N$ to any substring of $(P_i,\upsilon_i)$ of length at least $\frac{\alpha^4_s}{10^{12}}|P_i|$ the $\overline{f}$ distance to a corresponding part of $(\overline{P}_i,\overline{\upsilon}_i)$ must be less than $\frac{\alpha^2_s}{2000}$. Since $\upsilon_i=\upsilon^{\prime}_i$ there exists a substring $(\tilde{P}_i,\tilde{\upsilon}_i)$ of $(P_i,\upsilon_i)$ such that $|\tilde{P}_i|>(1-\frac{\alpha^4_s}{10^{12}})|P_i|$ and $\phi_N(\tilde{P}_i,\tilde{\upsilon}_i)$ is matched to a string $(\hat{P}_i,\hat{\upsilon}_i)$ in $(\overline{P}_i,\overline{\upsilon}_i)$, where $\hat{P}_i$ is of the same pattern type as $P_i$, that is, $[\hat{P}_{i}]_{s}$
			must lie in $\eta_{g_{s}}[P_{i}]_{s}$ for $g_s=g_s(w,\upsilon)$. For each of those $(P_i,\upsilon_i)$ we have 
			\begin{equation}\label{eq:GoodMatch2}
				\overline{f}\left(\phi_N(P_{i},\upsilon_i),(\hat{P}_{i},\hat{\upsilon}_i)\right)<\frac{\alpha^4_{s}}{10^{11}}.
			\end{equation}
			We deduce from \eqref{eq:GoodMatch2} that for a proportion of at least $\frac{998}{1000}$ of the shaded $n$-blocks $(w_{i,l},\upsilon_{i,l})$
			occurring as substitution instances in $(P_{i},\upsilon_i)$ we have $$\overline{f}\left(\phi_{N}(w_{i,l},\upsilon_{i,l}),(\hat{A}_{i,l},\hat{\upsilon}_{i,l})\right)<\frac{\alpha^4_{s}}{2\cdot 10^7}$$ for a string $(\hat{A}_{i,l},\hat{\upsilon}_{i,l})$ in $(\hat{P}_i,\hat{\upsilon}_i)$. 
			
			Overall, for a proportion of at least $\frac{124}{125}$ of the shaded $n$-blocks $(w_{i,l},\upsilon_{i,l})$
			occurring as substitution instances in $(w,\upsilon)$ we have $\overline{f}\left(\phi_{N}(w_{i,l},\upsilon_{i,l}),(\hat{A}_{i,l},\hat{\upsilon}_{i,l})\right)<\frac{\alpha^4_{s}}{2\cdot 10^7}$ for a string $(\hat{A}_{i,l},\hat{\upsilon}_{i,l})$ in $(\hat{P}_i,\hat{\upsilon}_i)$. Since $\upsilon \in \mathcal{S}$ we obtain from \eqref{eq:EquiShading} that for a proportion of at least $\frac{8}{10}$ of $n$-blocks $w_{i,l}\in \mathcal{W}_n$ we have for a proportion of at least $\frac{9}{10}$ of all shading sequences $\upsilon_{i,l} \in \{0,1\}^{h_n}$ that 
			$$\overline{f}\left(\phi_{N}(w_{i,l},\upsilon_{i,l}),(\hat{A}_{i,l},\hat{\upsilon}_{i,l})\right)<\frac{\alpha^4_{s}}{2\cdot 10^7}=\frac{\alpha_{s}}{200} \cdot \frac{\alpha^3_{s}}{10^5}$$ 
			for a string $(\hat{A}_{i,l},\hat{\upsilon}_{i,l})$ in $(\hat{P}_i,\hat{\upsilon}_i)$. Then Lemma \ref{lem:groupelement} implies that for these  $n$-blocks $w_{i,l}$ for a proportion of at least $\frac{8}{10}$ of all shading sequences $\upsilon_{i,l} \in \{0,1\}^{h_n}$ there is exactly one $\overline{w}_{i,l}\in rev(\mathcal{W}_{n})$
			with $|\hat{A}_{i,l}\cap\overline{w}_{i,l}|\geq(1-\frac{\alpha^3_{s}}{10^5})h_n$
			and $\overline{w}_{i,l}$ must be of the form $\left[\overline{w}_{i,l}\right]_{s}=\eta_{h_{i,l}}\left[w_{i,l}\right]_{s}$
			for a unique $h_{i,l}\in G_{s}$, which is of odd parity. As in \eqref{eq:GforN}, we also get $$\overline{f}\left(\phi_{N}(w_{i,l},\upsilon_{i,l}),(\overline{w}_{i,l},\overline{\upsilon}_{i,l})\right)<\frac{\alpha^4_{s}}{2\cdot 10^7}+\frac{\alpha^3_{s}}{10^5}<\frac{\alpha_{s}}{4}$$ for some shading sequence $\overline{\upsilon}_{i,l} \in \{0,1\}^{h_n}$. Since $[\hat{P}_i]_s$ is contained in $\eta_{g_s}([P_i]_s)$ and the $s$-equivalence classes	
			within the building tuple of $P_{i}$ have disjoint $G_{s}$ orbits by step~(\ref{item:step1}) of the construction in Section~\ref{sec:review},
			we obtain that the $h_{i,l}$ are equal to $g_s$. Thus, for $g_s = g_s(w,\upsilon)$ and for a proportion of at least $\frac{8}{10}$ of $n$-blocks $\tilde{w}\in \mathcal{W}_n$ we have for a proportion of at least $\frac{8}{10}$ of all shading sequences $\tilde{\upsilon} \in \{0,1\}^{h_n}$ that there is $\hat{w}\in \rev(\mathcal{W}_{n})$
			with $\left[\hat{w}\right]_{s}=\eta_{g_{s}}\left[\tilde{w}\right]_{s}$
			and $\overline{f}\left(\phi_{N}(\tilde{w},\tilde{\upsilon}),(\hat{w},\hat{\upsilon})\right)<\frac{\alpha_{s}}{4}$ for some shading sequence $\hat{\upsilon} \in \{0,1\}^{h_{n}}$.
		\end{proof}
		
		\begin{proof}[Proof of Claim \ref{claim:UniqueG}]
			Suppose now that there are words $w,w^{\prime}\in\mathcal{W}^{\prime}_{n+1}$ and shadings $\upsilon \in \mathcal{S}(w)$, $\upsilon^{\prime} \in \mathcal{S}(w^{\prime})$
			such that the statement of Claim~\ref{claim:UniqueGpre} holds for elements $g_s(w,\upsilon),g_s(w^{\prime},\upsilon^{\prime})\in G_{s}$ with $g_s(w,\upsilon)\neq g_s(w^{\prime},\upsilon^{\prime})$. Then for a proportion of at least $\frac{6}{10}$ of all $n$-words $\tilde{w} \in \mathcal{W}_{n}$ and a proportion of at least $\frac{6}{10}$ of all shading sequences $\tilde{\upsilon} \in \{0,1\}^{h_n}$ the conclusion of Claim~\ref{claim:UniqueGpre} holds simultaneously for $g_s(w,\upsilon)$ and $g_s(w^{\prime},\upsilon^{\prime})$. In particular, there exist $\tilde{w}\in \mathcal{W}_n$, $\tilde{\upsilon}\in \{0,1\}^{h_n}$, $\hat{w},\hat{w}^{\prime}\in \rev(\mathcal{W}_n)$, $\hat{\upsilon},\hat{\upsilon}^{\prime}\in \{0,1\}^{h_n}$ with $[\hat{w}]_s=\eta_{g_s(w,\upsilon)}[\tilde{w}]_s$, $[\hat{w}^{\prime}]_s=\eta_{g_s(w^{\prime},\upsilon^{\prime})}[\tilde{w}]_s$, $\overline{f}\left(\phi_{N}(\tilde{w},\tilde{\upsilon}),(\hat{w},\hat{\upsilon})\right)<\frac{\alpha_s}{4}$, and $\overline{f}\left(\phi_{N}(\tilde{w},\tilde{\upsilon}),(\hat{w}^{\prime},\hat{\upsilon}^{\prime})\right)<\frac{\alpha_s}{4}$. By the triangle inequality this implies $$\overline{f}\left( (\hat{w},\hat{\upsilon}),(\hat{w}^{\prime},\hat{\upsilon}^{\prime})\right)<\frac{\alpha_s}{2}$$ and therefore $\overline{f}(\hat{w},\hat{w}^{\prime})<\frac{\alpha_s}{2}$.
			This contradicts $\left[\hat{w}\right]_{s}=\eta_{g_s(w,\upsilon)}\left[\tilde{w}\right]_{s}\neq\eta_{g_s(w^{\prime},\upsilon^{\prime})}\left[\tilde{w}\right]_{s}=\left[\hat{w}^{\prime}\right]_{s}$
			and Lemma \ref{lem:fDist}. %\cite[Proposition 56]{GK3}. 
			Hence, there is a unique group
			element $g_{s}\in G_{s}$ working for all $(n+1)$-words $w$ in $\mathcal{W}^{\prime}_{n+1}$ and all shadings in $\mathcal{S}(w)$.
		\end{proof}
		
		In the next step in the proof of Lemma~\ref{lem:codeGroup}, we check that the same group element $g_{s}\in G_{s}$
		works for all $\phi_{N}$, $N\geq N(s)$. Let $k\in\mathbb{Z}^{+}$.
		For $n \geq n(s,N,k),$ property \ref{item:C2} guarantees that for every $w\in\mathcal{W}^{\prime}_{n},$ for a proportion of at least $1-\sqrt{\gamma_s}$ of possible shading sequences $\upsilon \in \{0,1\}^{h_n}$, we have $\overline{f}\left(\phi_{N}(w,\upsilon),\phi_{N+k}(w,\upsilon)\right)<\alpha^4_s$. On the one hand, let $g_{s}$ be
		the element in $G_{s}$ such that for every $w\in \mathcal{W}^{\prime}_{n},$ for a proportion of at least $\frac{99}{100}$ of possible shading sequences $\upsilon \in \{0,1\}^{h_{n}}$, there exist $\overline{w}\in \rev(\mathcal{W}_{n})$
		with $\left[\overline{w}\right]_{s}=\eta_{g_{s}}\left[w\right]_{s}$
		and some shading sequence $\overline{\upsilon} \in \{0,1\}^{h_{n}}$ such that $\overline{f}\left(\phi_{N}(w,\upsilon),(\overline{w},\overline{\upsilon})\right)<\frac{\alpha_{s}}{4}$.
		On the other hand, let $h\in G_{s}$ such that for every $w\in\mathcal{W}^{\prime}_{n},$
		for a proportion of at least $\frac{99}{100}$ of possible shading sequences $\upsilon \in \{0,1\}^{h_{n}}$, there exist  $\overline{w}^{\prime}\in \rev(\mathcal{W}_{n})$ with $\left[\overline{w}^{\prime}\right]_{s}=\eta_h\left[w\right]_{s}$
		and some shading sequence $\overline{\upsilon}^{\prime} \in \{0,1\}^{h_{n}}$ such that $\overline{f}\left(\phi_{N+k}(w,\upsilon),(\overline{w}^{\prime},\overline{\upsilon}^{\prime})\right)<\frac{\alpha_{s}}{4}$.
		Assume $g_{s}\neq h$. Then $\left[\overline{w}^{\prime}\right]_{s}=\eta_h[w]_{s}\neq\eta_{g_{s}}\left[w\right]_{s}=\left[\overline{w}\right]_{s}$.
		However, for a proportion of at least $\frac{97}{100}$ of shading sequences $\upsilon \in \{0,1\}^{h_n}$, we have
		\begin{align*}
			&\overline{f}\left((\overline{w},\overline{\upsilon}),(\overline{w}^{\prime},\overline{\upsilon}^{\prime})\right)\\
			\leq & \overline{f}\left((\overline{w},\overline{\upsilon}),\phi_{N}(w,\upsilon)\right)+\overline{f}\left(\phi_{N}(w,\upsilon),\phi_{N+k}(w,\upsilon)\right)+\overline{f}\left(\phi_{N+k}(w,\upsilon),(\overline{w}^{\prime},\overline{\upsilon}^{\prime})\right) \\
			< & \frac{\alpha_{s}}{4}+ \alpha^4_s + \frac{\alpha_{s}}{4} <\alpha_{s}.
		\end{align*}
		This contradicts Lemma \ref{lem:fDist}. %\cite[Proposition 56]{GK3}. 
		Thus, we conclude that $g_{s}\in G_{s}$
		works for all $\phi_{N}$ with $N\geq N(s)$.
		
		Suppose that for some $s\in\mathbb{N}$, we have $g_{s}^{\prime}\coloneqq\rho_{s+1,s}(g_{s+1})\neq g_{s}$.
		Let $N\geq N(s+1)$. Let $n$ be sufficiently large such that $g^{\prime}_s\in G^n_s$ and Claim~\ref{claim:UniqueG} holds. Let $w\in\mathcal{W}^{\prime}_{n}$.
		Then $\eta_{g_{s}^{\prime}}\left[w\right]_{s}\neq\eta_{g_{s}}\left[w\right]_{s},$ and by applying
		Claim~\ref{claim:UniqueG}, we obtain for a proportion of at least $\frac{99}{100}$ of possible shading sequences $\upsilon \in \{0,1\}^{h_{n}}$, the existence of $\overline{w}_1\in \rev(\mathcal{W}_{n})$
		with $\left[\overline{w}_1\right]_{s}=\eta_{g_{s}}\left[w\right]_{s}$
		and some shading sequence $\overline{\upsilon}_1 \in \{0,1\}^{h_{n}}$ such that
		\[
		\overline{f}\left(\phi_{N}(w,\upsilon),(\overline{w}_1,\overline{\upsilon}_1)\right)<\frac{\alpha_{s}}{4}.
		\]
		Similarly, Claim \ref{claim:UniqueG} yields for a proportion of at least $\frac{99}{100}$ of possible shading sequences $\upsilon \in \{0,1\}^{h_{n}}$, the existence of $\overline{w}_2\in \rev(\mathcal{W}_{n})$
		with $\left[\overline{w}_2\right]_{s+1}=\eta_{g_{s+1}}\left[w\right]_{s+1}$
		and some shading sequence $\overline{\upsilon}_2 \in \{0,1\}^{h_{n}}$ such that
		\[
		\overline{f}\left(\phi_{N}(w,\upsilon),(\overline{w}_2,\overline{\upsilon}_2)\right)<\frac{\alpha_{s+1}}{4}.
		\]
		Since the $G_{s+1}^{n}$ action is subordinate to the $G_{s}^{n}$
		action on $\mathcal{W}_{n}/\mathcal{Q}_{s}^{n}$ by specification
		\ref{item:A7}, we see that $\left[\overline{w}_{2}\right]_{s+1}=\eta_{g_{s+1}}\left[w\right]_{s+1}$
		lies in $\eta_{g_{s}^{\prime}}\left[w\right]_{s}$. Hence, $\left[\overline{w}_{2}\right]_{s} = \eta_{g_{s}^{\prime}}\left[w\right]_{s} \neq\eta_{g_{s}}\left[w\right]_{s}=\left[\overline{w}_{1}\right]_{s}$,
		which implies $\overline{f}\left((\overline{w}_{1},\overline{\upsilon}_1),(\overline{w}_{2},\overline{\upsilon}_2)\right)>\alpha_{s}$
		by Lemma \ref{lem:fDist}.  %\cite[Proposition 56]{GK3}. 
		On the other hand, for a proportion of at least $\frac{98}{100}$ of possible shading sequences $\upsilon \in \{0,1\}^{h_{n}}$, we have
		\begin{align*}
			&\overline{f}\left((\overline{w}_{1},\overline{\upsilon}_1),(\overline{w}_{2},\overline{\upsilon}_2)\right)\\
			\leq & \overline{f}\left((\overline{w}_{1},\overline{\upsilon}_1),\phi_{N}(w,\upsilon)\right)+\overline{f}\left(\phi_{N}(w,\upsilon),(\overline{w}_{2},\overline{\upsilon}_2)\right) \\
			<& \frac{\alpha_{s}}{4}+\frac{\alpha_{s+1}}{4}<\frac{\alpha_{s}}{2}.
		\end{align*}
		This contradiction shows that $(g_{s})_{s\in\mathbb{N}}$ must satisfy
		$g_{s}=\rho_{s+1,s}(g_{s+1})$ for all $s$.
	\end{proof}
	
	\begin{proof}[Proof of Part 1 (b) in Theorem \ref{mainthm:Kproof}]
		If the positive entropy systems $\mathbb{S}=\Upsilon(\mathcal{T})$ and $\mathbb{S}^{-1}\cong \rev(\mathbb{S})$ are Kakutani equivalent, then they have to be evenly equivalent. Then the conclusion of Lemma~\ref{lem:codeGroup} implies that the
		tree $\mathcal{T}$ has an infinite branch.
	\end{proof}

	\section{Proof of Theorem \ref{mainthm:Kdiffeo}} \label{sec:smooth}
	
	\subsection{Construction of $K$-diffeomorphisms $\Phi_i(\mathcal{T})$} \label{subsec:ConstrSmooth}
	As in Section 
	\ref{subsec:smoothK}, $\mathbb{T}^n=(\mathbb{R}/\mathbb{Z})^n$ and $\lambda_n$ denotes
	Lebesgue measure on $\mathbb{T}^n$. To build the $K$-diffeomorphisms $\Phi_i(\mathcal{T})\in $ Diff$^{\infty}(\mathbb{T}^5,\lambda_5),$ $i=1,2$ in 
	Theorem \ref{thm:Kdiffeo}, we follow an approach in \cite{Ka80}. We start with a $C^{\infty}$
	Anosov diffeomorphism $g$ of $\mathbb{T}^{2}$ that preserves $\lambda_{2}$. Note that $g$ is ergodic by \cite{A67} and therefore topologically transitive.
	We can choose a $C^{\infty}$ function $\varphi:\mathbb{T}^2\mapsto \mathbb{R}^+$ with $||\mathrm{d}\varphi||$ close to 0
	so that there exists a closed orbit $x,gx,\dots,g^{n-1}x,$ for $g$ such that 
	\begin{equation}\label{eq:KatokCond}
		\sum_{j=0}^{n-1}\varphi(g^{j}(x))\ne n\int_{\mathbb{T}^2}\varphi \ \mathrm{d} \lambda_{2}.
	\end{equation}
	This condition implies that the cohomology equation 
	\[
	\psi(gx)-\psi(x)=\varphi(x)-\int_{\mathbb{T}^2}\varphi\ \mathrm{d}\lambda_2
	\]
	has no continuous solution $\psi$, that is, $\varphi$ is not cohomologous to a constant.
	In order for the Bernoulli shift $\mathbb{B}$ in Proposition~\ref{prop:Part1} below to be the $(1/2,1/2)$ Bernoulli shift, we also require $\varphi$ to be chosen so that 
	\begin{equation}\label{eq:EntropyCond}
		\int_{\mathbb{T}^2}\varphi\, \mathrm{d}\lambda_2=\frac{h_{\lambda_{2}}(g)}{\log2},
	\end{equation}
	where $h_{\lambda_2}(g)$ denotes the entropy of $g$ with respect to $\lambda_2.$ 
	As we shall see below, the required bound on $||\mathrm{d}\varphi||$ 
	depends only on $g$. Therefore, once $g$ is chosen, we let $\varphi$ be a small
	$C^1$ perturbation of the constant function on the right side of (\ref{eq:EntropyCond}) 
	that achieves (\ref{eq:KatokCond}) by making a small change in a neighborhood of $x$ not including
	$gx,\dots,g^{n-1}x$, while maintaining the value of the integral in (\ref{eq:EntropyCond}).
	
	If $F^s:\mathcal{T}rees\to$ Diff$^{\infty}(\mathbb{T}^{2},\lambda_{2})$
	is as in \cite[section 10]{GK3}, and $T=F^s(\mathcal{T})$  for some
	$\mathcal{T}\in\mathcal{T}rees,$ let $S^{t},$ $t\in\mathbb{R},$
	be the special flow over $T$ with roof function identically equal
	to $1$. As usual, $(y,1)$ is identified with $(Ty,0).$ That is, $S^{t}$
	is defined on $(\mathbb{T}^{2}\times[0,1])/((y,1)\sim(Ty,0)).$ Let
	$T_{g,\varphi}$ be the skew product defined by $T_{g,\varphi}(x,(y,t))=(gx,S^{\varphi(x)}(y,t))$,
	for $x,y\in\mathbb{T}^{2},$ $0\le t\le1,$ with $(x,y,1)$ identified
	with $(x,Ty,0)$. Then $T_{g,\varphi}$ preserves $\lambda_{5}$,
	and according to \cite{Ka80}, $T_{g,\varphi}$ is a $K$-automorphism.
	However, the identification on $\mathbb{T}^{2}\times\mathbb{T}^{2}\times[0,1]$
	is $(x,Ty,0)\sim(x,y,1)$, which depends on $T.$ 
	
	To change $T_{g,\varphi}$ to a $K$-automorphism $\widetilde{T}_{g,\varphi}$
	that is a diffeomorphism of $\mathbb{T}^{5}$ with the usual identification,
	we will apply Proposition~\ref{prop:homotopy}.
	
	\begin{prop}\label{prop:homotopy}
		For each $\mathcal{T}\in\mathcal{T}rees$, the map $T=F^s(\mathcal{T})\in \text{Diff}^{\,\infty}(\mathbb{T}^2,\lambda_2)$ constructed in 
		\cite[section 10]{GK3} is $C^{\infty}$ isotopic to the identity through diffeomorphisms that preserve $\lambda_2$.
		Moreover, the map taking $\mathcal{T}\in\mathcal{T}rees$ to the isotopy from Id to $T=F^s(\mathcal{T})$ can be chosen to be a continuous function from $\mathcal{T}rees$ to $C^{\infty}(\mathbb{T}^2\times [0,1],\mathbb{T}^2).$
	\end{prop}
	
	\begin{proof}
		In the application of the Anosov-Katok method that is used in \cite[section 10]{GK3}, we may construct each $T$ in the image of $F^s$ to be as close as we like in the $C^{\infty}$ topology to a fixed rotation $R_{\alpha_0}(y_1,y_2)=(y_1+\alpha_0,y_2)$ of the first coordinate in $\mathbb{T}^2$ (compare with \cite[Remark 39]{FW1}). By \cite[Chapter 2, Theorem 1.7]{Hi76} any $C^{\infty}$ map that is sufficiently close in the $C^1$ topology to a given $C^{\infty}$ diffeomorphism is itself a diffeomorphism. Therefore we may assume that for each $u\in [0,1]$ and each $T$ in the image of $F^s$ the map $uT+(1-u)R_{\alpha_0}$ is a diffeomorphism. Let $\beta:[0,1]\to [0,1]$ be a monotone increasing $C^{\infty}$ function such that $\beta(t)=p$ in a neighborhood of $p$ for $p=0, 1/2$, and $1$. The requirement that $\beta=1/2$ in a neighborhood of $1/2$ is needed to make the isotopy $J_t$ below be $C^{\infty}$. The requirement that $\beta(t)=p$ in neighborhoods of $p$ for $p=0,1$ is used later in our construction to make a smooth identification between $\mathbb{T}^2\times \{0\}$ and $\mathbb{T}^2\times \{1\}$. Let $J_t:\mathbb{T}^2\to \mathbb{T}^2$ for $0\le t\le 1$ be the $C^{\infty}$ isotopy from Id to $T$ given by $J_t=R_{2\beta(t)\alpha_0}$, for $0\le t\le 1/2$; $(2-2\beta(t))R_{\alpha_0}+2(\beta(t)-1/2)T$, for $1/2<t\le 1.$ 
		
		Each map $J_t$ is a diffeomorphism of $\mathbb{T}^2$, but it need not preserve $\lambda_2$ for those values of $t$ that are in $(1/2,1)$ where $\beta(t)\ne 1/2,1.$ We will compose $J_t$ with another diffeomorphism $L_t:\mathbb{T}^2\to\mathbb{T}^2$ such that $L_t$ is the identity for $t$ in neighborhoods of $0$ and $1$, and $(L_t\circ J_t)_{*}\lambda_2=\lambda_2.$ This step is done using Moser's Theorem \cite{Mo65}, but we will show how to construct $L_t$ explicitly in terms of Fourier series in order to demonstrate the $C^{\infty}$ dependence of $L_t$ on $t$ along with the continuous dependence of $L_t\circ J_t$ in C$^{\infty}(\mathbb{T}^2\times [0,1],\mathbb{T}^2)$ on $\mathcal{T}$. 
		
		We follow Moser's argument as presented in \cite[Theorem 5.1.27]{HK}.
		Fix $t\in[0,1]$ and use coordinates $y=(y_1,y_2)$ on $\mathbb{T}^2$. Let $\Omega_0$ and $\Omega_1$ be the two-forms representing $(J_t)_{*}\lambda_2$ and $\lambda_2$, respectively. Let $\Omega^{\prime}:=\Omega_1-\Omega_0.$ Since $\int_{\mathbb{T}^2}\Omega^{\prime}=0$, $\Omega^{\prime}=$ d$\Theta$ for some one-form $\Theta$. (This follows from the fact that the top deRham cohomology group is $\mathbb{R}.)$ To find such a one-form $\Theta$, we write $$\Omega^{\prime}=f(y_1,y_2)\ \mathrm{d}y_1\land\mathrm{d}y_2.$$
		Since $f$ is a real-valued $C^\infty$ function with $\int_{\mathbb{T}^2} f\ \mathrm{d}\lambda_2=0, $ we have
		$$f(y_1,y_2)=\sum_{(n,m)\in(\mathbb{Z}\times\mathbb{Z})\setminus \{(0,0)\}} a_{(n,m)}\mathrm{exp}(i(ny_1+my_2)),$$
		where $\overline{a}_{-(n,m)}=a_{(n,m)}$ and the Fourier coefficients $a_{(n,m)}$ decay faster than any polynomial rate, that is, $|a_{(n,m)}|\cdot ||(n,m)||^k\to 0$ as $||(n,m)||\to\infty$ for every positive integer $k$ (see \cite[Theorem 2.6]{Fo92}). Let $\Theta=p(y_1,y_2)\mathrm{d}y_1+q(y_1,y_2)\mathrm{d}y_2.$ We need
		\begin{equation}\label{eq:derivative}
			\frac{\partial q}{\partial y_1}-\frac{\partial p}{\partial y_2}=f(y_1,y_2).
		\end{equation}
		We choose 
		$$p(y_1,y_2)=p(y_2)=-\mathrm{Re}\Big[\sum_{m\in\mathbb{Z},m\ne 0}\frac{a_{(0,m)}}{im}\mathrm{exp}(imy_2)\Big],$$
		and $$ q(y_1,y_2)=\mathrm{Re}\Big[\sum_{n,m\in\mathbb{Z},n\ne 0}\frac{a_{(n,m)}}{in}\mathrm{exp}(i(ny_1+my_2))\Big].$$
		Then the Fourier coefficients of $p$ and $q$ also decay faster than any polynomial rate, and therefore $p$ and $q$ are $C^{\infty}$ functions. Clearly (\ref{eq:derivative}) is satisfied.
		
		The rest of the Moser construction proceeds as in \cite{HK}. With the explicit choice of  $\Theta$ given above, the map $L_t$ with $(L_t)_{*}(J_t)_{*}\lambda_2=\lambda_2$ that is obtained depends continuously in the $C^{\infty}$ topology on $t\in [0,1]$. Moreover, because $T=F^s(\mathcal{T})$ depends continuously in the $C^{\infty}$ topology on $\mathcal{T}$, so does the isotopy $I_t:=L_t\circ J_t.$
	\end{proof}
	
	Let $I:\mathbb{T}^{2}\times[0,1]\to\mathbb{T}^{2}$
	be the isotopy constructed in the proof of Proposition~\ref{prop:homotopy}. Then $I_{u}(y)=I(y,u)$ satisfies $I_{u}=$Id for $u$ in a neighborhood of $0$, $I_{u}=T$ for $u$ in a neighborhood of 1,
	and $I_{u}\in$ Diff$^{\infty}(\mathbb{T}^{2},\lambda_{2})$, for
	$0\le u\le1.$ Let $\widetilde{S}^{t}$ be the $C^{\infty}$ flow on $\mathbb{T}^{3}=(\mathbb{T}^{2}\times[0,1])/((y,1)\sim(y,0)),$
	defined by $\widetilde{S}^{t}(z,u)=(I_{u+t}(I_{u}^{-1}(z)),u+t),$
	for nonnegative $u$ and $t$ with $u+t\le1.$ For $u+t=1,$ $\widetilde{S}(z,u)=(T(I_{u}^{-1}(z)),1)\sim(T(I_{u}^{-1}(z)),0).$
	Define 
	\begin{align*}
		& \psi:(\mathbb{T}^{2}\times[0,1])/((Ty,0)\sim(y,1))\to\mathbb{T}^{3}=(\mathbb{T}^{2}\times[0,1])/((y,0)\sim(y,1))\\
		\text{by\ \ \ }& \psi(y,u)=(I_{u}(y),u),\text{\  for\ \ }y\in\mathbb{T}^2,\  0\le u\le1.
	\end{align*}
	Since $\psi(y,1)=(T(y),1)\sim(T(y),0)=\psi(T(y),0),$ $\psi$ is well-defined.
	Moreover, $\psi$ preserves $\lambda_{3},$ because each $I_{u}$
	preserves $\lambda_{2}.$ For nonnegative $u$ and $t$ with $u+t\le 1$, we have 
	\[
	\psi(S^{t}(y,u))=\psi(y,u+t)=(I_{u+t}(y),u+t),
	\]
	and
	\[
	\widetilde{S}^{t}(\psi(y,u))=\widetilde{S}^t(I_{u}(y),u)=(I_{u+t}(I_{u}^{-1}(I_{u}(y)),u+t)=(I_{u+t}(y),u+t).
	\]
	Therefore $\psi$ gives an isomorphism between the flows $S^{t}$
	and $\widetilde{S}^{t}$.\footnote{The same method can be used to obtain anti-classification results for isomorphism and Kakutani equivalence of smooth ergodic measure-preserving $\R$-flows on $\T^3$ from the anti-classification results for smooth diffeomorphisms of $\T^2$ in \cite{FW3} and \cite{GK3} via a suspension.} Consequently, the transformation $\widetilde{T}_{g,\varphi}$
	defined by $\widetilde{T}_{g,\varphi}(x,y,t)=(gx,\widetilde{S}^{\varphi(x)}(y,t))$
	is isomorphic to $T_{g,\varphi},$ and $\widetilde{T}_{g,\varphi}\in$
	Diff$^{\infty}(\mathbb{T}^{5},\lambda_{5}).$ We now state the following concretization of Theorem~\ref{mainthm:Kdiffeo}.
	
	\begin{thm}\label{thm:Kdiffeo}
		For a fixed choice of $g$ and $\varphi$
		as above and the continuous map $F^s:\mathcal{T}rees\to\mathcal{E}\cap\text{Diff}^{\,\infty}(\mathbb{T}^{2},\lambda_{2})$
		constructed in \cite[section~10]{GK3}, let
		$\Phi_{j}:\mathcal{T}rees\to\mathcal{K} \cap  \text{Diff}^{\,\infty}(\mathbb{T}^{5},\lambda_{5}),$
		$j=1,2,$ be the maps defined by $\Phi_{1}(\mathcal{T})=\widetilde{T}_{g,\varphi},$
		and $\Phi_{2}(\mathcal{T})=\widetilde{(T^{-1}})_{g,\varphi},$ where
		$T=F^s(\mathcal{T}).$ If we use the topology on $\mathcal{K} \cap \text{Diff}^{\,\infty}(\mathbb{T}^{5},\lambda_{5})$ induced by the $C^{\infty}$
		topology on $\text{Diff}^{\,\infty}(\mathbb{T}^{5},\lambda_{5})$, then $\Phi_{1}$
		and $\Phi_{2}$ are continuous. Moreover, $\Phi_{1}(\mathcal{T})$
		and $\Phi_{2}(\mathcal{T})$ satisfy the properties of Theorem \ref{mainthm:Kdiffeo}.
	\end{thm}
	
	\begin{proof}
		The continuity of $\Phi_{1}$ and $\Phi_{2}$ follows from 
		Proposition \ref{prop:homotopy} and the method of construction of $\widetilde{T}_{g,\varphi}$
		from $T.$ Since $T_{g,\varphi}$ and $\widetilde{T}_{g,\varphi}$
		are isomorphic, and $(T^{-1})_{g,\varphi}$ and $(\widetilde{T^{-1}})_{g,\varphi}$
		are isomorphic, to prove part (a) of Theorem~\ref{mainthm:Kdiffeo}, it suffices to prove
		that if $\mathcal{T}$ has an infinite branch, then $T_{g,\varphi}$
		and $(T^{-1})_{g,\varphi}$ are isomorphic. This is immediate from
		the fact that $T$ and $T^{-1}$ are isomorphic, which follows from
		part~(1) of Proposition 24 and its generalization in Theorem 25 in
		\cite{GK3}. 
		
		The proof of part (b) of Theorem~\ref{mainthm:Kdiffeo} follows from Propositions~\ref{prop:Part1} and \ref{prop:Part2} below.
	\end{proof}
	
	The following
	Proposition~\ref{prop:Part1} is based partly on an argument of Thouvenot presented
	in \cite[section 3]{KRV}.
	
	\begin{prop}\label{prop:Part1}
		If $T$ is an ergodic automorphism of $\mathbb{T}^2$ preserving $\lambda_2$, then $T_{g,\varphi}$
		is Kakutani equivalent to $T\times \mathbb{B},$ where $\mathbb{B}$ is the $(1/2,1/2)$
		Bernoulli shift on $\{0,1\}^{\mathbb{Z}}.$ 
	\end{prop}
	
	\begin{proof}
		Let 
		\[
		X_{g,\varphi}=\{(x,t)\in\mathbb{T}^{2}\times\mathbb{R}:0\le t\le\varphi(x)\}/((x,\varphi(x))\sim(g(x),0)),
		\]
		and let $S_{1}^{t}$ be the flow on $X_{g,\varphi}$ built as the
		flow under $\varphi$ over the transformation $g.$ Let 
		\[
		X_{T}=\{(y,s)\in\mathbb{T}^{2}\times\mathbb{R}:0\le s\le1\}/((y,1)\sim(T(y),0)),
		\]
		and let $S_{2}^{t}=S^{t}$ be the flow on $X_{T}$ built as the flow
		under the function identically equal to 1 over the transformation
		$T.$ Define $U^{r}$ as the diagonal flow on $X_{g,\varphi}\times X_{T}$
		given by $U^{r}(x,t,y,s)=(S_{1}^{r}(x,t),S_{2}^{r}(y,s)).$ Consider
		two canonical sections of this flow:
		\[
		X_{1}=\{((x,0),(y,s)):x,y\in\mathbb{T}^{2},0\le s\le1\}
		\]
		and 
		\[
		X_{2}=\{((x,t),(y,0)):x,y\in\mathbb{T}^{2},0\le t\le\varphi(x)\}.
		\]
		The first return map for $X_{1}$ is $T_{g,\varphi},$ and the first
		return map to $X_{2}$ is $(x,t,y)\mapsto(S_{1}^{1}(x,t),T(y)).$
		This is a direct product of the time one map $S_{1}^{1}$ of the flow
		$S^{t}_1$ on $X_{g,\varphi}$ and the transformation $T.$ 
		
		If $\norm{\mathrm{d}\varphi}$ is sufficiently small compared to the expansion
		factor of $g$ in the unstable direction, then the expansion effect
		of the flow $S_{1}^{t}$ in directions close to the unstable direction
		in the base of $X_{g,\varphi}$ overcomes the skewing effect in the
		flow direction, and similarly in the stable direction.
		Thus for 
		sufficiently small $\norm{\mathrm{d}\varphi},$ depending only of $g$, the flow $S^{t}_1$ is
		Anosov. Since $g$ is transitive, so is $S^{t}_1.$
		
		According to Anosov's alternative \cite{A67} (see also \cite{Do98}),
		an Anosov flow is a suspension with a constant roof function over
		an Anosov diffeomorphism or each strong unstable manifold and each
		strong stable manifold is dense in the space. The first alternative
		is not true for $S^{t}_1$ because $\varphi$ is chosen so that it is
		not cohomologous to a constant. Therefore the second alternative holds
		for $S^{t}_1$, and according to a theorem for transitive Anosov flows
		due to Bunimovi\v{c} \cite[Theorem~2]{Bu74}, this implies that $S_{1}^{t}$
		is Bernoulli. In particular, the time one map $S_{1}^{1}$ is Bernoulli. 
		
		By Abramov's formula, the entropy of $S_{1}^{1}$ is $h_{\lambda_{2}}(g)/\int_{\mathbb{T}^2}\varphi\ d\lambda_{2}.$
		Thus, by our assumption~\eqref{eq:EntropyCond}, $S_{1}^{1}$ has entropy $\log2$, which is the entropy of $\mathbb{B}$.
		Therefore the first return map to $X_{2}$ is isomorphic to $T\times \mathbb{B}$.
		Since the first return maps to two cross-sections of an ergodic flow
		are Kakutani equivalent \cite{Kt43}, $T_{g,\varphi}$ is Kakutani
		equivalent to $T\times \mathbb{B}$.
	\end{proof}
	
	To complete the proof of part (b) of Theorem \ref{mainthm:Kdiffeo}, it now suffices to prove
	Proposition~\ref{prop:Part2} (see Section~\ref{subsec:ProofPropPart2}). Proposition~\ref{prop:Part2} is the main step in the proof
	of Theorem~\ref{mainthm:Kdiffeo}. 
	
	\begin{prop} \label{prop:Part2}
		If $T=F^s(\mathcal{T})$ and
		$\mathcal{T}$ does not have an infinite branch, then $T\times \mathbb{B}$
		is not Kakutani equivalent to $T^{-1}\times \mathbb{B}$,  where $\mathbb{B}$ is the $(1/2,1/2)$
		Bernoulli shift on $\{0,1\}^{\mathbb{Z}}.$
	\end{prop}
	
	Proposition \ref{prop:Part2} is the counterpart of Part~1~(b) of Theorem~\ref{mainthm:Kproof} for circular systems.
	
	\subsection{Review of circular systems} \label{subsec:ReviewCircular}
	
	In this subsection we review the construction of the continuous reduction $F^s: \trees \to \mathcal{E}\cap\text{Diff}^{\infty}(\mathbb{T}^{2},\lambda_{2})$ from \cite[section~10]{GK3} such that for $\mathcal{T}\in\mathcal{T}\kern-.5mm rees$ and $T=F^s(\mathcal{T})$:
	\begin{enumerate}
		\item If $\mathcal{T}$ has an infinite branch, then $T$ and $T^{-1}$
		are isomorphic.
		\item If $T$ and $T^{-1}$ are Kakutani equivalent, then $\mathcal{T}$
		has an infinite branch.
	\end{enumerate}
	A key ingredient in this passage to the smooth category is a functor $\mathcal{F}$ from \cite{FW2} that takes our construction sequence $\{\mathcal{W}_n\}_{n\in \N}$ from Section~\ref{sec:review} to a construction sequence $\{\mathcal{W}_{n}^{c}\}_{n\in\mathbb{N}}$ of a so-called \emph{circular symbolic system} in the alphabet $\Sigma\cup\{b,e\}$, where $b,e$ are two additional symbols (called \emph{spacers}) not contained in our basic alphabet $\Sigma$ of cardinality $2^{12}$. These
	symbolic systems were introduced in \cite[section~4]{FW1}. By \cite[Theorem 60]{FW1} circular systems can be realized by the Anosov-Katok
	method as area-preserving $C^{\infty}$ diffeomorphisms on $\mathbb{T}^2$ provided that another sequence $(l_{n})_{n\in\mathbb{N}}$ of parameters grows sufficiently fast. 
	
	To review the definition of the functor we define a circular
	construction sequence $\{\mathcal{W}_{n}^{c}\}_{n\in\mathbb{N}}$
	and bijections $c_{n}:\mathcal{W}_{n}\to\mathcal{W}_{n}^{c}$ by induction:
	\begin{itemize}
		\item Let $\mathcal{W}_{0}^{c}=\Sigma$ and $c_{0}$ be the identity map. 
		\item Suppose that $\mathcal{W}_{n}$, $\mathcal{W}_{n}^{c}$ and $c_{n}$
		have already been defined. We construct $\mathcal{W}_{n+1}$ following the steps described in Section~\ref{sec:review}. In particular, we recall that every word in $\mathcal{W}_{n+1}$ is built by concatenating $k_n$ words in $\mathcal{W}_{n}$ by specification \ref{item:E2}. Then we define numbers 
		\[
		q_{n+1}=k_{n}l_{n}q_{n}^{2}
		\]
		and 
		\[
		p_{n+1}=p_{n}k_{n}l_{n}q_{n}+1,
		\]
		where we set $p_{0}=0$ and $q_{0}=1$. Furthermore, for $n>0$ we let $j_{q_n}=q_n$ and  $j_{i}\in\{0,\dots,q_{n}-1\}$, $i=0,\dots,q_n-1$, be such that
		\[
		j_{i}\equiv\left(p_{n}\right)^{-1}i\;\mod q_{n}.
		\]
		If $n=0$ we take $j_{0}=0$. With these numbers we introduce the \emph{circular operator}
		\[
		\mathcal{C}_{n}\left(c_{n}\left(w_{0}\right),c_{n}\left(w_{1}\right),\dots,c_{n}\left(w_{k_{n}-1}\right)\right)=\prod_{i=0}^{q_{n}-1}\prod_{j=0}^{k_{n}-1}\left(b^{q_{n}-j_{i}}\left(c_n(w_{j})\right)^{l_{n}-1}e^{j_{i}}\right).
		\]
		Then we define
		\[
		\mathcal{W}_{n+1}^{c}=\left\{ \mathcal{C}_{n}\left(c_{n}\left(w_{0}\right),c_{n}\left(w_{1}\right),\dots,c_{n}\left(w_{k_{n}-1}\right)\right)\::\:w_{0}w_{1}\dots w_{k_{n}-1}\in\mathcal{W}_{n+1}\right\} 
		\]
		and the map $c_{n+1}$ by setting 
		\[
		c_{n+1}\left(w_{0}w_{1}\dots w_{k_{n}-1}\right)=\mathcal{C}_{n}\left(c_{n}\left(w_{0}\right),c_{n}\left(w_{1}\right),\dots,c_{n}\left(w_{k_{n}-1}\right)\right).
		\]
		We note that each word in $\mathcal{W}_{n+1}^{c}$ has length $k_{n}l_{n}q_{n}^{2}=q_{n+1}$.
	\end{itemize}
	
	We denote the resulting symbolic system by $\mathbb{K}^c$.
	\begin{rem*}
		Using the operator 
		\begin{equation*}
			\mathcal{C}_{n}^{r}\left(w_{0},w_{1},\dots,w_{k_{n}-1}\right)=\prod_{i=0}^{q_{n}-1}\prod_{j=0}^{k_{n}-1}\left(e^{q_{n}-j_{i+1}}w_{k_{n}-j-1}^{l_{n}-1}b^{j_{i+1}}\right)
		\end{equation*}
		\begin{comment}
		\begin{equation*}
		\mathcal{C}_{n}^{r}\left(w_{0}w_{1}\dots w_{k_{n}-1}\right)=\prod_{i=0}^{q_{n}-1}\prod_{j=0}^{k_{n}-1}\left(e^{q_{n}-j_{i+1}}\left(\rev \left( c_n(w_{k_{n}-j-1})\right)\right)^{l_{n}-1}b^{j_{i+1}}\right)%\label{eq:RevCircOper}
		\end{equation*}
		\end{comment}
		we can give a construction sequence $\left\{ \rev(\mathcal{W}_{n}^{c})\right\} _{n\in\mathbb{N}}$
		for $\rev(\mathbb{K}^c)\cong \left(\mathbb{K}^{c}\right)^{-1}$
		via 
		\begin{align*}
			& \rev(\mathcal{W}_{n+1}^{c})=\\
			& \left\{ \mathcal{C}_{n}^{r}\left(\rev(c_{n}(w_{0})),\rev(c_{n}(w_{1})),\dots,\rev(c_{n}(w_{k_{n}-1}))\right):w_{0}w_{1}\dots w_{k_{n}-1}\in\mathcal{W}_{n+1}\right\} .
		\end{align*}
	\end{rem*}
	Moreover, we will
	use the following map from substrings of the underlying 
	system $\mathbb{K}$ to the circular system $\mathbb{K}^c$: 
	\begin{equation*}
		\mathcal{C}_{n,i}\left(w_{s}w_{s+1}\dots w_{t}\right)=\prod_{j=s}^{t}\left(b^{q_{n}-j_{i}}\left(c_{n}\left(w_{j}\right)\right)^{l_{n}-1}e^{j_{i}}\right)
	\end{equation*}
	for any $0\leq i\leq q_{n}-1$ and $0\leq s\leq t\leq k_{n}-1$. We also have the map  
	\begin{equation*}
		\mathcal{C}^r_{n,i}\left(w_{s}w_{s+1}\dots w_{t}\right)=\prod_{j=0}^{t-s}\left(e^{q_{n}-j_{i+1}}\left(\rev\left(c_{n}\left(w_{t-j}\right)\right)\right)^{l_{n}-1}b^{j_{i+1}}\right)
	\end{equation*}
	from substrings of $\mathbb{K}$ into $\rev(\mathbb{K}^c)$.
	
	In \cite[section~10.4]{GK3} we obtained the following $\overline{f}$ estimates for the circular systems corresponding to our constructions described in Section~\ref{sec:review}.
	
	\begin{lem}[\cite{GK3}, Proposition 90] \label{lem:fCircular}
		For every $s \in \Z^+$ there is $0<\alpha^{c}_s<\alpha_s$ such that for every $n\geq M(s)$ we have
		\begin{equation}
			\overline{f}\left(\mathcal{A},\overline{\mathcal{A}}\right)>\alpha_{s}^{c}
		\end{equation}
		on any substrings $\mathcal{A}$, $\overline{\mathcal{A}}$ of at
		least $q_{n}/R_{n}^{c}$ consecutive symbols in $c_n(w)$ and $c_n(\overline{w})$,
		respectively, for $w,\overline{w}\in\mathcal{W}_{n}$ with $[w]_{s}\neq[\overline{w}]_{s}$.
	\end{lem}  
	As in Lemma~\ref{lem:fDist}, the sequence $(\alpha^c_s)_{s\in \N}$ is decreasing. Furthermore, $R_{1}^{c}=R_{1}$ (with $R_{1}$ from the sequence $(R_{n})_{n\in\mathbb{N}}$ in Lemma~\ref{lem:fDist}), and $R_{n}^{c}=\lfloor\sqrt{l_{n-2}}\cdot k_{n-2}\cdot q_{n-2}^{2}\rfloor$ for $n\geq2$.
	
	\subsection{Proof of Proposition \ref{prop:Part2}} \label{subsec:ProofPropPart2}
	
	Let $\mathcal{T} \in \trees$. We look at the symbolic systems $\mathbb{S}^c_1\coloneqq \mathbb{K}^c\times \mathbb{B}$ and $\mathbb{S}^c_2\coloneqq \rev(\mathbb{K}^c)\times \mathbb{B}$, where $\mathbb{K}^c\cong F^s(\mathcal{T})$ is the circular system described in section~\ref{subsec:ReviewCircular}. By making small adjustments to the arguments from Section~\ref{subsec:Non-Equiv} we prove that $\mathbb{S}^c_1$ and $\mathbb{S}^c_2$ are not Kakutani equivalent if $\mathcal{T}$ does not have an infinite branch. We start with the following analogues of Lemmas~\ref{lem:DifferentPatterns} and~\ref{lem:groupelement}.
	
	\begin{lem}\label{lem:DifferentPatternsCircular}
		Let $P$ and $\overline{P}$ be $s$-Feldman patterns in $\mathbb{K}$ of $\mathcal{Q}_s^{n-1}$ equivalence classes, where $n-1\ge M(s)$. Furthermore, let $0\leq i_1,i_2 <q_{n-1}$ and $\mathcal{G},\overline{\mathcal{G}}$ be substrings of $\mathcal{C}_{n-1,i_1}(P)$ and $\mathcal{C}^r_{n-1,i_2}(\overline{P})$, respectively. Assume $P$ and $\overline{P}$ are different types of $s$-Feldman patterns, and $\mathcal{G}$ and $\overline{\mathcal{G}}$ both have length at least 
		$\frac{|P|\cdot l_{n-1}q_{n-1}}{2^{2e(n-1)}h_{n-1}}$. Suppose $\phi$ is a code of length $K$ from $\mathbb{S}^c_1$ to $\mathbb{S}^c_2$, where $q_{n-1}$ is large compared to $K$. Then there is a one-to-one map from shadings $\upsilon$ in $\{0,1\}^{|\mathcal{G}|}$ such that 
		\[
		\overline{f}\left(\phi(\mathcal{G},\upsilon),(\overline{\mathcal{G}},\overline{\upsilon})\right)<\frac{(\alpha^c_s)^2}{2000}
		\]
		for \emph{some} $\overline{\upsilon}\in \{0,1\}^{|\overline{\mathcal{G}}|}$ to shadings $\upsilon^{\prime} \in \{0,1\}^{|\mathcal{G}|}$ such that 
		\[
		\overline{f}\left(\phi(\mathcal{G},\upsilon^{\prime}),(\overline{\mathcal{D}},\overline{u})\right)>\frac{\alpha^c_s}{20} 
		\]
		for \emph{all} strings $(\overline{\mathcal{D}},\overline{u})$ in $\mathbb{S}^c_2$.
	\end{lem}
	
	\begin{proof}
		The proof follows along the lines of Lemma~\ref{lem:DifferentPatterns} since the analogue of Lemma~\ref{lem:ReverseOrder} also holds true in circular systems.
	\end{proof}

	\begin{lem}\label{lem:groupelementCircular}
		Let $s\in\mathbb{N}$, $0<\theta<\frac{1}{6}$,  $0<\delta<\frac{(\alpha^c_s)^3}{4\cdot 10^4}$,
		$0<\varepsilon<\frac{\alpha^c_{s}}{200}\delta$, and $\phi$ be a finite
		code from $\mathbb{S}^c_1$ to $\mathbb{S}^c_2$. 
		For $n$ sufficiently large (depending only on the  code length, $s$, $\delta$, and $\varepsilon$),
		for every $w\in\mathcal{W}_{n}$ and $0\leq i_1<q_{n-1}$ such that for a proportion of at least $1-\theta$ of the shading sequences $\upsilon \in \{0,1\}^{k_{n-1}l_{n-1}q_{n-1}}$ of $\mathcal{C}_{n-1,i_1}(w)$ there is a string $(\overline{\mathcal{A}},\overline{\upsilon})$ in $\mathbb{S}^c_2$ with $\overline{f}\left(\phi(\mathcal{C}_{n-1,i_1}(w),\upsilon),(\overline{\mathcal{A}},\overline{\upsilon})\right)<\varepsilon$, the following condition holds for at least a proportion $1-2\theta$ of the shading sequences $\upsilon \in \{0,1\}^{k_{n-1}l_{n-1}q_{n-1}}$ of $\mathcal{C}_{n-1,i_1}(w)$: 
		For every string $(\overline{\mathcal{A}},\overline{\upsilon})$
		in $\mathbb{S}^c_2$ with $\overline{f}\left(\phi(\mathcal{C}_{n-1,i_1}(w),\upsilon),(\overline{\mathcal{A}},\overline{\upsilon})\right)<\varepsilon$
		there must be exactly one string of the form $\mathcal{C}^r_{n-1,i_2}(\tilde{w})$
		with $|\overline{\mathcal{A}}\cap \mathcal{C}^r_{n-1,i_2}(\tilde{w})|\geq(1-\delta)k_{n-1}l_{n-1}q_{n-1}$ for some $0\leq i_2<q_{n-1}$, $\tilde{w}\in \mathcal{W}_{n}$. Moreover, $\tilde{w}$ must be of the form $\left[\tilde{w}\right]_{s}=g\left[w\right]_{s}$
		for a unique $g=g(w,\upsilon)$ $\in G_{s}$, which is necessarily of odd parity.
	\end{lem}
	
	\begin{proof}
		Following the proof of Lemma~\ref{lem:groupelement} we assume that there are $\tilde{w}_1,\tilde{w}_2 \in \mathcal{W}_n$ and $0\leq i_{2,1},i_{2,2}<q_{n-1}$ with $\abs{\overline{\mathcal{A}}\cap \mathcal{C}^r_{n-1,i_{2,j}}(\tilde{w}_j)}\geq (\delta-3\varepsilon)k_{n-1}l_{n-1}q_{n-1}$ for $j=1,2$. We denote by $\overline{\mathcal{A}}_j$ the part of $\overline{\mathcal{A}}$ corresponding to $\mathcal{C}^r_{n-1,i_{2,j}}(\tilde{w}_j)$. Furthermore, $\phi(\mathcal{A}_{j},\upsilon_j)$ denotes the parts in $\phi(\mathcal{C}_{n-1,i_1}(w),\upsilon)$ corresponding to $(\overline{\mathcal{A}}_j,\overline{\upsilon}_j)$ under a best possible $\overline{f}$ match. In order to prove the circular counterpart of Claim~\ref{claim:claim1} we divide $\mathcal{A}_j$ and $\overline{\mathcal{A}}_j$ into circular images $\mathcal{C}_{n-1,i_1}(P_{n-1,\ell})$ and $\mathcal{C}^r_{n-1,i_{2,j}}(\overline{P}_{n-1,\ell})$, where $P_{n-1,\ell}$ and $\overline{P}_{n-1,\ell}$ are $s$-Feldman patterns in $w$ and $\tilde{w}_j$, respectively. Then the proof follows along the lines of Lemma~\ref{lem:groupelement} using Lemma~\ref{lem:DifferentPatternsCircular}. 
	\end{proof}
	
	Assume $\mathbb{S}^c_1$ and $\mathbb{S}^c_2$ are evenly equivalent. This time we apply \cite[Corollary~71]{GK3} to $\mathbb{S}^c_1$ with $\gamma=\gamma^c_s\coloneqq \frac{(\alpha^c_s)^8}{4\cdot 10^{24}}$ for $s\in \mathbb{Z}$. We obtain a sequence of finite codes $(\phi^c_{\ell})_{\ell \in \N}$ from $\mathbb{S}^c_1$ to $\mathbb{S}^c_2$ and a collection $\mathcal{W}^{\prime}_n \subset \mathcal{W}_n$ (that includes at least $1-\sqrt{\gamma^c_s}$ of the $n$-words) such that for every $w\in \mathcal{W}^{\prime}_n$ and $0\leq i<q_{n-1}$ the circular strings $\mathcal{C}_{n-1,i}(w)$ satisfy analogues of properties \ref{item:C1} and \ref{item:C2}.
	\begin{lem}
		\label{lem:codeGroupCircular}Suppose that $\mathbb{S}^c_1$ and $\mathbb{S}^c_2$ are evenly equivalent and let $s\in\mathbb{N}$. There is a unique $g_{s}\in G_{s}$
		such that for every $N\geq N(s)$ and sufficiently large $n\in\mathbb{N}$
		we have for every $w\in\mathcal{W}_{n}^{\prime}$ that for a proportion of at least $\frac{99}{100}$ of possible shading sequences $\upsilon \in \{0,1\}^{q_n}$ there is $\tilde{w}\in \mathcal{W}_{n}$
		with $\left[\tilde{w}\right]_{s}=g_{s}\left[w\right]_{s}$
		and some shading sequence $\tilde{\upsilon} \in \{0,1\}^{q_n}$ such that
		\begin{equation}\label{eq:codeGroupCircular}
			\overline{f}\left(\phi^c_{N}(c_n(w),\upsilon),(\rev(c_n(\tilde{w})),\tilde{\upsilon})\right)<\frac{\alpha^c_{s}}{4},
		\end{equation}
		where $\left(\phi^c_{\ell}\right)_{\ell \in\mathbb{N}}$ is a sequence of finite
		codes as described above.
		
		Moreover, $g_{s}\in G_{s}$ is of odd parity and the sequence $(g_{s})_{s\in\mathbb{N}}$
		satisfies $g_{s}=\rho_{s+1,s}(g_{s+1})$ for all $s\in\mathbb{N}$.
	\end{lem}
	
	\begin{proof}
		Let $w\in\mathcal{W}^{\prime}_{n} \subset \mathcal{W}_{n}$ and we recall that $c_{n}(w)=\prod_{i=0}^{q_{n-1}-1}\mathcal{C}_{n-1,i}(w)$. As in the proof of Lemma \ref{lem:codeGroup}
		we use property~\ref{item:C1} to choose
		$n$ sufficiently large such that for a proportion of at least $1-\sqrt{\gamma^c_s}$ of possible shading sequences $\upsilon_i \in \{0,1\}^{k_{n-1}l_{n-1}q_{n-1}}$ there exists $z \in \mathbb{S}^c_2$ with
		\[
		\overline{f}\left(\phi_{N}^{c}\left(\mathcal{C}_{n-1,i}(w),\upsilon_i\right),z\upharpoonright[0,k_{n-1}l_{n-1}q_{n-1}-1]\right)<\gamma^c_s.
		\]
		We denote $(\overline{\mathcal{A}}_{i},\hat{\upsilon}_i)\coloneqq z\upharpoonright[0,k_{n-1}l_{n-1}q_{n-1}-1]$
		in $\mathbb{S}^{c}_2$. By Lemma~\ref{lem:groupelementCircular} with $\delta^c_s=\frac{(\alpha^c_s)^6}{10^{22}}$ we have for a proportion of at least $1-2\sqrt{\gamma^c_s}$ of possible shading sequences that
		there is a string of the form $\mathcal{C}_{n-1,j_{i}}^{r}(\tilde{w})$
		with $|\overline{\mathcal{A}}_{i}\cap\mathcal{C}_{n-1,j_{i}}^{r}(\tilde{w})|\geq(1-\delta^c_{s})k_{n-1}l_{n-1}q_{n-1}$
		for some $0\leq j_{i}<q_{n-1}$ and $\tilde{w}\in \mathcal{W}_{n}$, where $\tilde{w}$ must be of the form $\left[\tilde{w}\right]_{s}=g\left[w\right]_{s}$
		for a unique $g=g(w,\upsilon)\in G_{s}$ of odd parity. By Fact~\ref{fact:omit_symbols} we conclude
		$$\overline{f}\left(\phi_{N}^{c}\left(\mathcal{C}_{n-1,i}(w),\upsilon_i\right),\left(\mathcal{C}_{n-1,j_{i}}^{r}(\tilde{w}),\overline{\upsilon}_i\right)\right)<\gamma^c_s+\delta^c_{s}$$ for some shading sequence $\overline{\upsilon}_i\in \{0,1\}^{k_{n-1}l_{n-1}q_{n-1}}$.
		Since this holds true for every $i\in\left\{ 0,\dots,q_{n-1}-1\right\} $
		and the newly introduced spacers occupy a proportion of at most $\frac{1}{l_{n-1}}$,
		we obtain for $n$ sufficiently large that $$\overline{f}\left(\phi_{N}^{c}\left(c_n(w),\upsilon\right),\left(\prod^{q_{n-1}-1}_{i=0}\mathcal{C}_{n-1,i}^{r}(\tilde{w}),\tilde{\upsilon}\right)\right)<\gamma^c_s+\delta^c_{s}+\frac{2}{l_{n-1}}<\frac{(\alpha_{s}^{c})^{6}}{10^{20}}$$
		for some shading sequence $\tilde{\upsilon} \in \{0,1\}^{q_n}$. The word $\tilde{w}=\tilde{w}_0\tilde{w}_1\dots\tilde{w}_{k_{n-1}-1}\in \mathcal{W}_n$ satisfies 
		\begin{align*}
			\prod^{q_{n-1}-1}_{i=0}\mathcal{C}_{n-1,i}^{r}(\tilde{w}) & =\mathcal{C}_{n-1}^{r}\left(\rev(c_{n-1}(\tilde{w}_{0})),\rev(c_{n}(\tilde{w}_{1})),\dots,\rev(c_{n}(\tilde{w}_{k_{n-1}-1}))\right) \\
			& =\rev(c_n(\tilde{w}))\in \rev(\mathcal{W}_{n}^{c})
		\end{align*}
		and $\left[\tilde{w}\right]_{s}=g_{s}\left[w\right]_{s}$.
		
		The rest of the proof follows along the lines of Lemma~\ref{lem:codeGroup}.
	\end{proof} 
	
	\begin{proof}[Proof of Proposition~\ref{prop:Part2}]
		Since the systems $T\times \mathbb{B}\cong \mathbb{S}^c_1$ and $T^{-1}\times \mathbb{B}\cong \mathbb{S}^c_2$ have the same positive entropy, a Kakutani equivalence has to be an even equivalence. Then the conclusion of Lemma~\ref{lem:codeGroupCircular} implies that the
		tree $\mathcal{T}$ has an infinite branch.
	\end{proof}
	
	\subsection{A continuum of pairwise non-Kakutani equivalent $K$-diffeomorphisms} \label{subsec:continuum}
	We now indicate an explicit construction of  a continuum of pairwise non-Kakutani equivalent $K$-diffeomorphisms in Diff$^{\infty}(\mathbb{T}^5,\lambda_5)$ of the same metric entropy. This is not needed for the proof of Theorem~\ref{mainthm:Kdiffeo}, but it illustrates some of the techniques in Sections~\ref{subsec:ConstrSmooth}--\ref{subsec:ProofPropPart2} in a simpler context. As we explain at the end of this subsection, Theorem~\ref{mainthm:Kdiffeo} also holds if we consider only $K$-diffeomorphisms of constant metric entropy.  
	Then the existence of an uncountable family of pairwise non-Kakutani equivalent $K$-diffeomorphisms of the same metric entropy already follows, because, in general, any analytic equivalence relation in a Polish space with only countably many equivalence classes must be Borel: The equivalence classes partition the space into analytic sets, and the complement of each equivalence class is therefore a countable union of analytic sets, which is itself analytic. Hence each equivalence class is analytic and co-analytic, which by a theorem of Suslin \cite[Section 14.C]{Kechris} implies that it is Borel. Thus the equivalence relation is the countable union of Cartesian products of these equivalence classes with themselves, which is Borel. 
	
	Two sequences $(x_n)$ and $(y_n)$ in $\{0,1\}^{\mathbb{N}}$ are said to be $E_0$ equivalent if there exists $N\in \mathbb{N}$ such that $x_n=y_n$ for all $n\ge N$.
	Ornstein, Rudolph, and Weiss \cite[Chapter 12]{ORW} constructed a family of zero-entropy measure-preserving automorphisms indexed by sequences in $\{0,1\}^{\mathbb{Z}}$
	such that two automorphisms in this family are Kakutani equivalent if and only if the corresponding sequences are $E_0$ equivalent. (In the language of descriptive set theory this means that $E_0$ reduces to Kakutani equivalence in this family.) The cardinality of the  collection of $E_0$ equivalence classes in $\{0,1\}^{\mathbb{N}}$ is that of the continuum, because the map from $\{0,1\}^{\mathbb{N}}$ to $E_0$ equivalence classes that takes $(x_n)$ to the equivalence class containing $(x_0,x_0,x_1,x_0,x_1,x_2,\dots)$ is one-to-one. There are two known methods for obtaining
	families in Diff$^{\infty}(\mathbb{T}^2,\lambda_2)$ indexed by sequences in $\{0,1\}^{\mathbb{N}}$
	such that two diffeomorphisms in the family are Kakutani equivalent if and only if the corresponding sequences are $E_0$ equivalent. The first is due to M. Benhenda \cite{Be15}.
	
	A second way of obtaining such a family of diffeomorphisms in Diff$^{\infty}(\mathbb{T}^2,\lambda_2)$ is to start with the family of automorphisms from \cite{ORW} as above, and view them as odometer-based systems as described in \cite{FW2}. Then apply the functor in \cite{FW2} to map these odometer-based systems to circular-based systems, which can be realized as $C^{\infty}$ diffeomorphisms of $\mathbb{T}^2$ that preserve $\lambda_2$. At the time that \cite{FW2} was written, it was an open question (due to Thouvenot) whether Kakutani equivalence (or non-Kakutani equivalence) is preserved under this functor. So it was not clear that there would still be a continuum of pairwise non-Kakutani equivalent transformations after applying the functor in \cite{FW2}. In fact, in general, neither Kakutani equivalence nor non-Kakutani equivalence is preserved under this functor \cite{GK2}, but they are in this particular situation, because the circular operator $\mathcal{C}_n$ (see Subsection \ref{subsec:ReviewCircular}) maps each Feldman pattern in the odometer-based system to repetitions of a Feldman pattern in the circular system, once the $b$'s and $e$'s added at stage $n$ are removed.  The argument in \cite[Chapter 12]{ORW} that shows that two automorphisms corresponding to $E_0$ inequivalent sequences are not Kakutani equivalent generalizes to the systems obtained after applying the functor. 
	
	For each $T$ in the family of zero-entropy diffeomorphisms of $\mathbb{T}^2$ that is obtained either by \cite{Be15} or by \cite{FW2}, we construct the $K$-diffeomorphism $T_{g,\varphi}$ as in Subsection~\ref{subsec:ConstrSmooth}. Then by Proposition~\ref{prop:Part1}, $T_{g,\varphi}$ is Kakutani equivalent to $T\times \mathbb{B}$, where $\mathbb{B}$ is the $(1/2,1/2)$ Bernoulli shift on $\{0,1\}^{\mathbb{Z}}.$ Finally, the argument in \cite[Chapter 13]{ORW} shows that for $T$ and $S$ in the above family corresponding to $E_0$ inequivalent sequences, $T\times \mathbb{B}$ is not Kakutani equivalent to $S\times\mathbb{B}.$ (This is similar to the proof of our Proposition~\ref{prop:Part2},  but the argument is simpler because no equivalence relations on words are involved.) Therefore $T_{g,\varphi}$ is not Kakutani equivalent to $S_{g,\varphi}.$
	
	By a formula for the entropy of a skew product due to D. Newton \cite{Ne69}, each $T_{g,\varphi}$ has the same metric entropy as $g$. Since the same $g$ is used in the construction of all of the $K$-diffeomorphisms in our family, they all have the same metric entropy. (Note: Although it did not
	matter for the proof of Theorem \ref{mainthm:Kdiffeo}, Newton's formula also implies that all of the $K$-diffeomorphisms that are in the images of the maps $\Phi_j$, $j=1,2$, constructed in the proof of Theorem \ref{thm:Kdiffeo} have the same
	metric entropy.) 
	
	\subsection*{Acknowledgements}
	We thank M.~Foreman and J.-P.~Thouvenot for helpful conversations that initiated this work, and we are grateful to A. Kanigowski for his valuable feedback on an earlier draft. The first author thanks Jagiellonian University for its hospitality during part of the time that this paper was written, and we thank Indiana University for supporting a research visit by the second author.

\end{document}